\documentclass[11pt,a4paper,openany]{book}
\usepackage[left=2.5cm,right=2.5cm,top=2.5cm,bottom=2.5cm]{geometry}
\usepackage[english]{babel}
\usepackage[T1]{fontenc}
\usepackage[utf8]{inputenc}
\usepackage{lmodern}
\usepackage{csquotes}
\usepackage{microtype}

\usepackage{amsmath,amssymb,amsfonts,amsthm,mathtools,esint}
\usepackage{mathrsfs}
\usepackage{dsfont}
\numberwithin{equation}{chapter}

\newcommand{\C}{\mathcal C}

\renewcommand{\H}{\mathcal H}

\newcommand{\T}{\mathbb T}

\newcommand{\CC}{\mathbb C}

\newcommand{\RR}{\mathbb R}
\newcommand{\TT}{\mathbb T}
\newcommand{\ZZ}{\mathbb Z}

\usepackage{xcolor}
\usepackage{graphicx}
\usepackage{subcaption}
\usepackage{float}
\usepackage{wrapfig}
\usepackage[rightcaption]{sidecap}
\usepackage[labelfont=bf]{caption}
\usepackage{tikz}
\usetikzlibrary{positioning,shapes.geometric,arrows,shadows,shadows.blur,shapes.misc}
\usepackage[shortlabels,inline]{enumitem}
\setlist[itemize]{itemsep=5pt}
\setlist[enumerate]{itemsep=5pt}
\graphicspath{{Images/}{../Images/}}
\setkeys{Gin}{width=0.85\textwidth}

\definecolor{linkscolor}{RGB}{4,36,58}
\usepackage{hyperref}
\hypersetup{colorlinks=true,linkcolor=linkscolor,citecolor=linkscolor,urlcolor=linkscolor,linktoc=all}
\usepackage[nameinlink]{cleveref}
\usepackage{fancyhdr}
\usepackage[titles]{tocloft}
\usepackage[explicit]{titlesec}
\titleformat{\chapter}[display]
{\centering\normalfont\Huge}
{\titlerule[4pt]\vspace{3pt}\titlerule[1.5pt]\vspace{3pt}\textbf{\MakeUppercase{\chaptername} \thechapter}}
{0pt}
{\titlerule[1.5pt]\vspace{1pc}\Huge\MakeUppercase{#1}}

\titleformat{name=\chapter,numberless}[display]
{\centering\normalfont\Huge}
{\titlerule[1.5pt]\vspace{-20pt}}
{0pt}
{\Huge\MakeUppercase{#1}}

\usepackage[most]{tcolorbox}
\tcbuselibrary{theorems,skins,breakable}

\definecolor{defscol}{HTML}{ecd8d7}
\definecolor{asumscol}{HTML}{ecd8d7}
\definecolor{rmkscol}{HTML}{313160}
\definecolor{exmscol}{HTML}{e04b52}
\definecolor{lemscol}{HTML}{2c3943}
\definecolor{thmscol}{HTML}{595765}
\definecolor{prpscol}{HTML}{9c98b1}
\definecolor{corscol}{HTML}{dfd9fd}
\definecolor{clmscol}{HTML}{165c58}
\definecolor{facscol}{HTML}{28a8a1}

\newtcbtheorem[number within=section]{mydefinition}{Definition}{enhanced,breakable,frame hidden,fonttitle=\bfseries\large,coltitle=black,colbacktitle=defscol!40!white,colback=defscol!20!white}{defn}
\newtcbtheorem[use counter from=mydefinition]{mytheorem}{Theorem}{enhanced,breakable,frame hidden,fonttitle=\bfseries\large,coltitle=black,colbacktitle=thmscol!40!white,colback=thmscol!20!white}{thm}
\newtcbtheorem[use counter from=mydefinition]{mylemma}{Lemma}{enhanced,breakable,frame hidden,fonttitle=\bfseries\large,coltitle=black,colbacktitle=lemscol!40!white,colback=lemscol!20!white}{lem}
\newtcbtheorem[use counter from=mydefinition]{myproposition}{Proposition}{enhanced,breakable,frame hidden,fonttitle=\bfseries\large,coltitle=black,colbacktitle=prpscol!30!white,colback=prpscol!20!white}{prop}
\newtcbtheorem[use counter from=mydefinition]{mycorollary}{Corollary}{enhanced,breakable,frame hidden,fonttitle=\bfseries\large,coltitle=black,colbacktitle=corscol!40!white,colback=corscol!20!white}{cor}
\newtcbtheorem[use counter from=mydefinition]{myassumption}{Assumptions}{enhanced,breakable,frame hidden,fonttitle=\bfseries\large,coltitle=black,colbacktitle=asumscol!40!white,colback=asumscol!20!white}{asum}
\newtcbtheorem[use counter from=mydefinition]{myclaim}{Claim}{enhanced,breakable,frame hidden,fonttitle=\bfseries\large,coltitle=black,colbacktitle=clmscol!40!white,colback=clmscol!20!white}{clm}
\newtcbtheorem[use counter from=mydefinition]{myfact}{Fact}{enhanced,breakable,frame hidden,fonttitle=\bfseries\large,coltitle=black,colbacktitle=facscol!40!white,colback=facscol!20!white}{fact}

\newenvironment{myexample}{\tcolorbox[blanker,breakable,left=5mm,parbox=false,before upper={\parindent15pt},after skip=10pt,borderline west={1mm}{0pt}{clmscol!40!white}]}{\textcolor{clmscol!40!white}{\hbox{}\nobreak\hfill$\blacksquare$}\endtcolorbox}

\usepackage{todonotes}

\usepackage{mdframed}

\newenvironment{chapterabstract}

\newcommand{\BigTitle}{{\bf Rotating Vortex Patches}}
\newcommand{\LittleTitle}{{\bf Rigidity, Bifurcation, and Unified Structures}}
\title{\Huge \BigTitle\\[0.4cm]\Large \LittleTitle}
\author{
\Large Taoufik Hmidi\\
\\
\Large New York University Abu Dhabi\\
\texttt{th2644@nyu.edu}
}
\date{\today}

\begin{document}

\frontmatter

\maketitle

\chapter*{Abstract}

\addcontentsline{toc}{chapter}{Abstract}

This monograph presents a systematic account of several classical and recent developments in the theory of rotating vortex patches, or V-states, for the two-dimensional Euler equations and related active scalar models. Its purpose is both expository and structural: we revisit some of the foundational results of the subject, provide detailed proofs and alternative formulations, and present more recent results within a common analytical framework.

We begin with the basic theory of two-dimensional Euler vortex patches and the contour dynamics formulation of rigidly rotating solutions. We then develop two complementary approaches to the V-state equation, based on Cauchy integrals and conformal mappings, and discuss their connections with potential theory and Faber polynomials. These tools are used to revisit classical examples, including Rankine vortices and Kirchhoff ellipses, as well as several rigidity and classification results for rotating patches. A substantial part of the monograph is devoted to Burbea's bifurcation theory, for which we give a detailed treatment in Hölder spaces, including the functional-analytic and spectral ingredients underlying the construction of noncircular V-states.

The final part develops a unified approach to rotating patches for a broad class of incompressible active scalar equations. The main point is that the bifurcation mechanism can be formulated in terms of structural properties of the interaction kernel, rather than through model-specific computations. In particular, complete monotonicity leads to a spectral factorization that isolates a universal component of the linearized problem and provides a common framework encompassing the Euler, generalized surface quasi-geostrophic, quasi-geostrophic shallow-water, and related models. The monograph concludes with a discussion of open problems and perspectives concerning rigidity, global bifurcation, singular geometries, and recurrent vortex dynamics.

\clearpage

\tableofcontents
\mainmatter
\chapter{General Introduction}
\section{Rigid vortex motion: global overview }

The mathematical theory of vortex motion lies at the crossroads of partial differential equations,
geometric analysis, complex analysis and fluid mechanics.
Since the pioneering works of Helmholtz and Kelvin in the nineteenth century,
vortices have been recognized as the fundamental building blocks of incompressible flows.
Their remarkable persistence, their ability to interact over long distances,
and the intricate patterns generated by their dynamics
have made them one of the central objects of mathematical fluid mechanics.

Among all coherent vortex structures, those undergoing a rigid-body motion occupy a distinguished position.
They provide exact solutions of highly nonlinear evolution equations,
yet exhibit an unexpectedly rich mathematical structure.
From a physical viewpoint,
they describe long-lived coherent structures observed in laboratory experiments,
geophysical flows and numerical simulations.
From a mathematical viewpoint,
they constitute nonlinear equilibria around which one can investigate
stability, bifurcation, spectral properties and long-time dynamics.

The history of rigid vortex motions dates back to the seminal work of
Kirchhoff \cite{Kirch}, who discovered in 1874 that an elliptical region of constant vorticity
rotates uniformly while preserving its shape.
This remarkable observation revealed for the first time
that the Euler equations admit nontrivial coherent structures beyond the obvious circular vortices introduced earlier by Rankine.
For almost one century,
Rankine vortices and Kirchhoff ellipses remained essentially the only explicit examples of rotating vortex patches.

A major turning point occurred in the late 1970s, when Deem and Zabusky \cite{DZ78} numerically discovered rotating vortex patches exhibiting polygonal symmetries. Their computations suggested that the Euler equations possess infinitely many families of noncircular rotating equilibria, revealing a much richer structure than previously anticipated. Shortly afterwards, Burbea \cite{Burbea82} established their existence rigorously by combining conformal mapping techniques with local bifurcation theory.

Burbea's pioneering work laid the foundations for the modern theory of
rotating vortex patches. Remarkably, however, his paper remained largely
unnoticed within the PDE community for several decades. Around 2012,
together with Joan Mateu and Joan Verdera, we rediscovered Burbea's work
and revisited his approach using modern PDE techniques. This contributed
to a renewed interest in the analytical theory of rotating vortex patches.

Since then, the subject has experienced remarkable progress, driven by increasingly refined analytical and numerical methods. What initially appeared as a local bifurcation problem near the Rankine vortex has developed into a rich research area at the crossroads of nonlinear partial differential equations, complex analysis, geometric analysis, and dynamical systems. Several directions have emerged, concerning the topology and regularity of the patches, the global structure of bifurcation branches, rigidity and symmetry phenomena, and extensions to other active scalar equations.

One important direction concerns the topology of the vorticity support.
Beyond the simply connected configurations arising from Burbea's theorem,
the existence theory has been extended to doubly and multiply connected
domains, revealing a considerably richer family of rotating coherent
structures.

A second fundamental direction concerns the regularity of the vortex boundary. Although local bifurcation arguments naturally produce solutions in spaces of limited regularity, subsequent analysis reveals a remarkable smoothing phenomenon. Starting from H\"older regularity, successive works established higher-order regularity, $C^\infty$ smoothness, and eventually, under suitable assumptions, real analyticity of the boundary. Besides their intrinsic interest, these results provide important analytical tools for the study of geometric properties, uniqueness, and the finer structure of bifurcation branches.

The global continuation of the local branches has also attracted considerable attention. While the Crandall--Rabinowitz theorem describes the solution set only in a neighborhood of the Rankine vortex, numerical computations and global bifurcation techniques reveal a much richer picture away from the bifurcation point. Branches may develop turning points, reconnect with other branches, form closed loops, or approach singular limiting configurations characterized by corners or self-contact. Understanding the global organization of these branches and their possible limiting geometries remains one of the challenging problems in the theory.

Closely related to these questions is the study of qualitative and geometric properties of rotating patches. Rigidity and symmetry results seek to determine when the geometry of a patch is forced by its angular velocity or by additional structural assumptions. Convexity, geometric inequalities, singular limiting profiles, stability, and long-time dynamics provide further directions in which the interaction between the geometry of the patch and the underlying fluid dynamics becomes particularly pronounced.

In parallel, the theory has been extended well beyond the two-dimensional Euler equations to a variety of active scalar models arising notably in geophysical fluid dynamics. Prominent examples include the Surface Quasi-Geostrophic equation (SQG), its generalized versions (gSQG), the quasi-geostrophic shallow-water equations (QGSW), and several regularized models; see, for instance, \cite{DHH16,DHHM,DHMV16,Garcia21,Garcia20,GS23,GHJ20,HMH23,Gom19,HMW20,HW22,HM16b,HM17,HMV15,HXX23,Rou23a}. Despite substantial differences in the singularity and decay of their interaction kernels, these equations exhibit remarkably similar families of rotating coherent structures. Their construction relies on closely related ingredients: contour dynamics, conformal parametrizations, spectral analysis, and bifurcation theory.

A central theme of this monograph is that the classical bifurcation mechanism discovered by Burbea for the Euler equation is not intrinsically tied to the logarithmic interaction kernel. At the linearized level, it is governed by a spectral structure that can be isolated through complete monotonicity of the radial kernel. This observation ultimately leads to a unified bifurcation theory encompassing Euler, gSQG, QGSW, and a broader class of active scalar equations. This perspective will be developed systematically in the final chapter.

More recently, the study of coherent vortex structures has moved beyond relative equilibria toward genuinely time-dependent dynamics. In particular, time quasi-periodic vortex patches have been constructed for the Euler equation and related active scalar models; see \cite{BHM23,HHM23,HR22,HHR23,HR21,Rou23b}. These solutions exhibit several interacting temporal frequencies and open a new direction in the study of recurrent vortex dynamics.

Our primary focus is on simply connected uniformly rotating vortex
patches, with particular emphasis on rigidity, local bifurcation from
the Rankine vortex, and the structural mechanisms underlying their
existence for different active scalar equations.

\section{Organization of the monograph}

The monograph is organized around a natural progression from the basic analytical framework of the two-dimensional Euler equations to rigidity, bifurcation, and ultimately a unified theory of rotating vortex patches for a broad class of active scalar equations. Each chapter develops a different facet of the theory, while introducing tools that will be used throughout the subsequent analysis.

{\bf Chapter~2} introduces the analytical and geometric foundations for the study of rigid vortex motions in the two-dimensional Euler equations. We first review the vorticity formulation and the classical global well-posedness theory, with particular emphasis on vortex patches and the persistence of their boundary regularity. We then introduce rotating vortices and discuss some of their elementary rigidity properties arising from the conservation laws of the Euler flow. The chapter concludes with the derivation of the contour dynamics formulation and the boundary equation characterizing rotating vortex patches. These results provide the basic framework for the rigidity and bifurcation theories developed later.

{\bf Chapters~3 and 4} explore two complementary reformulations of the rotating patch equation, based respectively on the Cauchy transform and conformal mapping techniques. The guiding principle is that the same geometric condition defining a $V$-state can be expressed in different analytical forms, each revealing a distinct structure of the problem.

In Chapter~3, the boundary equation is reformulated through the Cauchy transform of the patch. This converts the $V$-state problem into an inverse problem of potential-theoretic nature and provides an efficient framework for recovering the classical Rankine vortices and Kirchhoff ellipses. It also leads to an inverse characterization of ellipses, illustrating the rigidity encoded in the Cauchy transform of the domain.

Chapter~4 develops a complementary approach based on the exterior Riemann conformal mapping. The Cauchy transform is expressed in terms of the associated Faber polynomials, leading to an algebraic formulation of the $V$-state equation adapted to the conformal geometry of the domain. This representation provides a particularly transparent description of Kirchhoff ellipses and yields further rigidity results, including the characterization of rotating patches whose exterior conformal mapping has a finite Laurent expansion. Together, Chapters~3 and 4 highlight the interplay between vortex dynamics, complex analysis, potential theory, and geometric function theory.

{\bf Chapter~5} is devoted to global rigidity phenomena. Its central question
is whether the angular velocity can constrain, or even completely
determine, the geometry of a rotating vortex patch. We discuss the
rigidity theorem asserting that a simply connected rotating patch with
angular velocity
\(
\Omega\notin\left(0,\frac12\right)
\)
must be a Rankine vortex.
Particular attention is given to the critical value
\(\Omega=\frac12\), for which we present two proofs based respectively
on the maximum principle and on the conformal/Faber-polynomial
formulation. The chapter concludes with the general variational approach developed by  Gómez-Serrano, Park, Shi and Yao \cite{GPSY},
where Talenti's isoperimetric inequality plays a central role.

{\bf Chapter~6} is devoted to Burbea's celebrated construction of noncircular
rotating vortex patches bifurcating from the Rankine vortex. Burbea's
original argument introduced the conformal mapping formulation and
combined it with bifurcation tools to construct, for each
integer \(m\geq2\), a branch of \(m\)-fold symmetric \(V\)-states.
However, some functional-analytic aspects of the original proof,
formulated in Hardy spaces, require additional justification. We revisit
the argument in a Hölder framework and provide a complete proof of the
local bifurcation result.

The analysis is organized around the nonlinear \(V\)-state functional,
its regularity and linearization at the disk, and the spectral properties
of the resulting Fourier multiplier. The explicit Euler spectrum allows
one to identify the kernel and range at each bifurcation value and to
verify the transversality condition required by the
Crandall--Rabinowitz theorem. This explicit spectral structure also
provides the natural starting point for the unified theory developed in
Chapter~7.

The final chapter is devoted to a unified theory of rotating vortex patches for a broad class of incompressible active scalar equations. Following Burbea's pioneering work for the Euler equations, analogous bifurcation results have been established for several active scalar models arising in geophysical fluid dynamics, including the generalized surface quasi-geostrophic equations (gSQG) obtained in \cite{HH15,CCG16}, the quasi-geostrophic shallow-water equations (QGSW) in \cite{DHR19}, and various regularized SQG models. Although these equations are governed by interaction kernels with substantially different singularities, they all exhibit remarkably similar families of rotating coherent structures. Nevertheless, the existing proofs are highly model-dependent and often rely on delicate computations involving special functions, such as Gamma functions, Bessel functions, or hypergeometric functions, making it difficult to identify the common analytical mechanisms underlying these bifurcation phenomena.

The primary objective of this chapter is to show that these apparently different existence theories are in fact manifestations of a single abstract framework. We present the unified bifurcation theory recently developed by Hmidi, Xue and Xue in \cite{HXX26}, which applies to a broad class of incompressible active scalar equations. The starting point is the observation that the interaction kernels associated with these models all satisfy a common structural assumption, namely that the function $-K_0'$ is completely monotone. Combined with a mild integrability condition near the origin, this property is sufficient to recover the entire bifurcation theory in a unified manner, independently of the particular equation under consideration.

The central novelty of this theory is a new spectral approach to the
linearized operator around the Rankine vortex, based on the factorization formula for the spectrum
\[
\lambda_n
=
2\int_0^\infty
\phi_n(x)\,
\frac{d\mu(x)}{x},
\qquad
\phi_n(x)
=
\int_0^\pi
e^{-2x\sin\eta}\,
e^{2in\eta}\,d\eta.
\]
The functions $\phi_n$ are universal, depending only on the Fourier mode,
whereas the dependence on the active scalar model is entirely encoded in
the positive measure $\mu$ arising from the Bernstein representation of
the kernel. This factorization separates the universal spectral structure
from the model-dependent interaction kernel and reduces the analysis to
the study of the family $\{\phi_n\}_{n\geq1}$. Their positivity and
monotonicity, obtained through an ODE and comparison argument, yield the
spectral properties required by the Crandall--Rabinowitz theorem.

\section*{Acknowledgments}

The author would like to thank Haroune Houamed, Liutang Xue, and Zhilong Xue
for many valuable discussions and for the considerable time and effort they
devoted to carefully checking several parts of the manuscript.

\chapter{Two-dimensional Euler equations and rigid vortex motions}

\begin{chapterabstract}
{\it This chapter introduces the analytical and geometric foundations for the study of rigid vortex motions in the two-dimensional Euler equations. We first review the vorticity formulation and the classical global well-posedness theory, with particular emphasis on vortex patches and the persistence of their boundary regularity. We then introduce rotating vortices and discuss some of their rigidity properties arising from the conservation laws of the Euler flow. Finally, we derive the contour dynamics formulation and the boundary equation characterizing rotating vortex patches. These results provide the basic framework for the rigidity and bifurcation theories developed in the subsequent chapters.}
\end{chapterabstract}

\section{Planar two-dimensional Euler equations}

The motion of an ideal incompressible fluid in the plane is governed by the two-dimensional Euler equations. Written in terms of the velocity field $v=(v_1,v_2)$ and the pressure $p$, they take the form
\begin{equation}\label{Euler-velocity}
\left\{
\begin{aligned}
&\partial_t v+(v\cdot\nabla)v+\nabla p=0,\\
&\operatorname{div}v=0,\\
&v|_{t=0}=v_0,
\end{aligned}
\right.
\end{equation}
where the incompressibility condition expresses the conservation of volume along the flow. 

The Euler equations were introduced by Euler in the eighteenth century as the fundamental equations describing the motion of an ideal fluid. Nevertheless, the geometric structure of fluid motion became substantially clearer through the work of Helmholtz and Kelvin in the nineteenth century. Helmholtz's analysis of vortex motion revealed that the rotational component of the velocity, rather than the velocity itself, is the natural quantity governing many persistent structures observed in ideal fluids. Kelvin's circulation theorem further showed that the circulation of the velocity along a closed material curve is conserved by the evolution. These ideas laid the foundations of modern vortex dynamics.
\\
In the planar setting, the vorticity is the scalar function
\[
\omega(z,t)
=
\partial_1v_2(z,t)-\partial_2v_1(z,t),
\qquad
z\in\CC,\quad t\geq0.
\]
It measures the local rotational motion of the fluid. Taking the curl
of the velocity formulation \eqref{Euler-velocity} eliminates the
pressure and yields the vorticity equation
\begin{equation}\label{vorticity}
\begin{cases}
\partial_t\omega(z,t)+v(z,t)\cdot\nabla\omega(z,t)=0,
& z\in\CC,\quad t>0,\\
v(z,t)=\nabla^\perp\psi(z,t),
\qquad
\Delta\psi(z,t)=\omega(z,t),\\
\omega(z,0)=\omega_0(z),
\end{cases}
\end{equation}
where
\[
\nabla^\perp=(-\partial_2,\partial_1).
\]
Thus, the vorticity is transported by the incompressible velocity
field, while the velocity is recovered from the vorticity through the
stream function \(\psi\).
Since
\[
\Delta\left(\frac{1}{2\pi}\log|z|\right)=\delta_0
\]
in the sense of distributions, the stream function associated with a
sufficiently localized vorticity is given by
\[
\psi(z,t)
=
\frac{1}{2\pi}
\int_{\CC}\log|z-\zeta|\,\omega(\zeta,t)\,d\zeta.
\]
Identifying \(\RR^2\) with \(\CC\), we represent the velocity field by
the complex-valued function
\[
v=v_1+iv_2.
\]
Introducing the Wirtinger derivative
\[
\overline{\partial}
=
\tfrac12\left(\partial_{x_1}+i\partial_{x_2}\right),
\]
the relation \(v=\nabla^\perp\psi\) becomes
\[
v(z,t)=2i\,\overline{\partial}\psi(z,t).
\]
Moreover,
\[
2\overline{\partial}\log|z|
=
\frac{1}{\overline z},
\qquad z\neq0.
\]
Consequently, the velocity generated by the vorticity \(\omega(t)\) admits the
Biot--Savart representation
\begin{align}\label{Biot-Savart-real}
v(z,t)
\nonumber&=
\frac{1}{2\pi}
\int_{\RR^2}
\frac{(z-\zeta)^\perp}{|z-\zeta|^2}
\omega(\zeta,t)\,dA(\zeta)\\
&=
\frac{i}{2\pi}
\int_{\CC}
\frac{\omega(\zeta,t)}
{\overline z-\overline\zeta}\,
dA(\zeta).
\end{align}
Hence the Euler equations can be interpreted as an active scalar equation: the scalar vorticity is transported by a velocity field that is itself reconstructed nonlocally from the vorticity.
This formulation makes apparent the main analytical difficulty of the problem. Although the vorticity equation is formally a linear transport equation, the transporting velocity depends nonlinearly and nonlocally on the transported quantity. The singularity of the Biot--Savart kernel also implies that the regularity of the velocity depends delicately on the function space to which the vorticity belongs.

The absence of a vortex-stretching term is the fundamental feature distinguishing the two-dimensional Euler equations from their three-dimensional counterpart. Indeed, in three dimensions the vorticity vector satisfies
\[
\partial_t\omega+(v\cdot\nabla)\omega
=
(\omega\cdot\nabla)v,
\]
and the term on the right-hand side may amplify the vorticity. In two dimensions, this mechanism disappears and the vorticity is simply transported by the flow.
\\
More precisely, let \(X(t,x)\) denote the Lagrangian flow associated with \(v\), defined by
\[
\frac{d}{dt}X(t,x)=v\bigl(t,X(t,x)\bigr),
\qquad
X(0,x)=x.
\]
Then
\begin{equation}\label{vorticity-transport-flow}
\omega\bigl(t,X(t,x)\bigr)=\omega_0(x).
\end{equation}
Since \(\operatorname{div}v=0\), the flow map \(X(t,\cdot)\) preserves Lebesgue measure. Consequently, all distribution functions of the vorticity are conserved. In particular,
\begin{equation}\label{Lp-vorticity-conservation}
\|\omega(t)\|_{L^p}
=
\|\omega_0\|_{L^p},
\qquad
1\leqslant p\leqslant\infty,
\end{equation}
whenever these quantities are well defined. More generally, for every sufficiently regular function \(\beta:\RR\to\RR\),
\[
\int_{\CC}\beta\bigl(\omega(z,t)\bigr)\,dA(z)
=
\int_{\CC}\beta\bigl(\omega_0(z)\bigr)\,dA(z)\,,
\]
with $dA$ being two dimensional Lebesgue measure.
Thus the Euler dynamics preserves not only the total vorticity, but an infinite family of quantities usually referred to as Casimir invariants.

\section{Global well-posedness}
The global solvability of the two-dimensional Euler equations is one of the fundamental consequences of the transport structure described above. The first global existence result for sufficiently smooth planar flows goes back to Wolibner \cite{Wolibner1933}. Alternative proofs and substantial extensions were subsequently obtained by Kato \cite{Kato1967} and many other authors, leading to well-posedness results in various functional settings, including Sobolev, Besov, and Triebel--Lizorkin spaces. We refer to \cite{Chemin1998,MB02} and the references therein for further details.
The key mechanism behind global regularity in two dimensions is the conservation of the $L^\infty$ norm of the vorticity, which follows directly from \eqref{Lp-vorticity-conservation}.
Combined with logarithmic estimates controlling $\|\nabla v\|_{L^\infty}$ in terms of the vorticity, this prevents the growth of the relevant higher-order norms from producing a finite-time singularity and allows the local solution to be continued globally.
\\
A decisive step toward solutions with nonsmooth vorticity was made by Yudovich \cite{Yudovich1963}. In the whole-plane setting, his theorem may be stated as follows.

\begin{mytheorem}{}{thm:Yudovich}
Assume that
\[
\omega_0\in L^1(\RR^2)\cap L^\infty(\RR^2).
\]
Then the two-dimensional Euler equations admit a unique global weak solution such that
\[
\omega
\in
L^\infty\bigl(\RR_+;L^1(\RR^2)\cap L^\infty(\RR^2)\bigr).
\]
Moreover,
\[
\|\omega(t)\|_{L^p}
=
\|\omega_0\|_{L^p},
\qquad
1\leqslant p\leqslant\infty,
\]
and the vorticity is transported by a unique measure-preserving flow.
\end{mytheorem}

The significance of Yudovich's theorem is that bounded vorticity does not generally produce a Lipschitz velocity. Instead, the Biot--Savart law gives the logarithmic modulus of continuity
\begin{equation}\label{log-Lipschitz}
|v(x,t)-v(y,t)|
\leqslant
C|x-y|
\left(
1+\log\frac{1}{|x-y|}
\right),
\qquad
0<|x-y|\leqslant\tfrac12.
\end{equation}
This log-Lipschitz regularity is nevertheless sufficient to guarantee uniqueness of the Lagrangian flow through Osgood's criterion. Yudovich's argument therefore identifies a borderline regularity class in which the nonlinear transport equation remains globally well posed.

Beyond the Yudovich class, global weak solutions are known under substantially weaker
assumptions on the vorticity. For instance, if
\[
\omega_0\in L^p(\mathbb R^2),\qquad 1<p<\infty,
\]
one can construct global weak solutions with
\[
\omega\in L^\infty(\mathbb R_+;L^p(\mathbb R^2)),
\]
although uniqueness is generally open in this regime. At an even lower level of regularity,
Delort established global existence for certain vortex-sheet initial data, allowing the vorticity
to be a bounded Radon measure of distinguished sign, up to an \(H^{-1}\) contribution
\cite{Delort1991}. These results illustrate the broad existence theory available below the
Yudovich class, but they will not play a role in the vortex-patch analysis developed here.

\section{Vortex patches}

A particularly important class of Yudovich solutions is given by vortex patches. In this case, the initial vorticity has the form
\[
\omega_0=\gamma\mathbf 1_{D_0},
\]
where \(\gamma\in\RR^\star\) and \(D_0\subset\RR^2\) is a bounded domain. Since the vorticity is transported by the flow, the solution remains a characteristic function:
\[
\omega(t)=\gamma\mathbf 1_{D_t},
\qquad
D_t=X(t,D_0).
\]
Thus the infinite-dimensional Euler system reduces to the evolution of the boundary
\(\partial D_t\).
\\
Yudovich's theorem immediately guarantees the global existence and uniqueness of a weak solution for every bounded initial patch. It does not, however, imply that the boundary of the patch remains smooth. The global persistence of boundary regularity was established by Chemin \cite{Chemin1993,Chemin19930} and independently, through a different argument, by Bertozzi and Constantin \cite{BertozziConstantin1993}. Their results imply that if
\[
\partial D_0\in C^{1+\alpha},
\qquad
0<\alpha<1,
\]
then
\[
\partial D_t\in C^{1+\alpha}
\]
for every finite time \(t\geq0\). Higher regularity is propagated as well. 
\\

\section{Preliminaries on rotating vortices}

Rotating vortices constitute one of the most fundamental classes of coherent structures in two-dimensional ideal fluid dynamics. They
provide canonical examples of nontrivial solutions of the Euler
equations whose geometry is preserved by the flow, evolving only
through a rigid motion. Besides their intrinsic physical relevance,
such configurations play a central role in the mathematical analysis
of vortex dynamics, as they often arise as stationary or relative
equilibria of the vorticity equation and serve as the starting point
for bifurcation and stability theories.
\\
The purpose of this section is to recall the basic framework and
establish a few elementary properties of rigidly rotating vortices
that will be used throughout this monograph. In particular, we formulate
the notion of a rotating solution, identify the geometric constraints
imposed by the Euler dynamics, and derive several simple consequences
that will considerably simplify the subsequent analysis.
 Recall that the vorticity
\[
\omega=\partial_1v^2-\partial_2v^1
\]
satisfies the transport equation
\begin{equation}\label{vorts1}
\partial_t\omega+v\cdot\nabla\omega=0,
\qquad
v=\nabla^\perp\Delta^{-1}\omega.
\end{equation}
A planar rigid motion is, in general, a combination of a translation
and a rotation. Since our main interest lies in rotating coherent
structures, we restrict throughout this section to pure rotations
about a fixed center.
For \(\theta\in\RR\), let
\[
Q_\theta
=
\begin{pmatrix}
\cos\theta&-\sin\theta\\
\sin\theta&\cos\theta
\end{pmatrix}
\]
denote the counterclockwise rotation matrix of angle \(\theta\). The
rotation of center \(x_0\in\RR^2\) and angle \(\theta\) is then defined by
\[
\mathcal R_{x_0,\theta}(x)
=
x_0+Q_\theta(x-x_0).
\]

\begin{mydefinition}{}{defa1}
Let
\[
\omega_0\in L^1(\RR^2)\cap L^\infty(\RR^2),
\]
and let \(\omega\) be the unique global Yudovich solution of
\eqref{vorts1} with initial datum \(\omega_0\). We say that
\(\omega_0\) is a \emph{rotating vorticity} about \(x_0\) if there
exists a smooth, nonconstant function
\[
\theta:\RR_+\to\RR
\]
such that
\begin{equation}\label{rotating-vorticity}
\omega(t,x)
=
\omega_0\bigl(\mathcal R_{x_0,-\theta(t)}(x)\bigr),
\qquad
x\in\RR^2,\quad t\geq0.
\end{equation}
\end{mydefinition}
Thus, at every time, the vorticity profile is obtained from its
initial configuration by a rigid rotation around \(x_0\). The
assumption
\[
\omega_0\in L^1(\RR^2)\cap L^\infty(\RR^2)
\]
ensures, by Yudovich's theorem, that the corresponding Euler solution
is globally defined and unique.
In the vortex-patch setting, the preceding definition has a simple
geometric interpretation. Suppose that
\[
\omega_0=\chi_D,
\]
where \(D\subset\RR^2\) is a bounded domain. Then \(\omega_0\) is a
rotating vorticity if and only if
\[
\omega(t)=\chi_{D_t},
\qquad
D_t=\mathcal R_{x_0,\theta(t)}(D).
\]
Such a domain \(D\) is called a \emph{rotating vortex patch}, or a
\emph{V-state}.
\\
We next recall two elementary invariants of Euler
flows.
Let \(\omega\) be a compactly supported solution of \eqref{vorts1}.
Its total vorticity is defined by
\[
m(t)
=
\int_{\RR^2}\omega(t,x)\,dx.
\]
If \(m(t)\neq0\), its center of vorticity is defined by
\[
\overline{x}(t)
=
\frac{1}{m(t)}
\int_{\RR^2}x\,\omega(t,x)\,dx.
\]
The terminology ``center of vorticity'' is sometimes preferable to
``center of mass,'' since the vorticity need not be nonnegative.

\begin{myproposition}{}{cons1}
Let \(\omega_0\) be a smooth, compactly supported initial vorticity
with
\[
m(0)\neq0.
\]
Then, for every \(t\geqslant0\),
\[
m(t)=m(0)
\qquad\text{and}\qquad
\overline{x}(t)=\overline{x}(0).
\]
\end{myproposition}

\begin{proof}
Since \(\operatorname{div}v=0\), the vorticity equation may be written
in conservative form as
\[
\partial_t\omega+\operatorname{div}(v\omega)=0.
\]
Integrating over \(\RR^2\), and using the spatial decay of the
solution, yields
\[
\frac{d}{dt}\int_{\RR^2}\omega(t,x)\,dx=0.
\]
Hence
\[
m(t)=m(0).
\]
Equivalently, if \(X(t,\cdot)\) denotes the flow generated by \(v\),
then
\[
\omega(t,x)
=
\omega_0\bigl(X^{-1}(t,x)\bigr).
\]
Because the flow is measure-preserving, the total vorticity is
conserved.
It remains to prove the conservation of the first moment. For
\(j\in\{1,2\}\), define
\[
f_j(t)
=
\int_{\RR^2}x_j\omega(t,x)\,dx.
\]
Using the vorticity equation and integrating by parts, we find
\[
\begin{aligned}
f_j'(t)
&=
-\int_{\RR^2}x_jv\cdot\nabla\omega\,dx\\
&=
\int_{\RR^2}v^j\omega\,dx.
\end{aligned}
\]
We claim that
\[
\int_{\RR^2}v^j\omega\,dx=0,
\qquad j=1,2.
\]
 Indeed,
writing
\[
v(x)
=
\int_{\RR^2}K(x-y)\omega(y)\,dy,
\qquad
K(z)=\frac{1}{2\pi}\frac{z^\perp}{|z|^2},
\]
one has
\[
\begin{aligned}
\int_{\RR^2}v(x)\omega(x)\,dx
&=
\iint_{\RR^2\times\RR^2}
K(x-y)\omega(x)\omega(y)\,dx\,dy\\
&=0,
\end{aligned}
\]
because \(K(y-x)=-K(x-y)\).
Therefore
\[
f_j'(t)=0,
\qquad j=1,2,
\]
and hence
\[
\int_{\RR^2}x\,\omega(t,x)\,dx
=
\int_{\RR^2}x\,\omega_0(x)\,dx.
\]
Combining this identity with the conservation of \(m(t)\) gives
\[
\overline{x}(t)=\overline{x}(0).
\]
This ends the proof.
\end{proof}
The notion of a rotating vortex patch introduced above still allows, a priori,
for arbitrary centers of rotation and time-dependent angular velocities.
However, the Euler dynamics is much more rigid. The conservation laws of the
flow, together with the geometric structure of rigid rotations, impose strong
constraints on any rotating solution. In particular, the center of rotation
cannot be chosen freely: it is uniquely determined by the conserved first
moment of the vorticity. Moreover, the angular speed of rotation cannot vary
with time, except in the trivial case of the Rankine vortex. Thus every
nontrivial rotating vortex patch performs a \emph{uniform} rotation about its
center of mass. These observations considerably simplify the analysis of
rotating patches, reducing the problem to the determination of a profile
rotating with a constant angular velocity. This rigidity is summarized in the following proposition.

\begin{myproposition}{}{center-mass}
Let \(D\subset\RR^2\) be a bounded non radial simply connected domain with
\(C^1\) boundary, and let \(\omega_0=\chi_D\). Assume that the
corresponding vortex patch   undergoes a nontrivial rigid rotation about
a fixed point \(x_0\), namely
\[
D_t=\mathcal R_{x_0,\theta(t)}(D),
\qquad \theta\in C^1(\RR_+).
\]
 Then the following holds.
\begin{enumerate}
\item The point \(x_0\) coincides with the center of mass of \(D\):
\[
x_0=\frac{1}{|D|}\int_Dx\,dx.
\]
\item The angular velocity is constant. More precisely, there exist
constants \(\Omega,\theta_0\in\RR\) such that
\[
\theta(t)=\Omega t+\theta_0,
\qquad t\geqslant 0.
\]
\end{enumerate}
\end{myproposition}

\begin{proof}
Let
\[
X(t)=\frac{1}{|D_t|}\int_{D_t}x\,dx
\]
denote the center of mass of the patch.  By Proposition
\ref{prop:cons1}, we infer
\[
\overline{x}(t)=\overline{x}(0),
\qquad t\geqslant 0.
\]
 We may absorb the initial phase into the
parametrization and assume that \(\theta(0)=0\). Using
\[
D_t=\mathcal R_{x_0,\theta(t)}(D),\quad \hbox{with}\quad \theta(0)=0
\]
and performing a change of variables, we obtain
\[
\begin{aligned}
\overline{x}(t)
&=\frac{1}{|D|}\int_D
\bigl(x_0+Q_{\theta(t)}(x-x_0)\bigr)\,dx\\
&=x_0+Q_{\theta(t)}\bigl(\overline{x}(0)-x_0\bigr).
\end{aligned}
\]
Consequently,
\[
\bigl(Q_{\theta(t)}-\textnormal{Id}\bigr)\bigl(\overline{x}(0)-x_0\bigr)=0.
\]
We distinguish two cases. Suppose first that there exists $t>0$ such that $\theta(t)\notin 2\pi\ZZ$. For such a time, the matrix $Q_{\theta(t)}-\textnormal{Id}$ is invertible, and therefore
\[
\overline{x}(0)=x_0.
\]
On the other hand, if $\theta(t)\in 2\pi\ZZ$ for every $t\geqslant0$, then
\[
D_t=D,\qquad t\geqslant0,
\]
and hence the patch is stationary. By Fraenkel's rigidity result, see Theorem \ref{thm:Rigidity-Gom-Yap}, the domain $D$ must then be a disk, which is a contradiction.
This proves the first assertion.\\
We now prove that the angular velocity is constant. After translating
the coordinates, we may assume without loss of generality that
\(x_0=0, \theta(0)=0\). 
Let
\(
s\longmapsto\gamma_0(s)
\)
be a \(C^1\) parametrization of \(\partial D\). Then
\[
\gamma_t(s)=e^{i\theta(t)}\gamma_0(s)
\]
parametrizes \(\partial D_t\). The kinematic boundary condition states
that the difference between the velocity of the parametrization and
the fluid velocity is tangent to the boundary, see Section \ref{sec-kinematic}. In complex notation,
this reads
\[
\operatorname{Im}\left(
\bigl(\partial_t\gamma_t(s)-v(t,\gamma_t(s))\bigr)
\overline{\partial_s\gamma_t(s)}
\right)=0.
\]
Using the rotational covariance of the Biot--Savart law, we have
\[
v\bigl(t,e^{i\theta(t)}z\bigr)
=e^{i\theta(t)}v_0(z),
\]
where \(v_0\) is the velocity generated by \(\chi_D\). Moreover,
\[
\partial_t\gamma_t
=i\dot\theta(t)e^{i\theta(t)}\gamma_0,
\qquad
\partial_s\gamma_t
=e^{i\theta(t)}\gamma_0'.
\]
Substituting these identities into the boundary equation gives
\[
\dot\theta(t)
\operatorname{Re}\left(
\gamma_0(s)\overline{\gamma_0'(s)}
\right)
=
\operatorname{Im}\left(
v_0(\gamma_0(s))
\overline{\gamma_0'(s)}
\right).
\]
Equivalently,
\[
\tfrac{\dot\theta(t)}{2}
\tfrac{d}{ds}|\gamma_0(s)|^2
=
\operatorname{Im}\left(
v_0(\gamma_0(s))
\overline{\gamma_0'(s)}
\right).
\]
The right-hand side is independent of \(t\). If there exists
\(s_0\) such that
\[
\frac{d}{ds}|\gamma_0(s_0)|^2\neq0,
\]
then
\[
\dot\theta(t)
=
2\,
\frac{
\operatorname{Im}\left(
v_0(\gamma_0(s_0))
\overline{\gamma_0'(s_0)}
\right)}
{
\displaystyle\frac{d}{ds}|\gamma_0(s_0)|^2
},
\]
which is independent of \(t\). Thus \(\dot\theta(t)=\Omega\) for
some constant \(\Omega\), and hence
\[
\theta(t)=\Omega t+\theta_0.
\]
It remains to consider the case
\[
\tfrac{d}{ds}|\gamma_0(s)|^2=0
\qquad\text{for every }s.
\]
Then \(|\gamma_0(s)|\) is constant on \(\partial D\). Since \(D\) is
simply connected and its boundary is a Jordan curve, \(D\) must be a
disk centered at the origin. This is precisely the Rankine vortex,
which has been excluded by assumption. The proof is complete.
\end{proof}

\section{Contour dynamics equation and rotating patches}\label{sec-kinematic}

The evolution of a vortex patch may be described entirely through the
motion of its boundary. Indeed, since the vorticity is transported by
the Euler flow, the boundary of the patch is a material curve and
therefore moves with the fluid velocity. However, a parametrization of
the boundary is not uniquely determined: one may freely add a
tangential velocity without changing the underlying curve. Thus, the
geometrically relevant quantity is only the normal component of the
boundary velocity. The purpose of this section is to formulate this
kinematic condition in a parametrization-independent form and then
specialize it to rigidly rotating patches.

Let \(D_t\subset\RR^2\) be a simply connected vortex patch whose
boundary remains sufficiently smooth for all times under
consideration. Consider two proper parametrizations of \(\partial D_t\),
\[
\alpha\longmapsto z(\alpha,t),
\qquad
\beta\longmapsto\eta(\beta,t),
\]
where the parameter intervals are understood with their endpoints
identified. Assume that both parametrizations are continuously
differentiable with respect to the spatial parameter and time. Since
they describe the same curve, there exists a time-dependent change of
parameter \(\alpha=\alpha(\beta,t)\) such that
\[
\eta(\beta,t)=z(\alpha(\beta,t),t).
\]
Differentiating with respect to time gives
\[
\partial_t\eta(\beta,t)
=
\partial_\alpha z(\alpha,t)\,\partial_t\alpha(\beta,t)
+
\partial_tz(\alpha,t).
\]
The vector \(\partial_\alpha z(\alpha,t)\) is tangent to
\(\partial D_t\). Hence, taking the scalar product with the exterior
unit normal \(n\) at the common point
\[
z(\alpha,t)=\eta(\beta,t),
\]
we obtain
\begin{equation}\label{boundary}
\partial_t\eta(\beta,t)\cdot n
=
\partial_tz(\alpha,t)\cdot n.
\end{equation}
Therefore, the normal velocity of the boundary is independent of the
chosen parametrization. By contrast, the tangential component depends
on the parametrization and has no effect on the geometry of the
evolving patch.
\\
Since \(\partial D_t\) is transported by the Euler flow, its normal
velocity must coincide with the normal component of the fluid
velocity. Consequently,
\begin{equation}\label{vortex}
\partial_tz(\alpha,t)\cdot n
=
v(z(\alpha,t),t)\cdot n.
\end{equation}
This is the kinematic equation governing the motion of the patch
boundary. Equivalently,
\[
\bigl(\partial_tz(\alpha,t)-v(z(\alpha,t),t)\bigr)\cdot n=0,
\]
which expresses the fact that the difference between the velocity of
the chosen parametrization and the fluid velocity is tangent to the
boundary.
We next rewrite this condition in terms of the stream function. For a
vortex patch of unit intensity $\omega(t)= \chi_{D_t}$, and using \eqref{Biot-Savart-real} we get
\begin{equation}\label{stream}
\psi(z,t)
=
\frac{1}{2\pi}
\int_{D_t}\log|z-\zeta|\,d\zeta
=
\left(\frac{1}{2\pi}\log|\cdot|*\chi_{D_t}\right)(z).
\end{equation}
and 
\[
v(z,t)
=
\frac{i}{2\pi}
\int_{D_t}\frac{1}{\overline{z-\zeta}}\,d\zeta
=
\frac{i}{2\pi}
\left(\frac{1}{\overline z}*\chi_{D_t}\right)(z).
\]
Let \(\tau\) denote the positively oriented unit tangent vector to
\(\partial D_t\), and let \(n\) be the exterior unit normal. With the
convention above, one has
\[
n=-\tau^\perp.
\]
Therefore,
\[
v\cdot n
=
\nabla^\perp\psi\cdot n
=
-\nabla\psi\cdot\tau.
\]
If \(s\) denotes the arc-length parameter along \(\partial D_t\), then
\[
\nabla\psi\cdot\tau=\tfrac{d\psi}{ds},
\]
and hence
\[
v(z(\alpha,t),t)\cdot n
=
-\tfrac{d}{ds}\psi(z(\alpha,t),t).
\]
Substituting this identity into \eqref{vortex}, we obtain
\begin{equation}\label{vortex2}
\tfrac{d}{ds}\psi(z(\alpha,t),t)
=
-\partial_tz(\alpha,t)\cdot n.
\end{equation}
Equation \eqref{vortex2} is the contour-dynamics formulation that will
be used below. It involves only the normal component of the boundary
velocity and is therefore invariant under time-dependent
reparametrizations of the curve. In the case of a rigidly rotating
patch, the left-hand side can be compared with the normal velocity
generated by the rotation, leading to a stationary equation in the
co-rotating frame.

We now specialize the contour dynamics equation to the class of
rigidly rotating vortex patches, also known as \emph{V-states}. The
time dependence of such solutions is entirely described by a uniform
rotation of the domain, allowing the free-boundary evolution to be
reduced to a stationary equation on the initial boundary.\\
Assume that the vortex patch rotates with constant angular velocity
\(\Omega\) about its center of mass. By Proposition
\ref{prop:center-mass}, after a translation we may assume that the center
of mass coincides with the origin. The evolution of the patch is then
given by
\[
D_t=e^{i\Omega t}D,
\]
where \(D\subset\CC\) is a fixed simply connected domain. Let
\(\alpha\mapsto z(\alpha)\) be a proper \(C^1\)-parametrization of
\(\partial D\). Then
\[
z(\alpha,t)=e^{i\Omega t}z(\alpha)
\]
is a parametrization of \(\partial D_t\), and
\[
\partial_tz(\alpha,t)
=
i\Omega e^{i\Omega t}z(\alpha)
=
i\Omega z(\alpha,t).
\]
Substituting this identity into the kinematic equation
\eqref{vortex2} yields
\begin{equation}\label{rvortex}
\frac{d}{ds}\psi(z(\alpha,t),t)
=
-\partial_tz(\alpha,t)\cdot n
=
-i\Omega z(\alpha,t)\cdot n
=
\Omega z(\alpha,t)\cdot\tau.
\end{equation}

Since rotations preserve the Euclidean scalar product, we have
\[
z(\alpha,t)\cdot\tau(z(\alpha,t))
=
z(\alpha)\cdot\tau(z(\alpha)).
\]
Choosing the arc-length parametrization
\(\alpha=s\), we obtain
\[
\frac{d}{ds}|z(s)|^2
=
2\,z(s)\cdot\tau(s).
\]
Hence \eqref{rvortex} becomes
\begin{equation}\label{rvortex01}
\frac{d}{ds}\psi(z(s,t),t)
=
\frac{\Omega}{2}
\frac{d}{ds}|z(s)|^2.
\end{equation}
Next, we observe that the stream function is invariant in the
co-rotating frame. Indeed, using the representation
\eqref{stream} and performing the change of variables
\(y=e^{i\Omega t}\xi\), we obtain
\[
\begin{aligned}
\psi(z(s,t),t)
&=
\frac{1}{2\pi}
\int_{D_t}
\log|z(s,t)-y|\,dA(y)\\
&=
\frac{1}{2\pi}
\int_D
\log|z(s)-\xi|\,dA(\xi)
=
\psi(z(s)),
\end{aligned}
\]
where
\[
\psi(z)
=
\frac{1}{2\pi}
\int_D
\log|z-y|\,dA(y).
\]
Consequently,
\begin{equation}\label{rvortex02}
\tfrac{d}{ds}\psi(z(s))
=
\tfrac{\Omega}{2}
\tfrac{d}{ds}|z(s)|^2.
\end{equation}
Integrating along the boundary yields
\begin{equation}\label{rvortex3}
\psi(z)-\tfrac{\Omega}{2}|z|^2
=
\mu,
\qquad
\forall z\in\partial D,
\end{equation}
for some constant \(\mu\in\RR\).
Equation \eqref{rvortex3} is the classical boundary formulation of
the V-states equation: the effective potential
\(\psi-\frac{\Omega}{2}|z|^2\) is constant along the boundary of the
patch.
\\
Finally, recalling that
\[
v=\nabla^\perp\psi,
\]
equation \eqref{rvortex} is equivalently written as
\begin{equation}\label{master0}
\bigl(v(x)-\Omega x^\perp\bigr)\cdot n(x)=0,
\qquad
x\in\partial D,
\end{equation}
where \(n(x)\) denotes the outward unit normal vector to
\(\partial D\). Thus, in the rotating frame, the relative velocity
has no normal component on the boundary, which is precisely the
kinematic characterization of a rotating vortex patch.

\chapter{Rotating patches: Reformulation via Cauchy integrals}
{\it



In this  chapter we review the
Cauchy--Pompeiu formula and its relation to the Cauchy transform of a
domain. We then derive an  equivalent complex formulation of the
rotating-patch equation in terms of boundary Cauchy integrals.  Finally, we illustrate the strength of
this approach by recovering the classical Rankine vortices and
Kirchhoff ellipses through explicit computations of their Cauchy
transforms. We
next discuss the inverse problem associated with the Cauchy transform
and establish rigidity results for domains whose Cauchy transform has
a prescribed algebraic form.
}

\section{Cauchy--Pompeiu formula}
The velocity field generated by a vortex patch is naturally expressed
through the Cauchy transform of the patch. The purpose of this section
is to derive a boundary integral representation of this transform by
means of the Cauchy--Pompeiu formula. This reformulation allows the
rotating vortex equation to be written entirely in terms of boundary
quantities and provides the starting point for the contour dynamics
formulation.
\\
Let $D\subset\CC$ be a finitely connected domain whose boundary
$\Gamma=\partial D$ consists of finitely many positively oriented
$C^1$ Jordan curves. If $\varphi$ is of class
$C^1$ on an open neighborhood of $\overline D$, then the
Cauchy--Pompeiu formula reads
\begin{equation}\label{GC}
\varphi(z)\chi_D(z)
=
\frac{1}{2\pi i}
\int_{\Gamma}
\frac{\varphi(\xi)}{\xi-z}\,d\xi
-
\frac{1}{\pi}
\int_D
\frac{\partial\varphi}{\partial\overline{\xi}}(\xi)
\frac{1}{\xi-z}\,dA(\xi),
\qquad z\in\CC.
\end{equation}
When $z\in\partial D$, the boundary integral is understood as a Cauchy principal value.
We shall apply this identity to the function
$\varphi(z)=\overline z$, whose $\bar\partial$-derivative satisfies
\[
\partial_{\overline z}\overline z=1.
\]
Substituting this choice into \eqref{GC} yields
\begin{equation}\label{GC01}
\overline z\,\chi_D(z)
=
\frac{1}{2\pi i}
\int_{\Gamma}
\frac{\overline{\xi}}{\xi-z}\,d\xi
+
\mathcal C(\chi_D)(z),
\qquad z\in\CC,
\end{equation}
where
\begin{equation}\label{Cauchy}
\mathcal C(\chi_D)(z)
=
\frac{1}{\pi}
\int_D
\frac{1}{z-\xi}\,dA(\xi),
\qquad z\in\CC,
\end{equation}
denotes the Cauchy transform of the domain $D$.
We now use this identity to reformulate the rotating vortex equation.
Recall that the velocity field is related to the stream function by
\[
\tfrac{i}{2}\,\overline{v(z)}
=
\partial_z\psi(z)
=
\tfrac14\,\mathcal C(\chi_D)(z).
\]
Taking the interior boundary limit in \eqref{GC01}, we obtain
\begin{equation}\label{Cauchy-1}
\mathcal C(\chi_D)(z)
=
\fint_{\partial D}
\frac{\overline z-\overline\xi}{\xi-z}\,d\xi,
\qquad z\in\CC,
\end{equation}
where we use the notation
\[
\fint_{\partial D}
:=
\frac{1}{2\pi i}
\int_{\partial D}.
\]
Consequently,
\begin{equation}\label{BoundaryVelocity}
4\partial_z\psi(z)
=
\fint_{\partial D}
\frac{\overline z-\overline\xi}{\xi-z}\,d\xi,
\qquad z\in\CC.
\end{equation}
Substituting this representation into \eqref{master0}, we arrive at the
following complex formulation of the rotating vortex equation:
\begin{equation}\label{Master1-0}
F(\Omega,D)(z)
:=
\operatorname{Re}
\left\{
\bigl(\mathcal C(\chi_D)(z)-2\Omega\,\overline z\bigr)
\tau(z)
\right\}
=0,
\qquad z\in\partial D,
\end{equation}
where $\tau(z)$ denotes the positively oriented unit tangent vector to
$\partial D$ at the point $z$.

\section{ The Cauchy transform and the Plemelj formulas}

The Cauchy transform provides a convenient complex representation of
the velocity field associated with a vortex patch. The aim of this
section is to relate it to the boundary Cauchy integrals through the
Plemelj--Sokhotski formulas. Besides yielding a boundary formulation
of the rotating vortex equation, this representation is particularly
useful for explicit computations in several important examples.
Recall from \eqref{Cauchy} that the Cauchy transform of a bounded domain $D\subset\CC$ is
defined by
\[
\mathcal C(\chi_D)(z)
=
\frac1\pi
\int_D
\frac{1}{z-\xi}\,dA(\xi),
\qquad z\in\CC.
\]
It is well known that $\mathcal C(\chi_D)$ is continuous on
$\CC$, holomorphic on $\CC\setminus\overline D$, and
vanishes at infinity.
\\
Assume now that the boundary
$\Gamma=\partial D$ is of class $C^1$. We introduce the interior and
exterior Cauchy integrals
\[
\gamma^+(z)
=
\fint_\Gamma
\frac{\overline\xi}{\xi-z}\,d\xi,
\qquad z\in D,
\]
and
\[
\gamma^-(z)
=
\fint_\Gamma
\frac{\overline\xi}{\xi-z}\,d\xi,
\qquad
z\in\CC\setminus\overline D,
\]
where we recall the notation
\[
\fint_\Gamma
=
\frac{1}{2\pi i}
\int_\Gamma .
\]
Both functions $\gamma^\pm$ are holomorphic in their respective domains.
Moreover, they extend continuously up to the boundary. Indeed,
using the identity
\[
\frac{\overline\xi}{\xi-z}
=
\frac{\overline\xi-\overline z}{\xi-z}
+
\frac{\overline z}{\xi-z},
\]
together with the Cauchy integral formula, one obtains
\[
\gamma^\pm(z)
=
\fint_\Gamma
\frac{\overline\xi-\overline z}{\xi-z}\,d\xi
+
\overline z\,\chi_D(z),
\]
where the first integral defines a continuous function of $z$. For simplicity, we denote these
boundary traces by the same symbols $\gamma^\pm$.
\\
The Plemelj--Sokhotski formulas, see for instance\cite[p.~143]{V}, state that
\[
\gamma^+(z)
=
\operatorname{p.v.}
\fint_\Gamma
\frac{\overline\xi}{\xi-z}\,d\xi
+
\frac{\overline z}{2},
\qquad z\in\Gamma,
\]
and
\[
\gamma^-(z)
=
\operatorname{p.v.}
\fint_\Gamma
\frac{\overline\xi}{\xi-z}\,d\xi
-
\frac{\overline z}{2},
\qquad z\in\Gamma,
\]
where the integrals are understood in the Cauchy principal value
sense. Their difference yields the classical jump relation
\begin{equation}\label{plem}
\overline z
=
\gamma^+(z)-\gamma^-(z),
\qquad z\in\Gamma.
\end{equation}
The Cauchy transform can now be reconstructed from the boundary
Cauchy integrals. Indeed, by the  formula \eqref{GC01} we infer
\begin{equation}\label{Cauchydins}
\mathcal C(\chi_D)(z)
=
\overline z-\gamma^+(z),
\qquad
z\in\overline D,
\end{equation}
and
\begin{equation}\label{Cauchyfora}
\mathcal C(\chi_D)(z)
=
-\gamma^-(z),
\qquad
z\in\CC\setminus D.
\end{equation}
By the continuity of the Cauchy transform and the boundary traces of
$\gamma^\pm$, these identities remain valid on $\Gamma$.
Consequently, the boundary equation \eqref{Master1-0} can be written
solely in terms of the interior Cauchy integral
\begin{equation}\label{Master2-0}
F(\Omega,D)(z)
=
\operatorname{Re}
\Bigl\{
\bigl((1-2\Omega)\overline z-\gamma^+(z)\bigr)
\tau(z)
\Bigr\}
=
0,
\qquad
z\in\partial D,
\end{equation}
where $\tau(z)$ denotes the positively oriented unit tangent vector to
$\partial D$.

\section{Rankine vortices and Kirchhoff ellipses}
In this section, we compute explicitly the Cauchy transform
\eqref{Cauchy} for two fundamental examples of rotating vortex
patches: the Rankine vortex and the Kirchhoff ellipse. This enables us
to verify directly that both configurations satisfy the boundary
equation \eqref{Master2-0}, thereby recovering the classical result
that they evolve as uniformly rotating solutions of the Euler
equations.
\subsubsection{Rankine vortices}
We begin with the unit disk
\[
 D=\{z\in\CC:\ |z|<1\}.
\]
The Cauchy transform for an arbitrary disk follows immediately from
this case by translation and scaling.
 Since $\xi \overline{\xi} =1 $ on $\partial D=\mathbb{T},$ then
$$
\gamma^+(z)=\fint_{\mathbb{T}}\frac{1}{\xi(\xi-z)}d\xi=0, \quad  \,z\in
D\quad\hbox{and} \quad  \gamma^-(z)=-\frac1z, \quad  z\notin
\overline{D}.
$$
Therefore
\begin{equation*}
\mathcal{C}(\chi_D)(z)=
 \left\{
\begin{array}{ll}
\overline{z}, \quad z\in {D} \\
\frac1z, \quad z\notin {D}.
\end{array} \right.
\end{equation*}
 For a disk of center  $z_0$ and radius $r$  translating and
 dilating the previous result gives
     \begin{equation*}
\mathcal{C}(\chi_D)(z)=
 \left\{
\begin{array}{ll}
\overline{z}-\overline{z}_0, \quad  z\in D \\
\frac{r^2}{z-z_0}, \quad z \notin D.
\end{array} \right.
     \end{equation*}
Substituting this expression into the boundary equation
\eqref{Master1-0}, we find that
\[
F(\Omega,z)=0,
\qquad z\in\partial D,
\]
for every $\Omega\in\mathbb R$. Consequently, every disk is a rotating
vortex patch. Since a disk is radially symmetric, its shape is
preserved under rotations, and hence the angular velocity $\Omega$ is
arbitrary.

\subsubsection{Kirchhoff ellipses}
We now turn to the second classical family of rotating vortex patches,
namely the Kirchhoff ellipses. Our goal is to recover, through the
boundary formulation developed in the previous sections, Kirchhoff celebrated result asserting that ellipses constitute uniformly
rotating solutions of the Euler equations for a unique angular
velocity. This example further illustrates the effectiveness of the
Cauchy transform approach in deriving explicit rotating solutions.
\\
Let $D$ be the domain enclosed by a noncircular  ellipse
\[
\frac{x^2}{a^2}+\frac{y^2}{b^2}=1.
\]
Without loss of generality, we assume that
\[
a> b>0,
\]
and define
\[
c^2=a^2-b^2.
\]

\begin{mytheorem}{}{Kirchhoff}
The elliptical vortex patch $\chi_D$ rotates uniformly if and only if
\[
\Omega=\frac{ab}{(a+b)^2}\cdot
\]
\end{mytheorem}
\begin{proof}
The equation of the ellipse can be rewritten in complex coordinates
as
\[
\overline z=Qz+F(z),
\qquad z\in\partial D,
\]
where
\[
Q=\frac{a-b}{a+b},
\qquad
F(z)=\frac{2ab}{z\left(1+\sqrt{1-\frac{c^2}{z^2}}\right)}.
\]
Notice that $F$ is holomorphic in $\CC\setminus[-c,c]$ and satisfies
\[
F(z)=O(|z|^{-1}),
\qquad |z|\to\infty.
\]
We first compute the interior Cauchy integral. By the decomposition
above,
\[
\gamma^+(z)
=
Q\fint_{\partial D}\frac{\xi}{\xi-z}\,d\xi
+
\fint_{\partial D}\frac{F(\xi)}{\xi-z}\,d\xi.
\]
Since $z\in D$, Cauchy's integral formula gives
\[
\fint_{\partial D}\frac{\xi}{\xi-z}\,d\xi=z.
\]
Moreover, the function
\[
\xi\longmapsto\frac{F(\xi)}{\xi-z}
\]
is holomorphic in the exterior of $\overline D$ and behaves like
$O(|\xi|^{-2})$ at infinity. Hence its contour integral vanishes,
which yields
\[
\gamma^+(z)=Qz,
\qquad z\in D.
\]
For later use, observe that if the ellipse has center $z_0$ and its
major axis forms an angle $\theta$ with the horizontal axis, then
\[
\gamma^+(z)
=
e^{-2i\theta}Q(z-z_0)+\overline{z_0},
\qquad z\in D.
\]
Next we compute the exterior Cauchy integral. Since
\[
\fint_{\partial D}\frac{\xi}{\xi-z}\,d\xi=0,
\qquad z\in\CC\setminus\overline D,
\]
we obtain
\[
\gamma^-(z)
=
\fint_{\partial D}\frac{F(\xi)}{\xi-z}\,d\xi.
\]

We claim that
\begin{equation}\label{Resid-the}
\fint_{\partial D}\frac{F(\xi)}{\xi-z}\,d\xi
=
-F(z),
\qquad
z\in\CC\setminus\overline D.
\end{equation}
To prove this, fix $R>|z|$ and consider the domain
\[
A_R=B(0,R)\setminus\overline D.
\]
The function
\[
\xi\longmapsto\frac{F(\xi)}{\xi-z}
\]
is holomorphic in $A_R$ except for the simple pole at $\xi=z$.
Applying the residue theorem and using the orientation
\[
\partial A_R=\partial D_R-\partial D,
\]
we obtain
\[
\fint_{\partial D_R}
\frac{F(\xi)}{\xi-z}\,d\xi
-
\fint_{\partial D}
\frac{F(\xi)}{\xi-z}\,d\xi
=
F(z).
\]
Since
\[
\left|
\fint_{\partial D_R}
\frac{F(\xi)}{\xi-z}\,d\xi
\right|
\leqslant
\frac{R}{R-|z|}
\sup_{|\xi|=R}|F(\xi)|,
\]
and $F(\xi)=O(|\xi|^{-1})$, we infer
\[
\lim_{R\to\infty}
\fint_{\partial D_R}
\frac{F(\xi)}{\xi-z}\,d\xi
=
0,
\]
which proves \eqref{Resid-the}. Consequently,
\[
\gamma^+(z)=Qz,
\qquad
z\in D,
\]
and
\[
\gamma^-(z)=-F(z),
\qquad
z\in\CC\setminus\overline D.
\]
Applying the representation formulas
\eqref{Cauchydins}--\eqref{Cauchyfora}, we obtain
\[
\mathcal C(\chi_D)(z)
=
\begin{cases}
\overline z-Qz,
&
z\in D,\\[2mm]
\displaystyle
\frac{2ab}
{z\left(1+\sqrt{1-\frac{c^2}{z^2}}\right)},
&
z\in\CC\setminus D.
\end{cases}
\]
A non-unit tangent vector to the ellipse is given by
\[
\tau(z)
=
-2a^2y+i\,2b^2x
=
i\Bigl((a^2+b^2)z+(b^2-a^2)\overline z\Bigr).
\]
Substituting the expression of the Cauchy transform into
\eqref{Master2-0}, we obtain
\[
F(\Omega,z)
=
-\textnormal{Im}\!\left\{
\Bigl((1-2\Omega)\overline z-Qz\Bigr)
\Bigl((a^2+b^2)z+(b^2-a^2)\overline z\Bigr)
\right\}.
\]
Expanding the product gives
\[
\begin{aligned}
\Bigl((1-2\Omega)\overline z-Qz\Bigr)
\Bigl((a^2+b^2)z+(b^2-a^2)\overline z\Bigr)
={}&
\Bigl((1-2\Omega)(a^2+b^2)+Q(a^2-b^2)\Bigr)|z|^2\\
&+(1-2\Omega)(b^2-a^2)\overline z^{\,2}
-Q(a^2+b^2)z^2.
\end{aligned}
\]
Since the first term is real,
\[
F(\Omega,z)
=
\textnormal{Im}\!\left\{
\Bigl((1-2\Omega)(b^2-a^2)
+Q(a^2+b^2)\Bigr)z^2
\right\}.
\]
Hence
\[
F(\Omega,\cdot)\equiv0
\quad\Longleftrightarrow\quad
(1-2\Omega)(b^2-a^2)+Q(a^2+b^2)=0.
\]
Using the expression of $Q$, this identity is equivalent to
\[
\Omega=\frac{ab}{(a+b)^2},
\]
which completes the proof.
\end{proof}

\section{An inverse characterization of ellipses}

The explicit computations carried out for the Rankine vortex and the
Kirchhoff ellipse naturally lead to the following inverse problem:
can one recover the geometry of a domain from the knowledge of its
interior Cauchy integral? In general, this inverse problem is not
uniquely solvable. However, the affine form of the interior Cauchy
integral turns out to be remarkably rigid. The next result shows that
if the interior Cauchy integral is an affine function, then the
underlying domain must necessarily be an ellipse. This provides a
characterization of Kirchhoff's ellipses solely in terms of their
Cauchy transform. This characterization was  established in \cite{HMV15}.
\begin{mytheorem}{}{Charac-ellip}
Let $\Gamma$ be a Jordan curve of class $C^1$ enclosing a bounded
domain $D$. Assume that there exist $Q\in \RR$ and $z_0\in \CC$ such
that
$$
\gamma^+(z)= \frac{1}{2\pi
i}\int_{\Gamma}\frac{\overline{\xi}}{\xi-z}d\xi=Q(z-z_0)+\overline{z_0},\quad
 z\in D.
$$
Then the curve $\Gamma$ is an ellipse of center $z_0$ with semi-axes
$a$ and $b$ satisfying
 \begin{equation*}
Q= \frac{a-b}{a+b}\cdot
     \end{equation*}
     \label{propinv}
\end{mytheorem}
A remarkable consequence of the proposition is that necessarily
\[
|Q|<1.
\]
Indeed, the limiting cases $Q=\pm1$ would force the ellipse to
degenerate into a line segment.

\begin{proof}
By the Plemelj jump formula \eqref{plem},
\[
\overline z
=
Q(z-z_0)+\overline{z_0}-\gamma^-(z),
\qquad z\in\Gamma,
\]
where $\gamma^-$ is holomorphic in
$\CC_\infty\setminus\overline D$ and satisfies
\[
\gamma^-(z)=O(z^{-1}),
\qquad z\to\infty.
\]
Squaring the preceding identity and taking the product with
$z-z_0$ give
\[
(\overline z-\overline{z_0})^2+(z-z_0)^2
=
(1+Q^2)(z-z_0)^2
-2Q(z-z_0)\gamma^-(z)
+\bigl(\gamma^-(z)\bigr)^2,
\]
and
\[
|z-z_0|^2
=
Q(z-z_0)^2
-(z-z_0)\gamma^-(z).
\]
We now form a linear combination of these two identities in order to
eliminate the quadratic term $(z-z_0)^2$. Let $A,B\in\mathbb R$. Then
\[
-A\Bigl((\overline z-\overline{z_0})^2+(z-z_0)^2\Bigr)
+B|z-z_0|^2
=
\bigl(BQ-A(1+Q^2)\bigr)(z-z_0)^2
+g(z),
\]
where
\[
g(z)
=
(2AQ-B)(z-z_0)\gamma^-(z)
-A\bigl(\gamma^-(z)\bigr)^2.
\]
Choosing
\[
A=Q,
\qquad
B=1+Q^2,
\]
cancels the quadratic term and yields
\[
-Q\Bigl((\overline z-\overline{z_0})^2+(z-z_0)^2\Bigr)
+(1+Q^2)|z-z_0|^2
=
g(z),
\qquad z\in\Gamma.
\]
The function $g$ is  holomorphic on $\CC\backslash
\overline{D}$ and has a limit at infinity given by
\begin{eqnarray*}
\lim_{z\to\infty}g(z)&=&(2AQ-B)\lim_{z\to \infty}\frac{1}{2 \pi i}\int_{\Gamma}\frac{\overline{\xi} z}{\xi-z}d\xi\\
&=& (1-Q^2) \frac{1}{2 \pi i}\int_{\Gamma}{\overline{\xi} }\,d\xi \\
&=& \frac{1-Q^2}{\pi} |D|,
\end{eqnarray*}
where we applied Green-Stokes in the last identity. Notice that $g$
has a continuous extension up to the boundary $\Gamma$ and takes
real values on this set. Then the imaginary part   of $g$ is a
harmonic function on the exterior domain $\CC\backslash
\overline{D}$, continuous up to the boundary and satisfying
$$
 \hbox{Im } g(z)=0, \quad z\in \Gamma \quad\hbox{and}\quad \lim_{z\to
\infty}\hbox{Im } g(z)=0 .
$$
By the maximum principle we conclude that $\hbox{Im } g$ is
identically zero on $\mathbb{C} \setminus \overline{D}$. Thus the
holomorphic function $g$ is real on $ \mathbb{C} \setminus
\overline{D}$ and consequently must be constant. This means that
$$
-Q\Big((\overline{z}-\overline{z_0})^2+(z-{z_0})^2\Big)+(1+Q^2)|z-z_0|^2=\frac{1-Q^2}{\pi} |D|,\quad
z\in \Gamma\,.
$$
 Set $X=\hbox{Re}(z-z_0)$ and $ Y=\hbox{Im
}(z-z_0)$ then
$$
(1-Q)^2 X^2+(1+Q)^2 Y^2=\frac{1-Q^2}{\pi} |D|, \quad \hbox{on}\quad \Gamma.
$$
If $|Q|>1$  the equation is impossible. When $Q=\pm1$, the equation degenerates into a line segment,
contradicting the assumption that $\Gamma$ is a Jordan curve
 enclosing a bounded domain.
If $|Q|<1$, this is precisely the equation of an ellipse centered at
 $z_0$. In particular, if $Q\geq0$ and $a\geqslant b$ denote the major and minor
semi-axes, then
\[
a=\frac{\sqrt C}{1-Q},
\qquad
b=\frac{\sqrt C}{1+Q},
\]
and therefore
\[
Q=\frac{a-b}{a+b}.
\]
  The proof of the desired
result is complete.
\end{proof}
    \chapter{Conformal mapping and Faber polynomial formulation of V-States}
{\it This chapter reformulates the V-states equation for simply connected vortex patches in terms of the exterior Riemann conformal mapping. Using Faber polynomials, we derive an explicit representation of the Cauchy transform and obtain a boundary formulation adapted to the conformal geometry of the domain. As applications, we recover Kirchhoff’s theorem for rotating ellipses and prove a rigidity result showing that ellipses are the only rotating vortex patches whose exterior conformal mapping has a finite Laurent expansion. This latter result is due to \mbox{Burbea $\cite{Burbea81}$.}}

\section{Riemann conformal mapping and Faber polynomials}
The main difficulty in the study of rotating vortex patches stems from the fact that the boundary equation is posed on the unknown interface $\partial D$. A standard way to overcome this difficulty is to transfer the problem onto the fixed unit circle by means of the Riemann conformal mapping associated with the domain. This approach, initiated by Burbea \cite{Burbea82}, has become a fundamental tool in the analysis of V-states. Moreover, the geometry of the domain is naturally encoded by the coefficients of the conformal mapping and by the associated Faber polynomials, which will play a central role throughout this chapter.
By the Riemann Mapping Theorem, there exists a biholomorphic mapping
\[
\Phi:\CC\setminus\overline{\mathbb D}\longrightarrow\CC\setminus\overline D,
\]
where
\[
\mathbb D=\{z\in\CC:\ |z|<1\}.
\]
Moreover, $\Phi$ admits the Laurent expansion
\[
\Phi(z)=a\left(z+\sum_{n\ge0}\frac{a_n}{z^n}\right),
\qquad |z|>1,
\]
for some nonzero complex number $a$. After a rotation of the independent variable, we may assume without loss of generality that $a>0$.
If $\partial D$ is a Jordan curve, Carath\'eodory's theorem asserts that $\Phi$ extends continuously to the unit circle $\TT=\{z\in\CC:\ |z|=1\}$. Furthermore, this extension is absolutely continuous, so that $\Phi'(w)$ exists for almost every $w\in\TT$ and belongs to $L^1(\TT)$. If, in addition, $\partial D$ is of class $C^{1+\alpha}$, with $\alpha\in(0,1)$, then the Kellogg--Warschawski theorem ensures that the boundary extension satisfies
\[
\Phi\in C^{1+\alpha}(\TT).
\]
The Faber polynomials $(F_n)_{n\ge0}$ associated with the domain $D$ are defined by
\[
F_n(z)=\textnormal{polynomial part of }[\Phi^{-1}(z)]^n,
\]
that is, if
\[
[\Phi^{-1}(z)]^n
=
F_n(z)
+
\sum_{k=1}^\infty\frac{a_{n,k}}{z^k},
\]
then $F_n$ is obtained by retaining only the nonnegative powers of $z$. They form a sequence of polynomials adapted to the geometry of the domain and satisfy a number of remarkable identities. One of the most important is the classical generating formula due to Faber,
\begin{equation}\label{Fab1}
\frac{\xi\Phi'(\xi)}{\Phi(\xi)-z}
=
\sum_{n\geqslant 0}F_n(z)\xi^{-n},
\qquad |\xi|>1,\quad z\in\overline D.
\end{equation}
Each polynomial $F_n$ has degree $n$, and the first two are given by
\begin{equation}\label{Tot1}
F_0(z)=1,
\qquad
F_1(z)=\frac{z}{a}-a_0.
\end{equation}
The Faber polynomials satisfy the recurrence relation, see for instance \cite{Curtis},
\[
F_{n+1}(z)
=
F_1(z)F_n(z)
-\frac1a\sum_{k=1}^{n-1}a_kF_{n-k}(z)
-(n+1)\frac{a_n}{a},
\qquad n\geqslant 1.
\]
They also admit the integral representation
\begin{equation}\label{TR1}
F_n(z)
=
\fint_{|\xi|=r}
\xi^n
\frac{\Phi'(\xi)}{\Phi(\xi)-z}\,d\xi,
\qquad r>1,\quad z\in\overline D,
\end{equation}
which will be repeatedly used in the sequel.

\section{Conformal mapping reformulation and link with Faber polynomials}
The purpose of this section is to reformulate the V-states equation in terms of the conformal mapping and the associated Faber polynomials. The key observation is that the Cauchy transform, which governs the boundary dynamics, can be expressed explicitly through the Faber expansion of the conformal mapping. This leads to a remarkable reformulation of the nonlinear free-boundary problem as a boundary identity involving a holomorphic function determined solely by the coefficients of the conformal mapping.
Assume that $\partial D$ is of class $C^1$, and let
\[
\Phi:\TT\longrightarrow\partial D
\]
be the boundary parametrization induced by the exterior conformal mapping. Performing the change of variables $z=\Phi(w)$ and $\xi=\Phi(\tau)$ in the boundary equation \eqref{Master1-0}, together with the representation \eqref{Cauchy-1} of the Cauchy transform, we obtain the equivalent formulation
\begin{align}\label{rotsq12}
\textnormal{Im}\{G(\Omega,\Phi(w))\}=0,
\qquad
w\in\TT,
\end{align}
where
\[
G(\Omega,\Phi(w))
=
\left(
2\Omega\,\overline{\Phi(w)}
+
\fint_{\TT}
\frac{\overline{\Phi(\xi)}-\overline{\Phi(w)}}
{\Phi(\xi)-\Phi(w)}
\Phi'(\xi)\,d\xi
\right)
w\Phi'(w).
\]
The next theorem establishes the fundamental link between the V-states equation and the Faber polynomials. It shows that the boundary condition can be integrated and rewritten in terms of a holomorphic function whose coefficients are precisely those of the Faber expansion associated with the conformal mapping.
\begin{mytheorem}{}{Faber-Vstate}
    Let $D$ be a simply connected bounded domain such that $\partial D$ is of class $C^{1+\alpha}, \alpha\in(0,1).$ Then ${\bf{1}}_D$ is rotating around the origin with the angular velocity $\Omega$ if and only if 
    $$
\tfrac{2\Omega-1}{2} |z|^2+\textnormal{Re } H(z)= \mu, \,\forall\,z\in \partial D
$$where 
\begin{equation*}
F(z):=a\sum_{n\geqslant0}\overline{a_n} F_n(z), \quad H(z):=\int_{z_0}^zF(\xi)d\xi
\end{equation*}
and $H$ is a holomorphic primitive of $F.$
\label{Theom-form-Fab}
\end{mytheorem}
\begin{proof}
We first establish the representation of the Cauchy integral in terms of the Faber polynomials. For $z\in D$, define
\[
I(z):=\fint_{\TT}\frac{\overline{\Phi(\xi)}-\overline z}{\Phi(\xi)-z}\,\Phi'(\xi)\,d\xi.
\]
On the unit circle, the Laurent expansion of $\Phi$ gives
\begin{equation}\label{overl}
\overline{\Phi(\xi)}
=
a\left(\xi^{-1}+\sum_{n\ge0}\overline{a_n}\xi^n\right),
\qquad \xi\in\TT.
\end{equation}
We first consider the case where $\partial D$ is analytic. Then the coefficients $(a_n)$ decay exponentially and the function
\[
\Phi^\ast(\xi):=
a\left(\xi^{-1}+\sum_{n\ge0}\overline{a_n}\xi^n\right)
\]
is holomorphic in an annulus $1<|\xi|<r_0$, for some $r_0>1$, satisfying
\[
\Phi^\ast(\xi)=\overline{\Phi(\xi)},
\qquad \xi\in\TT.
\]
Fix $1<r<r_0$. Since the integrand is holomorphic in the annulus bounded by $\TT$ and $r\TT$, Cauchy's theorem yields
\[
I(z)
=
\fint_{r\TT}
\frac{\Phi^\ast(\xi)-\overline z}{\Phi(\xi)-z}\,
\Phi'(\xi)\,d\xi.
\]
Using
\[
\fint_{r\TT}\frac{\Phi'(\xi)}{\Phi(\xi)-z}\,d\xi=1,
\]
we obtain
\[
I(z)
=
-\overline z
+
\fint_{r\TT}
\Phi^\ast(\xi)\frac{\Phi'(\xi)}{\Phi(\xi)-z}\,d\xi.
\]
Substituting the expansion of $\Phi^\ast$ gives
\[
\begin{aligned}
\fint_{r\TT}
\Phi^\ast(\xi)\frac{\Phi'(\xi)}{\Phi(\xi)-z}\,d\xi
={}&
a\fint_{r\TT}
\frac{\Phi'(\xi)}{\xi(\Phi(\xi)-z)}\,d\xi\\
&+
a\sum_{n\geqslant0}\overline{a_n}
\fint_{r\TT}
\xi^n\frac{\Phi'(\xi)}{\Phi(\xi)-z}\,d\xi.
\end{aligned}
\]
The first integral vanishes because its integrand is holomorphic in $|\xi|>1$ and behaves like $O(|\xi|^{-2})$ at infinity. By the integral representation \eqref{TR1} of the Faber polynomials,
\[
\fint_{r\TT}
\xi^n\frac{\Phi'(\xi)}{\Phi(\xi)-z}\,d\xi
=
F_n(z).
\]
Hence
\begin{equation}\label{T2}
I(z)
=
-\overline z
+
a\sum_{n\geqslant 0}\overline{a_n}F_n(z),
\qquad z\in D.
\end{equation}
We now extend this identity to the case where $\partial D$ is of class $C^{1+\alpha}$. The coefficients of the conformal map satisfy
\[
|a_n|
\le
\frac{C}{(1+n)^{1+\alpha}},
\]
while Theorem~2 of \cite{Lesley} asserts the existence of two constants $A$ and $B$  such that
\[
\sup_{z\in\overline D}|F_n(z)|
\le
A\log(1+n)+B.
\]
Therefore
\[
\sum_{n\ge0}|a_n|
\sup_{z\in\overline D}|F_n(z)|
<\infty,
\]
and the series defining $F$ converges absolutely and uniformly on $\overline D$.
To justify \eqref{T2}, consider the truncated conformal mappings
\[
\Phi_N(w)
=
a\left(
w+\sum_{n=0}^Na_nw^{-n}
\right).
\]
Since $\Phi\in C^{1+\alpha}(\TT)$, one has
\begin{equation}\label{T0}
\|\Phi_N-\Phi\|_{C^1(\TT)}
\longrightarrow0.
\end{equation}
For $N$ sufficiently large, $\Phi_N(\TT)$ bounds a simply connected domain $D_N$ containing every fixed compact subset of $D$. Define
\[
I_N(z)
:=
\fint_{\TT}
\frac{\overline{\Phi_N(\xi)}-\overline z}
{\Phi_N(\xi)-z}
\Phi_N'(\xi)\,d\xi.
\]
By dominated convergence,
\begin{equation}\label{TT1}
I_N(z)\longrightarrow I(z),
\qquad z\in D.
\end{equation}
Applying the analytic case to $\Phi_N$ gives
\begin{equation}\label{T3}
I_N(z)
=
-\overline z
+
a\sum_{n=0}^N
\overline{a_n}
F_{n,N}(z),
\end{equation}
where
\[
F_{n,N}(z)
=
\fint_{|\xi|=r}
\xi^n
\frac{\Phi_N'(\xi)}
{\Phi_N(\xi)-z}\,d\xi.
\]
For every fixed $n$,
\[
F_{n,N}(z)\longrightarrow F_n(z),
\qquad z\in D.
\]
Moreover,
\[
\sup_{z\in\overline{D_N}}
|F_{n,N}(z)|
\le
\widetilde A\log(1+n)+\widetilde B,
\]
with constants independent of $N$. Hence
\[
|a_n|
\sup_{z\in\overline{D_N}}
|F_{n,N}(z)|
\lesssim
(1+n)^{-1-\alpha}(1+\log(1+n)),
\]
and the right-hand side is summable. Dominated convergence therefore yields
\[
\lim_{N\to\infty}
\sum_{n=0}^N
\overline{a_n}
F_{n,N}(z)
=
\sum_{n\ge0}
\overline{a_n}
F_n(z).
\]
Passing to the limit in \eqref{T3} gives
\[
I(z)
=
-\overline z
+
a\sum_{n\ge0}
\overline{a_n}
F_n(z),
\qquad z\in D.
\]
Both sides extend continuously to $\overline D$, and thus
\[
I(z)
=
-\overline z
+
F(z),
\qquad z\in\overline D,
\]
where
\begin{equation}\label{master10}
F(z)
:=
a\sum_{n\ge0}
\overline{a_n}
F_n(z).
\end{equation}
Substituting this identity into \eqref{rotsq12}, we obtain
\[
\textnormal{Im}\left\{
\bigl((2\Omega-1)\overline{\Phi(w)}
+F(\Phi(w))\bigr)
w\Phi'(w)
\right\}
=
0.
\]
The estimates obtained above show that the series
defining $F$
converges uniformly on compact subsets of $D$. Since each $F_n$ is a polynomial, it follows that $F$ is holomorphic in $D$. As $D$ is simply connected, $F$ admits a holomorphic primitive. We therefore define
\begin{equation*}
H(z):=\int_{z_0}^{z}F(\xi)\,d\xi,
\end{equation*}
where the integral is independent of the chosen path joining $z_0$ to $z$. Consequently,
\[
H'(z)=F(z),\qquad z\in D.
\]
Writing $w=e^{it}$ and $z(t)=\Phi(e^{it})$, we have
\[
z'(t)=iw\Phi'(w).
\]
Therefore,
\[
\tfrac{d}{dt}
\left(
\tfrac{2\Omega-1}{2}
|\Phi(e^{it})|^2
+
\textnormal{Re} H(\Phi(e^{it}))
\right)
=
-
\textnormal{Im}\left\{
\bigl((2\Omega-1)\overline{\Phi(w)}
+F(\Phi(w))\bigr)
w\Phi'(w)
\right\}.
\]
Hence the boundary equation is equivalent to
\[
\tfrac{d}{dt}
\left(
\tfrac{2\Omega-1}{2}
|\Phi(e^{it})|^2
+
\textnormal{Re} H(\Phi(e^{it}))
\right)
=
0.
\]
Since $\partial D$ is connected, there exists a constant $\mu\in\mathbb R$ such that
\begin{align}\label{master11}
\tfrac{2\Omega-1}{2}
|\Phi(w)|^2
+
\textnormal{Re} H(\Phi(w))
=
\mu,
\qquad w\in\TT,
\end{align}
or equivalently,
\[
\tfrac{2\Omega-1}{2}|z|^2+\textnormal{Re} H(z)=\mu,
\qquad z\in\partial D.
\]
Conversely, differentiating this identity along the boundary immediately recovers \eqref{rotsq12}. This completes the proof.
\end{proof}

\section{Rigidity of ellipses}
The reformulation obtained in Theorem~\ref{Theom-form-Fab} provides a remarkably simple framework for studying rotating vortex patches. As a first application, we recover Kirchhoff's classical characterization of uniformly rotating ellipses by a short computation involving only the conformal mapping and the associated Faber polynomial. This illustrates the effectiveness of the new formulation and avoids the more involved computations based directly on the Cauchy transform.
\\
As a second application, we establish a rigidity result for conformal mappings with finite Laurent expansions. We prove that if the exterior conformal mapping of a rotating patch is a Laurent polynomial, then it necessarily has degree one. Consequently, the corresponding domain is an ellipse. This provides a complete characterization of Kirchhoff's ellipses within the class of rotating patches whose conformal mapping has finitely many non-zero coefficients. This latter result is obtained by Burbea \cite{Burbea81}.
\\
We summarize these results on the following statement.
\begin{mytheorem}{}{Kirchhoff}
The ellipse is rotating uniformly if and only if 
$$
\Omega=\frac{ab}{(a+b)^2}\cdot$$
Let ${\bf{1}}_D$ be a rotating patch of a simply connected bounded domain $D$ whose conformal mapping has a finite expansion, namely
$$
\Phi(w)=a\Big(w+\sum_{n=0}^N\frac{a_n}{w^n}\Big),\, w\in\TT$$
Then $D$ is an ellipse, in particular
$$
a_n=0,\forall n\in\{2,..,N\}.
$$
\end{mytheorem}
\begin{proof}
We begin with the Kirchhoff ellipse. Assume first that $a>b>0$ and set
\[
A:=\frac{a+b}{2},
\qquad
Q:=\frac{a-b}{a+b}.
\]
The associated exterior conformal mapping is
\[
\Phi(w)=A\left(w+\frac{Q}{w}\right),
\qquad |w|\ge1.
\]
Hence
\[
a_0=0,\qquad
a_1=Q,\qquad
a_n=0,\quad n\ge2.
\]
Since $F_1(z)=z/A$, it follows from \eqref{master10} that
\[
F(z)=AQF_1(z)=Qz,
\qquad
H(z)=\frac{Q}{2}z^2.
\]
The boundary identity \eqref{master11},writes as
\[
\tfrac{2\Omega-1}{2}|\Phi(w)|^2+\textnormal{Re} H(\Phi(w))=\mu,
\qquad w\in\TT.
\]
Since $\overline w=w^{-1}$ on $\TT$,
\[
|\Phi(w)|^2
=
A^2|w+Q\overline w|^2
=
A^2\bigl(1+Q^2+2Q\textnormal{Re}(w^2)\bigr),
\]
and
\[
\textnormal{Re} H(\Phi(w))
=
\tfrac{QA^2}{2}
\textnormal{Re}\!\left((w+Q\overline w)^2\right)
=
\tfrac{QA^2}{2}
\Bigl((1+Q^2)\textnormal{Re}(w^2)+2Q\Bigr).
\]
Therefore
\[
\tfrac{2\Omega-1}{2}|\Phi(w)|^2+\textnormal{Re} H(\Phi(w))
=
\tfrac{QA^2}{2}
\Bigl(2(2\Omega-1)+1+Q^2\Bigr)\textnormal{Re}(w^2)
+\textnormal{constant}.
\]
Since $Q\neq0$, the left-hand side is constant if and only if
\[
2(2\Omega-1)+1+Q^2=0,
\]
that is,
\[
\Omega
=
\frac{1-Q^2}{4}
=
\frac{ab}{(a+b)^2}.
\]
This proves the first assertion.
We now turn to the rigidity statement. Assume that the conformal mapping has the finite Laurent expansion
\[
\Phi(w)
=
a\left(
w+\sum_{n=0}^{N}\frac{a_n}{w^n}
\right),
\qquad |w|\ge1,
\]
with $a_N\neq0$. We shall prove that necessarily $N=1$.
Since $F_n$ is a polynomial of degree $n$, it follows from \eqref{master10} that $F$ is a polynomial of degree $N$. Consequently, its primitive $H$ is a polynomial of degree $N+1$, which we write as
\[
H(z)
=
\sum_{k=0}^{N+1}c_kz^k,
\qquad
c_{N+1}\neq0.
\]
The boundary identity \eqref{master11} becomes
\[
\tfrac{2\Omega-1}{2}|\Phi(w)|^2+\textnormal{Re} H(\Phi(w))
=
\mu,
\qquad
\forall w\in\TT.
\]
Expanding the left-hand side into Fourier modes, we compare the highest powers of $\overline w$. Since
\[
\Phi(w)
=
a\left(
w+a_0+\sum_{n=1}^{N}a_n\overline w^{\,n}
\right),
\]
the quadratic term contributes the Fourier mode
\[
(2\Omega-1)a^2
\textnormal{Re}\!\left(
a_N\overline w^{\,N+1}
\right).
\]
On the other hand, the leading term of $H$ produces the Fourier mode
\[
a^{N+1}
\textnormal{Re}\!\left(
c_{N+1}a_N^{\,N+1}
\overline w^{\,N(N+1)}
\right),
\]
whose coefficient is nonzero because $c_{N+1}\neq0$ and $a_N\neq0$.
Since the whole expression is constant on $\TT$, the highest Fourier modes must cancel. Therefore their orders coincide, namely
\[
N(N+1)=N+1.
\]
As $N\ge1$, we conclude that
\[
N=1.
\]
Hence
\[
\Phi(w)
=
a\left(
w+a_0+\frac{a_1}{w}
\right),
\]
which is precisely the conformal mapping of an ellipse. Therefore
\[
a_n=0,
\qquad n\ge2,
\]
and $D$ is an ellipse. This completes the proof.
\end{proof}
\section{Algebraic \(V\)-states}\label{sec-algebraic-vstates}

The rigidity result of the previous section admits a natural
interpretation in terms of algebraic curves. Indeed, a finite Laurent
expansion of the exterior conformal mapping automatically produces an
algebraic boundary. This observation suggests a broader rigidity
question in which the finite-expansion assumption is replaced by the
weaker requirement that the boundary itself be algebraic.
We say that a planar curve \(\Gamma\subset\mathbb R^2\) is
\emph{algebraic} if there exists a nonzero polynomial
\[
P\in\mathbb R[X,Y]
\]
such that
\[
\Gamma\subset\bigl\{(x,y)\in\mathbb R^2:\ P(x,y)=0\bigr\}.
\]
Accordingly, a simply connected \(V\)-state \(D\) will be called an
\emph{algebraic \(V\)-state} if its boundary \(\partial D\) is
algebraic.
\\
The following observation shows that the finite Laurent setting
considered above belongs to this class.

\begin{myproposition}{}{prop-finite-Laurent-algebraic}
Let \(D\subset\mathbb C\) be a bounded simply connected domain whose
exterior conformal mapping has a finite Laurent expansion
\[
\Phi(w)
=
a\left(
w+a_0+\frac{a_1}{w}+\cdots+\frac{a_N}{w^N}
\right),
\qquad |w|>1,
\]
where \(a>0\). Then \(\partial D\) is contained in a real algebraic
curve.
\end{myproposition}

\begin{proof}
The case \(N=0\) is immediate, since \(\partial D\) is then a circle.
We therefore assume, after reducing \(N\) if necessary, that
\(N\geq1\) and \(a_N\neq0\).
Let
\[
z=\Phi(w),\qquad |w|=1.
\]
Multiplying this identity by \(w^N\), we obtain
\[
A(w,z)=0,
\]
where
\[
A(w,z)
=
aw^{N+1}
+(aa_0-z)w^N
+aa_1w^{N-1}
+\cdots
+aa_N.
\]
Since \(|w|=1\), we have \(\overline w=w^{-1}\). Taking the complex
conjugate of \(z=\Phi(w)\), we find
\[
\overline z
=
a\left(
\frac1w+\overline{a_0}
+\overline{a_1}w+\cdots+\overline{a_N}w^N
\right).
\]
After multiplication by \(w\), this becomes
\[
B(w,\overline z)=0,
\]
where
\[
B(w,\zeta)
=
a\overline{a_N}w^{N+1}
+\cdots
+a\overline{a_1}w^2
+(a\overline{a_0}-\zeta)w+a.
\]
We now regard \(z\) and \(\zeta\) as independent complex variables and
eliminate \(w\). Let
\[
R(z,\zeta)
=
\mathop{\rm Res}\nolimits_w
\bigl(A(w,z),B(w,\zeta)\bigr)
\]
denote the resultant of the two polynomials with respect to \(w\).
Since the coefficients of \(A\) and \(B\) depend polynomially on
\(z\) and \(\zeta\), respectively, we have
\[
R\in\mathbb C[z,\zeta].
\]
Moreover, \(R\) is not identically zero. Indeed, if
\(w_1=w_1(z),\ldots,w_{N+1}=w_{N+1}(z)\) are the roots of \(A(\cdot,z)\), counted with
multiplicity, then the product formula for the resultant gives, up to
a nonzero multiplicative constant,
\[
R(z,\zeta)
=
\prod_{j=1}^{N+1}B(w_j,\zeta).
\]
Since the constant term of \(A\) is \(aa_N\neq0\), none of the roots
\(w_j\) vanishes. Each factor \(B(w_j,\zeta)\) is affine in \(\zeta\)
with coefficient \(-w_j\). Hence the coefficient of
\(\zeta^{N+1}\) in \(R\) is nonzero, and therefore
\[
R\not\equiv0.
\]
Now let \(z\in\partial D\). There exists \(w\in\mathbb T\) such that
\(z=\Phi(w)\). Consequently,
\[
A(w,z)=0,
\qquad
B(w,\overline z)=0.
\]
Thus the two polynomials $A(\cdot,z)$ and $B(\cdot,\overline{z})$ have a common root $w$, and the defining property
of the resultant yields
\[
R(z,\overline z)=0,
\qquad z\in\partial D.
\]
Thus
\[
R(x+iy,x-iy)=0,
\qquad (x,y)\in\partial D.
\]As the polynomial \((x,y)\mapsto R(x+iy,x-iy)\)  is not identically zero, then taking a nonzero real or imaginary
part therefore produces a non zero real polynomial $P$
such that
\[
P(x,y)=0,
\qquad \forall (x,y)\in\partial D.
\]
Hence \(\partial D\) is contained in a real algebraic curve. This ends the proof.
\end{proof}
We have therefore established the implication
\[
{
\text{finite Laurent exterior conformal mapping}
\quad\Longrightarrow\quad
\text{algebraic boundary}.
}
\]
The converse does not hold in general: an algebraic Jordan curve need
not have an exterior conformal mapping with only finitely many nonzero
Laurent coefficients. Thus the class of algebraic boundaries is
considerably broader than the finite Laurent class.
\\
Combining the preceding proposition with Burbea's rigidity result, we
know that every \(V\)-state whose exterior conformal mapping has a
finite Laurent expansion is necessarily a Kirchhoff ellipse. This
naturally suggests the following conjecture.

\medskip

\noindent
\textbf{Conjecture (rigidity of algebraic \(V\)-states).}
\emph{Let \(D\subset\mathbb C\) be a bounded simply connected
\(V\)-state for the two-dimensional Euler equations, with sufficiently
smooth boundary. If \(\partial D\) is algebraic, then \(D\) is an
ellipse.}

\medskip

\chapter{Some rigidity aspects of rotating patches}
{\it
This chapter is devoted to rigidity phenomena for uniformly rotating vortex patches and, in particular, to understanding how the angular velocity constrains the geometry of the patch. We first study the distinguished endpoint $\Omega=\frac12$, following the approach of Hmidi~$\cite{Hmidi15}.$ We present two complementary proofs, based respectively on the stream function and on the conformal mapping formulation developed in the previous chapter. Both approaches show that the only admissible simply connected patch is the disk.
We then consider the general rigidity problem for
\(
\Omega\notin\left(0,\frac12\right).
\)
Following the variational approach of G\'omez-Serrano, Park, Shi and Yao~$\cite{GPSY}$, we combine a variational identity with sharp extremal properties of the second moment and the torsional rigidity, the latter relying on Talenti's isoperimetric inequality. This yields the rigidity theorem asserting that every simply connected rotating vortex patch with angular velocity outside $\left(0,\frac12\right)$ is necessarily a Rankine vortex.
}

\section{The endpoint $\Omega=\frac12$}

The endpoint $\Omega=\frac12$ plays a distinguished role in the theory of rotating vortex patches. At this critical value, the boundary equation simplifies considerably and becomes sufficiently rigid to determine the geometry of the patch completely. The following theorem, due to Hmidi~\cite{Hmidi15}, shows that the only rotating patch with angular velocity $\Omega=\frac12$ is the Rankine vortex.

We present two different proofs. The first is new and exploits the conformal mapping formulation developed in the previous chapter, leading to a short argument based on Faber polynomials. The second follows the original approach of \cite{Hmidi15}, based on the stream function and the Cauchy transform.

\begin{mytheorem}{}{}
Let $D$ be a simply connected bounded domain with boundary of class $C^{1+\alpha}$, $\alpha\in(0,1)$, such that ${\bf1}_D$ rotates around the origin with angular velocity
\[
\Omega=\frac12.
\]
Then $D$ is necessarily a disk.
\end{mytheorem}

\begin{proof}
We give two different proofs.

\medskip
\noindent
\textbf{First proof: conformal mapping and Faber polynomials.}
\\
Applying Theorem~\ref{Theom-form-Fab}, we obtain
\[
\textnormal{Re} H(z)=C_0,
\qquad z\in\partial D,
\]
where
\[
F(z)=a\sum_{n\geqslant0}\overline{a_n}F_n(z),
\qquad
H'(z)=F(z).
\]
We have seen in the proof of Theorem~\ref{thm:Faber-Vstate} that as $\partial D$ is of class $C^{1+\alpha}$, the series defining $F$ converges uniformly on $\overline D$. Hence $F$ is holomorphic in $D$ and extends continuously to $\overline D$. Consequently, $H$ is holomorphic in $D$ and continuous on $\overline D$.
Therefore, the  function $\textnormal{Re} H$ is harmonic in $D$ and constant on $\partial D$. By the maximum principle,
\[
\textnormal{Re} H\equiv C_0
\qquad\hbox{in }D.
\]
The Cauchy--Riemann equations then imply that $H$ is constant throughout $D$. Therefore
\[
F(z)=H'(z)\equiv0.
\]
Using the triangular structure of the Faber polynomials,
\[
F_n(z)=a^{-n}z^n+\text{lower-order terms},
\]
we conclude recursively that
\[
a_n=0,
\qquad n\geqslant0.
\]
Hence
\[
\Phi(z)=az,
\]
which is the conformal mapping of a disk.

\medskip
\noindent
\textbf{Second proof: stream function.}
\\
According to the stream function formulation \eqref{rvortex3},
\[
\Phi_0(z):=\psi(z)-\frac14|z|^2=\mu,
\qquad z\in\partial D,
\]
for some constant $\mu$.
Since
\[
\Delta\psi={\bf1}_D,
\]
we obtain
\[
\Delta\Phi_0
=
{\bf1}_D-1.
\]
Hence $\Phi_0$ is harmonic inside $D$ and satisfies
\[
\begin{cases}
\Delta\Phi_0=0,&z\in D,\\
\Phi_0=\mu,&z\in\partial D.
\end{cases}
\]
The maximum principle therefore yields
\[
\Phi_0\equiv\mu
\qquad\hbox{in }D.
\]
Equivalently,
\[
\psi(z)=\mu+\frac14|z|^2,
\qquad z\in D.
\]
Differentiating gives
\[
\mathcal C(\chi_D)(z)
=
4\partial_z\psi(z)
=
\overline z,
\qquad z\in D.
\]
Since both sides are continuous up to $\partial D$,
\[
\mathcal C(\chi_D)(z)=\overline z,
\qquad z\in\overline D.
\]
Now define
\[
G(z):=z\,\mathcal C(\chi_D)(z),
\qquad
z\in\CC\setminus\overline D.
\]
The function $G$ is holomorphic in $\CC\setminus\overline D$, continuous up to the boundary, and satisfies
\[
G(z)=|z|^2,
\qquad z\in\partial D.
\]
Moreover,
\[
\lim_{z\to\infty}G(z)=\frac{|D|}{\pi}\in\mathbb R.
\]
Let
\[
u(z):=\textnormal{Im} G(z).
\]
Then $u$ is harmonic in $\CC\setminus\overline D$, vanishes on
$\partial D$, and tends to zero at infinity. By the maximum principle,
\[
u\equiv0
\qquad\hbox{in }\CC\setminus\overline D.
\]
Thus the holomorphic function $G$ takes only real values and hence it  must be constant. Therefore,
\[
|z|^2=\mathrm{const},
\qquad z\in\partial D,
\]
which shows that $\partial D$ is a circle. This completes the proof.
\end{proof}

\section{Talenti's isoperimetric inequality}

Talenti's isoperimetric principle is one of the fundamental results in the theory of elliptic partial differential equations. It compares the solutions of elliptic boundary value problems on an arbitrary domain with those on a disk of the same measure through Schwarz symmetrization. As a consequence, it yields sharp geometric estimates for various quantities associated with elliptic equations, such as the torsional rigidity, and identifies the disk as the unique extremal domain in many optimization problems.

For our purposes, we only need the particular case concerning the torsion function. Before stating Talenti's theorem, we establish an elementary geometric inequality showing that the disk uniquely minimizes the second moment among all planar domains of prescribed area.
\begin{myproposition}{}{Talenti-1}
Let $f:[0,\infty)\to[0,\infty)$ be a continuous and strictly increasing function. Let $D\subset\mathbb R^2$ be a bounded measurable set, and let $B$ be the disk centered at the origin such that
\[
|D|=|B|.
\]
Then
\[
\int_D f(|x|)\,dx\geqslant \int_B f(|x|)\,dx.
\]
Moreover, equality holds if and only if
\[
|D\Delta B|=0,
\]
where
\[
D\Delta B=(D\setminus B)\cup(B\setminus D)
\]
denotes the symmetric difference.
\end{myproposition}
\begin{proof}
Let $r_B$ denote the radius of the disk $B$. Since $|D|=|B|$, we have
\[
|D\setminus B|=|B\setminus D|.
\]
Moreover,
\[
|x|\leqslant r_B,\qquad x\in B\setminus D,
\]
and therefore, by the monotonicity of $f$,
\[
f(|x|)\leqslant f(r_B),\qquad x\in B\setminus D.
\]
Hence
\begin{align*}
\int_D f(|x|)\,dx-\int_B f(|x|)\,dx
&=\int_{D\setminus B}f(|x|)\,dx
-\int_{B\setminus D}f(|x|)\,dx\\
&\geqslant
\int_{D\setminus B}f(|x|)\,dx
-f(r_B)|B\setminus D|\\
&=
\int_{D\setminus B}\bigl(f(|x|)-f(r_B)\bigr)\,dx\\
&=
\int_{D\setminus\overline B}\bigl(f(|x|)-f(r_B)\bigr)\,dx\\
&\geqslant0.
\end{align*}
This proves the desired inequality.
\\
Assume now that equality holds. Then
\[
\int_{D\setminus\overline B}
\bigl(f(|x|)-f(r_B)\bigr)\,dx=0.
\]
Since $f$ is strictly increasing,
\[
f(|x|)-f(r_B)>0,
\qquad
x\in D\setminus\overline B.
\]
Hence the preceding integral can vanish only if
\[
|D\setminus B|=0.
\]
Using again the identity
\[
|D\setminus B|=|B\setminus D|,
\]
we also obtain
\[
|B\setminus D|=0.
\]
Therefore
\[
|D\Delta B|=0,
\]
which completes the proof.
\end{proof}
Let $D$ be a bounded open domain and  $p$ be the solution to the elliptic problem
\begin{equation}
\begin{cases}
\Delta p(x) = -2 & \text{in } D, \\
p(x) = 0 & \text{on } \partial D.
\end{cases}
\label{2.5}
\end{equation}
The next ingredient is a sharp isoperimetric estimate due to Talenti \cite{Talenti}. It provides a remarkable link between the solution $p$ of \eqref{2.5} and the geometry of the underlying domain. More precisely, the integral of $p$ over $D$ is maximized, among domains of prescribed area, by the disk. The rigidity of the equality case will be particularly important in our analysis: attaining the optimal bound forces the domain to be circular. Thus, this estimate will allow us to convert an analytic identity arising from the rotating patch equation into a geometric rigidity statement as we shall see in the following section.
\begin{mytheorem}{Talenti}{}
    Let $\alpha\in(0,1)$ and $D \subset \mathbb{R}^2$ be a bounded open domain   with $C^{1+\alpha}$ boundary, and let $p$ be defined as in \eqref{2.5}. 
Then we have
\[
\int_D p(x)\,dx \leqslant \frac{1}{4\pi} |D|^2,
\]
with equality if and only if $D$ is an open disk.
\end{mytheorem}

\section{Rigidity in the general case $\Omega\notin(0,\frac12)$}

We now turn to the general rigidity theorem, originally established by G\'omez-Serrano,  Park,  Shi and Yao~\cite{GPSY}. In contrast to the critical case $\Omega=\frac12$, where the rigidity follows from elementary harmonic arguments, the proof for $\Omega\notin(0,\frac12)$ relies on a subtle variational approach. The main ingredients are the isoperimetric inequalities established in the previous section together with a suitable integral identity satisfied by every rotating patch. Their combination shows that the only simply connected rotating patch with angular velocity $\Omega\notin(0,\frac12)$ is the Rankine vortex. More precisely, we intend to prove the following result.

\begin{mytheorem}{}{Rigidity-Gom-Yap}
    Let $D$ be a simply connected bounded domain with $C^{1+\alpha}$ boundary, where $\alpha\in(0,1)$. 
If $D$ is a rotating patch solution with angular velocity $\Omega$, where 
$\Omega\notin (0,\frac12)$, then $D$ must be a disk.
\end{mytheorem}
\begin{proof}
    Define
    $$
    \Phi_\Omega(x):=\psi(x)-\tfrac12\Omega |x|^2\quad\hbox{with}\quad \psi(x)=\tfrac{1}{2\pi}\int_{D}\log(|x-y|)\, dy$$
    The domain is rotating with angular velocity $\Omega$ is determined by the constraint
    $$\Phi_\Omega(x)=\mu,\forall x\in\partial D,$$
    with $\mu$ being a real constant. 
    Let $v$ be a free divergence vector field such that $v\in C^1(D)\cap C(\overline{D}).$ We consider the flux
    $$
    \mathcal{I}:=-\int_{D}v(x)\cdot\nabla\Phi_\Omega(x)\, dx.$$
    We will use two different ways to compute $\mathcal{I}$, and show that if $D$ is not a disk, then the two approaches lead to a contradiction for $\Omega \leqslant 0$ or $\Omega \geqslant \tfrac{1}{2}$.
\\
On the one hand, since $\Phi_\Omega$ is constant on $\partial D$ , the divergence theorem yields 
\begin{equation}\label{2.3}
\mathcal{I} 
= -\mu \int_{\partial D} v \cdot n \, d\sigma + \int_D (\nabla \cdot v)\, \Phi_\Omega \, dx
= -\mu \int_D \nabla \cdot v \, dx + \int_D (\nabla \cdot v)\, \Phi_\Omega \, dx 
= 0.
\end{equation}
On the other hand, we fix $v$ as follows,
\[
v(x) := -\nabla \varphi(x) \quad \text{in } D,
\]
where
\[
\varphi(x) := \tfrac{|x|^2}{2} + p(x) \quad \text{in } D.
\]
with $p$ being the solution to the Poisson equation \eqref{2.5}.\\
\\
In the proof, we show that with this choice of $v$, the quantity $I$ can be computed in a second way, yielding 
$\mathcal{I} > 0$ for $\Omega \leqslant 0$ and $\mathcal{I} < 0$ for $\Omega \geqslant \tfrac{1}{2}$. In both cases, this contradicts the 
identity $\mathcal{I} = 0$ obtained in \eqref{2.3}.
\\
On the other hand, we compute $\mathcal{I}$ as follows:
\begin{equation}
\mathcal{I} 
= - \int_D v \cdot \nabla \Phi_\Omega \, dx
= \int_D x \cdot \nabla \Phi_\Omega \, dx 
+ \int_D \nabla p \cdot \nabla \Phi_\Omega \, dx
=: \mathcal{I} _1 + \mathcal{I} _2.
\label{2.6}
\end{equation}
For $\mathcal{I} _1$, we have
\begin{align}
\mathcal{I} _1 
\nonumber&= \frac{1}{2\pi}\int_D x \cdot \nabla \big( \mathbf{1}_D * \log|\cdot| \big) \, dx 
- \Omega\int_D x \cdot  x \, dx \\
\nonumber&= \frac{1}{2\pi} \int_D \int_D x \cdot \frac{x - y}{|x - y|^2} \, dy\, dx 
- \Omega \int_D |x|^2 \, dx \\
\nonumber&= \frac{1}{4\pi} \int_D \int_D \frac{x \cdot (x - y) + y \cdot (y - x)}{|x - y|^2} \, dy\, dx 
- \Omega \int_D |x|^2 \, dx \\
&= \frac{1}{4\pi} |D|^2 - \Omega \int_D |x|^2 \, dx\label{2.7}.
\end{align}
To compute $\mathcal{I}_2$, we use the divergence theorem (together with the fact that $p = 0$ on $\partial D$) to obtain
\begin{equation}\label{2.8}
\mathcal{I}_2 
= - \int_D p \, \Delta \Phi_\Omega \, dx 
= (2\Omega - 1)\int_D p \, dx.
\end{equation}
Plugging \eqref{2.7} and \eqref{2.8} into \eqref{2.6}, we obtain
\begin{align}\label{Ineq-09}
\mathcal{I} 
= \frac{1}{4\pi} |D|^2 
- \Omega \int_D |x|^2 \, dx 
+ (2\Omega - 1)\int_D p \, dx.
\end{align}
Towards a contradiction, assume that $D \neq B$. Among all sets with the same area as $D$, the disk $B$ is uniquely characterized as the minimizer of the second moment. To see this, denote by $r_B$ the radius of $B$. Then,
\begin{align*}
\int_D |x|^2 \, dx - \int_B |x|^2 \, dx
&= \int_{D \setminus B} |x|^2 \, dx - \int_{B \setminus D} |x|^2 \, dx \\
&\geqslant \int_{D \setminus B} |x|^2 \, dx - \int_{B \setminus D} r_B^2 \, dx \\
&\geqslant \int_{D \setminus B} r_B^2 \, dx - \int_{B \setminus D} r_B^2 \, dx \\
&= 0,
\end{align*}
where the last equality follows from $|D \setminus B| = |B \setminus D|$, a consequence of $|D| = |B|$. Moreover, the inequality is strict whenever $D \neq B$. 
Thus, if $D \neq B$, we obtain
\[
\int_D |x|^2 \, dx > \int_B |x|^2 \, dx = \frac{1}{2\pi} |D|^2,
\]
where the last identity follows from a direct computation.\\
Plugging this into \eqref{Ineq-09} yields the following inequality for $\Omega \in \left[\tfrac{1}{2},\infty\right)$:
\[
\mathcal{I} 
\leqslant \frac{1}{4\pi}|D|^2 
- \Omega \frac{1}{2\pi}|D|^2 
+ (2\Omega - 1)\int_D p \, dx
= (1 - 2\Omega)\left( \frac{1}{4\pi}|D|^2 - \int_D p \, dx \right).
\]

On the other hand, for $\Omega \in (-\infty,0]$, we have
\[
\mathcal{I}
\geqslant \frac{1}{4\pi}|D|^2 
- \Omega \frac{1}{2\pi}|D|^2 
+ (2\Omega - 1)\int_D p \, dx
= (1 - 2\Omega)\left( \frac{1}{4\pi}|D|^2 - \int_D p \, dx \right).
\]
In both cases, we obtain a contradiction with $\mathcal{I} = 0$, and the proof is therefore complete.
\end{proof}
\section{Conclusion}

The rigidity results developed in this chapter reveal two complementary
mechanisms leading to radial symmetry. The first is analytic and exploits
the special structure of the $V$-state equation at distinguished values
of the angular velocity. At $\Omega=\frac12$, rigidity follows either
from the maximum principle and harmonicity of the stream function or
from the conformal formulation and the algebraic structure of Faber
polynomials.

The second mechanism is variational and relies on global geometric
quantities rather than on an explicit analysis of the boundary equation.
The general rigidity theorem combines identities satisfied by rotating
patches with sharp geometric inequalities, notably Talenti's
rearrangement inequality, whose equality case ultimately forces the
domain to be a disk.

These two approaches illustrate a useful dichotomy: rigidity may arise
either from the analytic structure of the governing equation or from
extremal properties of geometric functionals. Despite their different
nature, both mechanisms single out the Rankine vortex as the unique
admissible configuration in the corresponding rigidity regimes.

\chapter{Burbea's result revisited}
{\it
This chapter is devoted to a complete and rigorous proof of Burbea's celebrated theorem $\cite{Burbea82},$ which establishes the existence of nontrivial rotating vortex patches  with prescribed $m$-fold symmetry bifurcating from the Rankine vortex. Burbea's pioneering work laid the foundations of the modern theory of rotating vortex patches by introducing the conformal mapping formulation and the bifurcation approach. However, several delicate analytical issues remained unresolved, largely because the functional framework based on Hardy spaces was not well suited to the nonlinear problem. The primary objective of this chapter is to revisit this fundamental result in the natural setting of Hölder spaces and to present a complete, self-contained proof based on the Crandall--Rabinowitz bifurcation theorem, following the approach developed in $\cite{HMV13}.$
}
\section{General overview}

The purpose of this chapter is to present a complete and rigorous proof of Burbea's celebrated existence theorem for rotating vortex patches, or \(V\)-states, of the two-dimensional Euler equations. As discussed in the previous chapters, a vortex patch is a weak solution whose vorticity is the characteristic function of a bounded domain,
\(
\omega=\chi_D.
\)
The study of rotating vortex patches is motivated by the search for coherent structures whose geometry is preserved by the Euler flow up to a rigid rotation. The first explicit examples of such structures, discussed in detail in the previous chapters, are the Rankine vortices and the Kirchhoff ellipses. Among these, the unit disk occupies a distinguished position. It gives rise to the classical Rankine vortex and, as we shall see, provides the natural trivial configuration from which non elliptic rotating patches bifurcate. In this sense, the disk serves as the reference state for the local bifurcation theory developed in this chapter.

The first compelling evidence for the existence of rotating solutions beyond 
the classical family of Kirchhoff ellipses came from the numerical experiments 
of Deem and Zabusky \cite{DZ78}, which revealed remarkable families of rotating 
patches exhibiting prescribed \(m\)-fold symmetry. Motivated by these 
observations, Burbea \cite{Burbea82} established the first rigorous existence 
result by combining a conformal mapping formulation of the free-boundary 
problem with local bifurcation techniques inspired by the work of Keller and 
Langford \cite{Keller}.

Burbea's pioneering work introduced the fundamental ideas that underpin the 
modern theory of \(V\)-states. However, several delicate analytical aspects of 
the argument were left incomplete. In particular, the Hardy-space framework 
adopted there is not sufficiently well adapted to the regularity properties of 
the nonlinear boundary operator, and some of the key steps required for a 
complete application of the bifurcation theorem were only sketched.

The primary objective of this chapter is therefore to revisit Burbea's theorem in the natural setting of Hölder spaces and present a complete, self-contained proof based on the Crandall--Rabinowitz bifurcation theorem \cite{C-R71}, following the approach developed in \cite{HMV13}.

The starting point of Burbea's approach is the contour dynamics
equation derived in the previous chapter. We have seen in
\eqref{rotsq12} that a simply connected vortex patch \(\omega=\chi_D\) rotates around
the origin with angular velocity \(\Omega\) if and only if
\begin{align}\label{vort-patch-rot}
F(\Omega,f)=0\,,
\end{align}
where 
\begin{align}\label{rotsq14}
F(\Omega,f)(w)
:=
\operatorname{Im}
\Bigg[
\Bigg(
2\Omega\,\overline{\Phi(w)}
+
\fint_{\mathbb T}
\frac{\overline{\Phi(\xi)}-\overline{\Phi(w)}}
{\Phi(\xi)-\Phi(w)}
\,\Phi'(\xi)\,d\xi
\Bigg)
w\Phi'(w)
\Bigg],\quad\forall w\in\mathbb{T},
\end{align}
with
\[
\Phi:\mathbb C\setminus\overline{\mathbb D}
\longrightarrow
\mathbb C\setminus\overline D
\]
be the conformal mapping normalized at infinity by
\[
\Phi(z)=z+\sum_{n=0}^{\infty}a_nz^{-n}:=z+f(z),
\qquad |z|>1.
\]
Assuming that the boundary \(\partial D\) is rectifiable,
Carathéodory's theorem guarantees that \(\Phi\) extends continuously
and injectively up to the boundary. Consequently, the boundary of the
patch can be parametrized by
\begin{align}\label{conf-bound}
\Phi(w)
=
w+\sum_{n=0}^{\infty}a_n\overline w^{\,n}
=
w+f(w),
\qquad w\in\mathbb T.
\end{align}
Therefore, the unknown domain is encoded through the perturbation
\(f\). When
\[
\|f\|_{\mathrm{Lip}}<1,
\]
the mapping \(\Phi\) remains conformal and defines a Jordan domain.
Throughout this chapter we shall restrict ourselves to perturbations
satisfying this condition.
\\

Our objective is to solve \eqref{vort-patch-rot} in a neighborhood of the unit
disk, namely, for small perturbations \(f\) of the trivial solution \(f=0\).
The guiding principle is to regard the disk as a simple reference configuration
from which more intricate rotating patches may emerge through bifurcation.
To implement this strategy, we formulate the problem in terms of a nonlinear
operator \(F\) acting between suitable Banach spaces and apply the
Crandall--Rabinowitz bifurcation theorem. The analysis proceeds in several
steps. We first establish the required regularity properties of \(F\). We then
compute its linearization at the disk and analyze the corresponding spectral
structure. Finally, we verify the kernel, range, and transversality conditions
of the Crandall--Rabinowitz theorem. This will yield Burbea's existence result
for nontrivial rotating vortex patches, which is presented in
Section~\ref{sec-Burbea}.

\section{Basics on bifurcation theory and the Crandall--Rabinowitz theorem}

We briefly recall the local bifurcation principle that will be used in
the construction of Burbea's \(V\)-states. Consider a nonlinear equation
\[
F(\Omega,f)=0,
\qquad
F:\mathbb R\times X\to Y,
\]
between two Banach spaces, and suppose that
\[
F(\Omega,0)=0, \,\forall \Omega\in \mathbb{R}\,.
\]
 Thus,
\((\Omega,0)\) forms a branch of trivial solutions. In our setting,
\(f=0\) corresponds to the unit disk and \(\Omega\) is the angular
velocity.
The starting point of the bifurcation analysis is the linearized
operator
\[
L_\Omega:=\partial_fF(\Omega,0).
\]
As long as \(L_\Omega\) is invertible, the Implicit Function Theorem
precludes the existence of nontrivial solutions sufficiently close to
the trivial branch. Therefore, bifurcation can only occur at values of
\(\Omega\) for which \(L_\Omega\) loses invertibility.
The Crandall--Rabinowitz theorem provides the precise mechanism for
turning this linear degeneracy into a nonlinear branch of solutions.
Roughly speaking, if the linearized operator is Fredholm of index zero
with a one-dimensional kernel and if the degeneracy is crossed
transversally as the parameter varies, then a unique local curve of
nontrivial solutions bifurcates from the trivial branch. These are
exactly the properties that we shall verify for the rotating vortex
patch equation. In particular, each relevant Fourier mode will give
rise to a branch of \(m\)-fold symmetric \(V\)-states.

For completeness, we recall the version of the
Crandall--Rabinowitz theorem that will be used below
\cite{C-R71}. For a linear operator \(L\), we denote by \(N(L)\)
and \(R(L)\) its kernel and range, respectively.

\begin{mytheorem}{}{Crandall-Rab}
Let $X$, $Y$ be two Banach spaces, let $V$ be a neighborhood of $0$ in
$X$, and let
\[
F:\mathbb{R}\times V\to Y
\]
satisfy the following assumptions:

\begin{enumerate}
\item[(a)] $F(\Omega,0)=0$ for every $\Omega\in\mathbb{R}$.

\item[(b)] The partial derivatives
\[
\partial_\Omega F,\qquad \partial_xF,\qquad \partial_\Omega\partial_xF
\]
exist and are continuous.

\item[(c)] The spaces
\[
N(\partial_xF(0,0))
\qquad\text{and}\qquad
Y/R(\partial_xF(0,0))
\]
are one-dimensional.

\item[(d)] If
\[
N(\partial_xF(0,0))
=
\operatorname{span}\{x_0\},
\]
then
\[
\partial_\Omega\partial_xF(0,0)x_0
\notin
R(\partial_xF(0,0)).
\]
\end{enumerate}

Let $Z$ be any complement of $N(\partial_xF(0,0))$ in $X$. Then there exist a
neighborhood $U$ of $(0,0)$ in $\mathbb R\times X$, a positive number
$a$, and continuous functions
\[
\varphi:(-a,a)\to\mathbb R,
\qquad
\psi:(-a,a)\to Z,
\]
such that
\[
\varphi(0)=0,
\qquad
\psi(0)=0,
\]
and
\[
F^{-1}(0)\cap U
=
\Big\{
(\varphi(s),
s x_0+s\psi(s))
:\ |s|< a
\Big\}
\cup
\Big\{
(\Omega,0):(\Omega,0)\in U
\Big\}.
\]
Thus, a unique curve of nontrivial solutions bifurcates from
the trivial branch at $(0,0)$. 
\end{mytheorem}



\section{Functional tools}

In this section, we introduce the functional framework that will be used
throughout the proof of Burbea's theorem. We begin by recalling some basic
facts about Fourier series and H\"older spaces on the unit circle.
We identify the one-dimensional torus with the unit circle in the complex
plane,
\[
\mathbb T:=\bigl\{w\in\mathbb C:\ |w|=1\bigr\}.
\]
Every function \(f\in L^2(\mathbb T)\) admits the Fourier expansion
\[
f(w)=\sum_{n\in\mathbb Z}a_n w^n,
\qquad \text{for a.e. } w\in\mathbb T,
\]
where \((a_n)_{n\in\mathbb Z}\) denotes the sequence of Fourier
coefficients of \(f\) given by
\[
a_n
=
\fint_{\mathbb T}
f(w)\overline{w}^{\,n+1}\,dw,
\qquad n\in\mathbb Z.
\]
Let $\alpha\in(0,1)$ and $f:\mathbb T\to\mathbb C$. We define the Hölder seminorm by
\[
[f]_{C^\alpha(\mathbb T)}
=
\sup_{w_1\neq w_2}
\frac{|f(w_1)-f(w_2)|}
{|w_1-w_2|^\alpha},
\]
The H\"older space \(C^\alpha(\mathbb T)\) consists of all continuous
functions \(f:\mathbb T\to\mathbb C\) for which this seminorm is finite,
and is endowed with the norm
\[
\|f\|_{C^\alpha(\mathbb T)}
=
\|f\|_{L^\infty(\mathbb T)}
+
[f]_{C^\alpha(\mathbb T)}.
\]
It is known that for $C^\alpha(\mathbb T)$ is an algebra with the estimate 
\begin{align}\label{algebra-cal}
\|fg\|_{C^\alpha(\mathbb T)}\lesssim \|f\|_{C^\alpha(\mathbb T)}\|g\|_{L^\infty(\mathbb T)}+\|f\|_{L^\infty(\mathbb T)}\|g\|_{C^\alpha(\mathbb T)}
\end{align}
It is often convenient to identify $f$ with the associated $2\pi$-periodic function
\[
g(\theta)=f(e^{i\theta}),\qquad \theta\in\mathbb R,
\]
and to use the equivalent seminorm
\[
[g]_{C^\alpha_{\rm per}}
=
\sup_{\theta\neq\eta}
\frac{|g(\theta)-g(\eta)|}
{\left|\sin\!\left(\frac{\theta-\eta}{2}\right)\right|^\alpha}.
\]
Indeed, since
\[
|e^{i\theta}-e^{i\eta}|
=
2\left|\sin\!\left(\frac{\theta-\eta}{2}\right)\right|,
\]
we have
\[
[g]_{C^\alpha_{\rm per}}
=
2^\alpha [f]_{C^\alpha(\mathbb T)},
\]
and therefore the two seminorms, as well as the corresponding norms, are equivalent. \\
For a continuously differentiable function on $\mathbb T$, we define the tangential derivative by
\[
\partial_\tau f(w):=
\left.\frac{d}{dt}f(we^{it})\right|_{t=0},
\qquad w\in\mathbb T.
\]
It is immediate that 
$$
\partial_\tau f(w)=g^\prime(\theta), w=e^{ i\theta}$$
If $f$ is the boundary trace of a holomorphic function in $\mathbb D$, then
\[
\partial_\tau f(w)=iwf'(w),
\]
where $f'$ denotes the complex derivative.
The space $C^{1+\alpha}(\mathbb T)$ consists of all functions $f\in C^1(\mathbb T)$ such that $\partial_\tau f\in C^\alpha(\mathbb T)$, equipped with the norm
\[
\|f\|_{C^{1+\alpha}(\mathbb T)}
:=
\|f\|_{L^\infty(\mathbb T)}
+\|\partial_\tau f\|_{L^\infty(\mathbb T)}
+[\partial_\tau f]_{C^\alpha(\mathbb T)}.
\]
This norm is equivalent to
$$
\|g\|_{L^\infty(\mathbb{R})}
+\|g^\prime\|_{L^\infty(\mathbb R)}
+[g^\prime]_{{C^\alpha_{\rm per}}}
$$
Throughout the paper, we shall freely use either characterization according to convenience.
\\
We now introduce the Banach space that will be used to parametrize the
boundaries of the vortex patches. Define
\begin{align}\label{space-X}
X:=\Bigl\{
f\in C^{1+\alpha}(\mathbb T):
f(w)=\sum_{n\geqslant0}a_n  w^{-\,n},
\quad a_n\in\mathbb R
\Bigr\}.
\end{align}
Note that  the Fourier coefficients of functions in \(X\) are  vanishing for 
non-positive frequencies. This choice is naturally
motivated by the structure of the perturbation term appearing in the
conformal parametrization \eqref{conf-bound}. Moreover, functions in
\(X\) satisfy the symmetry relation
\[
f(\overline w)=\overline{f(w)},
\qquad w\in\mathbb T,
\]
which is equivalent to the fact that all their Fourier coefficients
are real. Consequently, the corresponding vortex patches are
symmetric with respect to the real axis. As we shall discuss later, the range of the nonlinear functional $F$ will be embedded in  the space
\begin{align}\label{space-Y}
Y
=
\Bigl\{
g\in C^\alpha(\mathbb T):
g(w)
=
\sum_{n\geqslant 1}g_n\,\operatorname{Im}(w^n),
\quad g_n\in\mathbb R
\Bigr\}.
\end{align}
Observe that every element of \(Y\) is a real-valued odd function on
the unit circle. This space naturally arises from the imaginary part
appearing in the rotating patch equation \eqref{rotsq14}.
\\
The spaces \(X\) and \(Y\) will serve as the domain and codomain of
the nonlinear functional \(F\) associated with the rotating vortex
patch equation. They are specifically designed so that the symmetry of
the problem is preserved and the hypotheses of the
Crandall--Rabinowitz theorem can be verified.

\section{Burbea's result}\label{sec-Burbea}

The numerical experiments of Deem and Zabusky \cite{DZ78}
revealed families of noncircular rotating vortex patches exhibiting
\(m\)-fold symmetry, suggesting that the classical Rankine and Kirchhoff
solutions belong to a much richer class of coherent structures.

\begin{figure}[H]
\centering
\includegraphics[width=3cm]{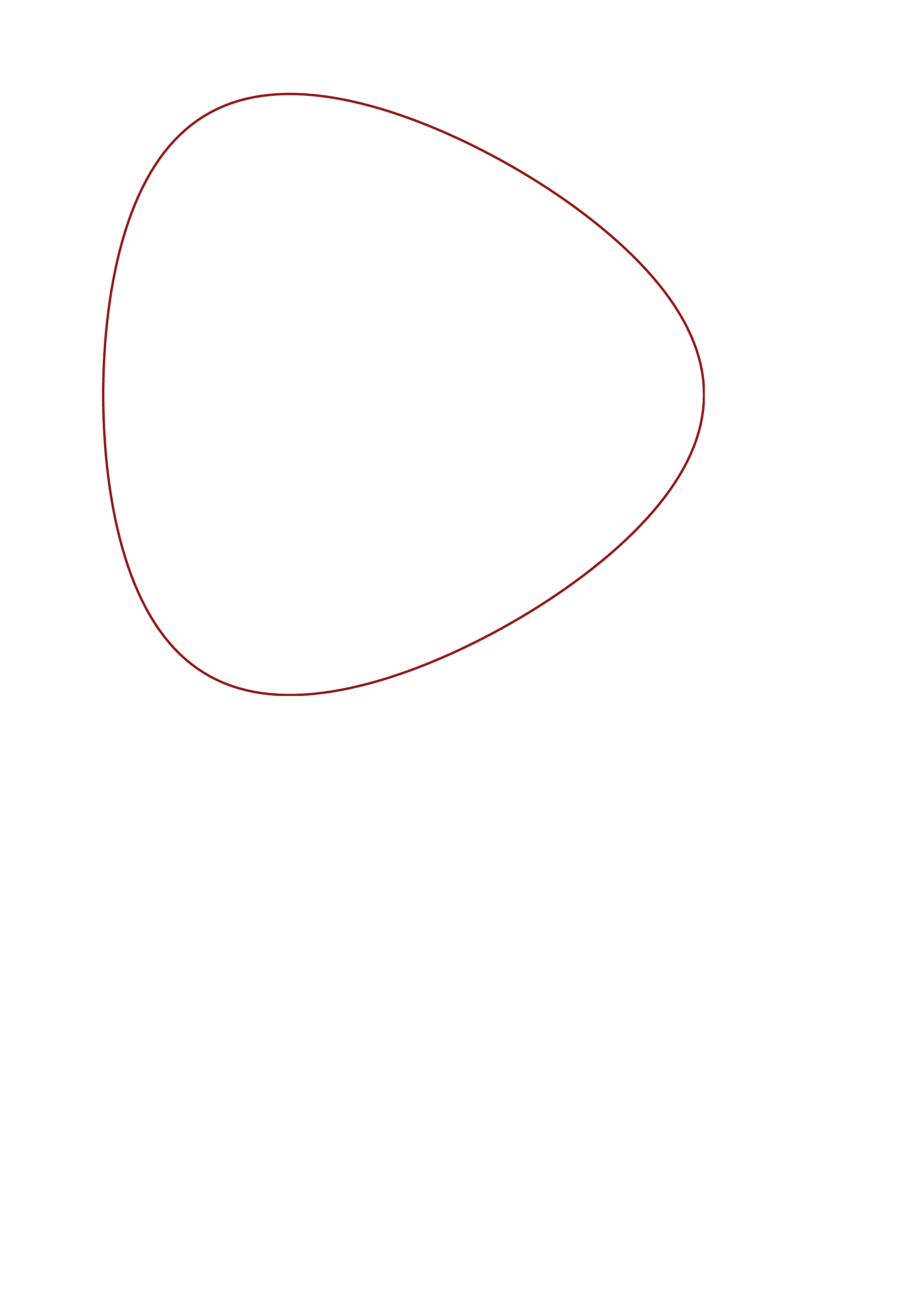}\hspace{0.2cm}
\includegraphics[width=3cm]{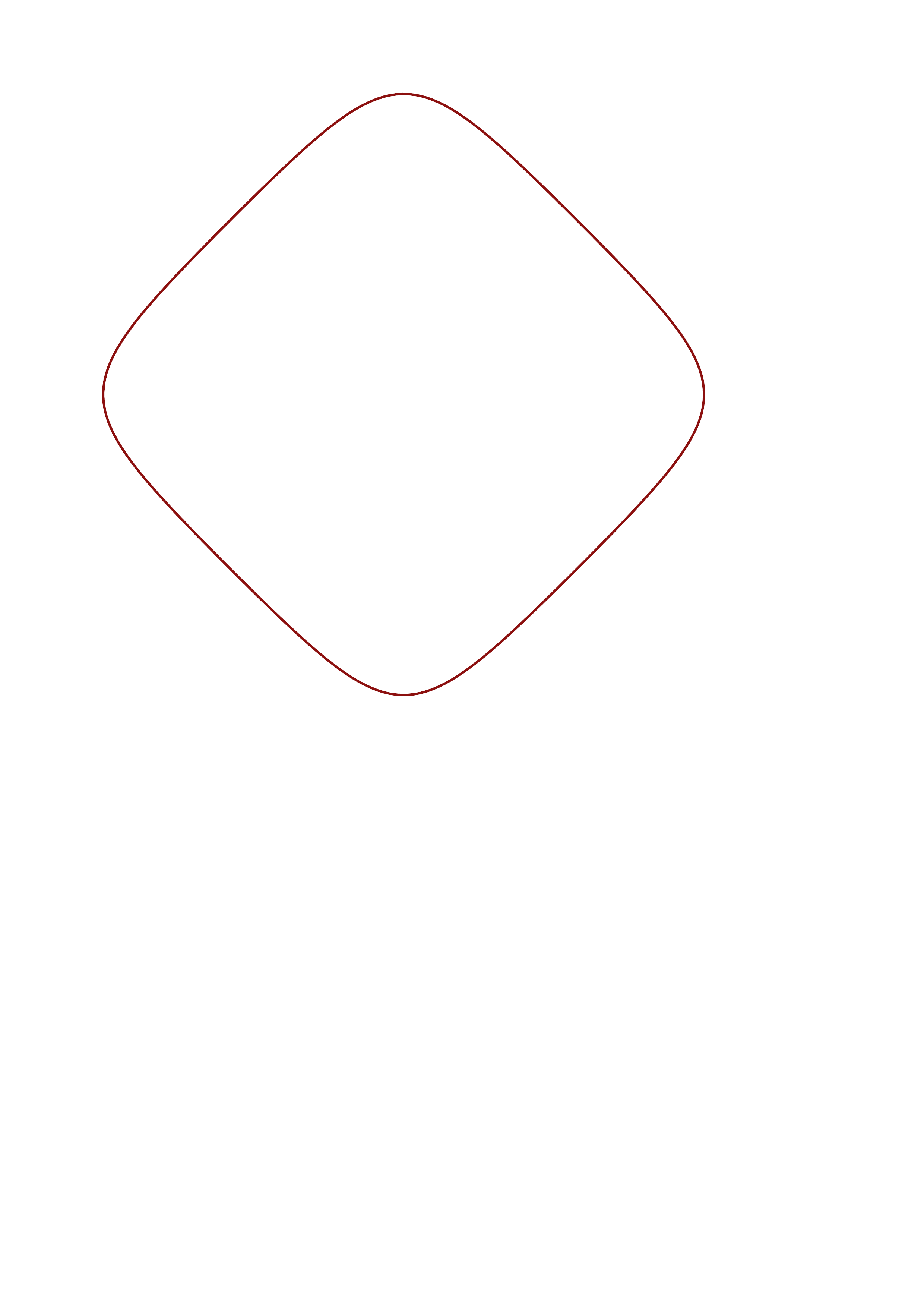}\hspace{0.2cm}
\includegraphics[width=3cm]{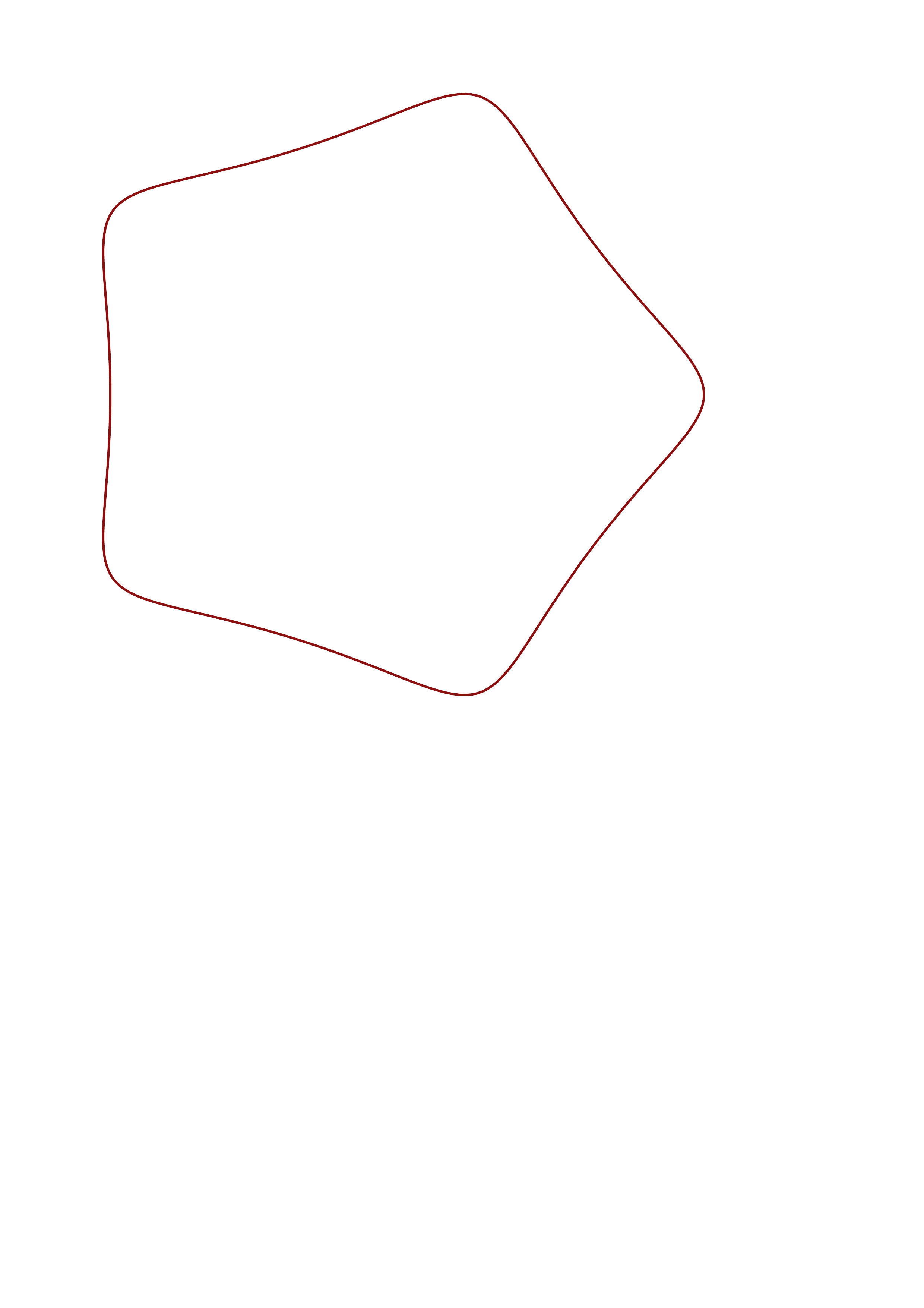}
\caption{Numerical \(V\)-states with \(m=3,4,5\)  observed
by Deem and Zabusky \cite{DZ78}.}
\label{fig:Deem-Zabusky}
\end{figure}
Burbea's theorem proved in \cite{Burbea82}  can be stated as follows.
\begin{mytheorem}{}{Burbea}
Let $\alpha\in(0,1)$ and let $m\ge 2$ be an integer. Then there exists
a family of $m$-fold symmetric rotating vortex patches
$
(V_m)_{m\ge2}
$
for the equation \eqref{vort-patch-rot}-\eqref{rotsq14}. Moreover, for each $m$ the branch
\(V_m\) bifurcates from the trivial Rankine vortex
at the angular velocity
\[
\Omega_m=\frac{m-1}{2m}.
\]

In addition, the boundary of the corresponding vortex patch belongs to
the H\"older class
\[
C^{1+\alpha}(\mathbb T).
\]
In fact, the boundary is analytic.
\end{mytheorem}




The bifurcation diagram associated with the theorem is represented in
Figure \ref{fig:bifurcation}. Each branch corresponds to a family of
rotating patches sharing the same symmetry group. Near the
bifurcation point, the solutions are small perturbations of the disk
and are essentially governed by the Fourier mode
\[
\cos(m\theta).
\]
As one moves away from the bifurcation point along the global branch,
the geometry of the patches becomes increasingly nonlinear, gradually
developing the characteristic shapes observed numerically by Deem and
Zabusky. At the limiting configuration, the boundary loses smoothness
and develops corners.

\begin{figure}[H]
\centering
\includegraphics[width=6cm]{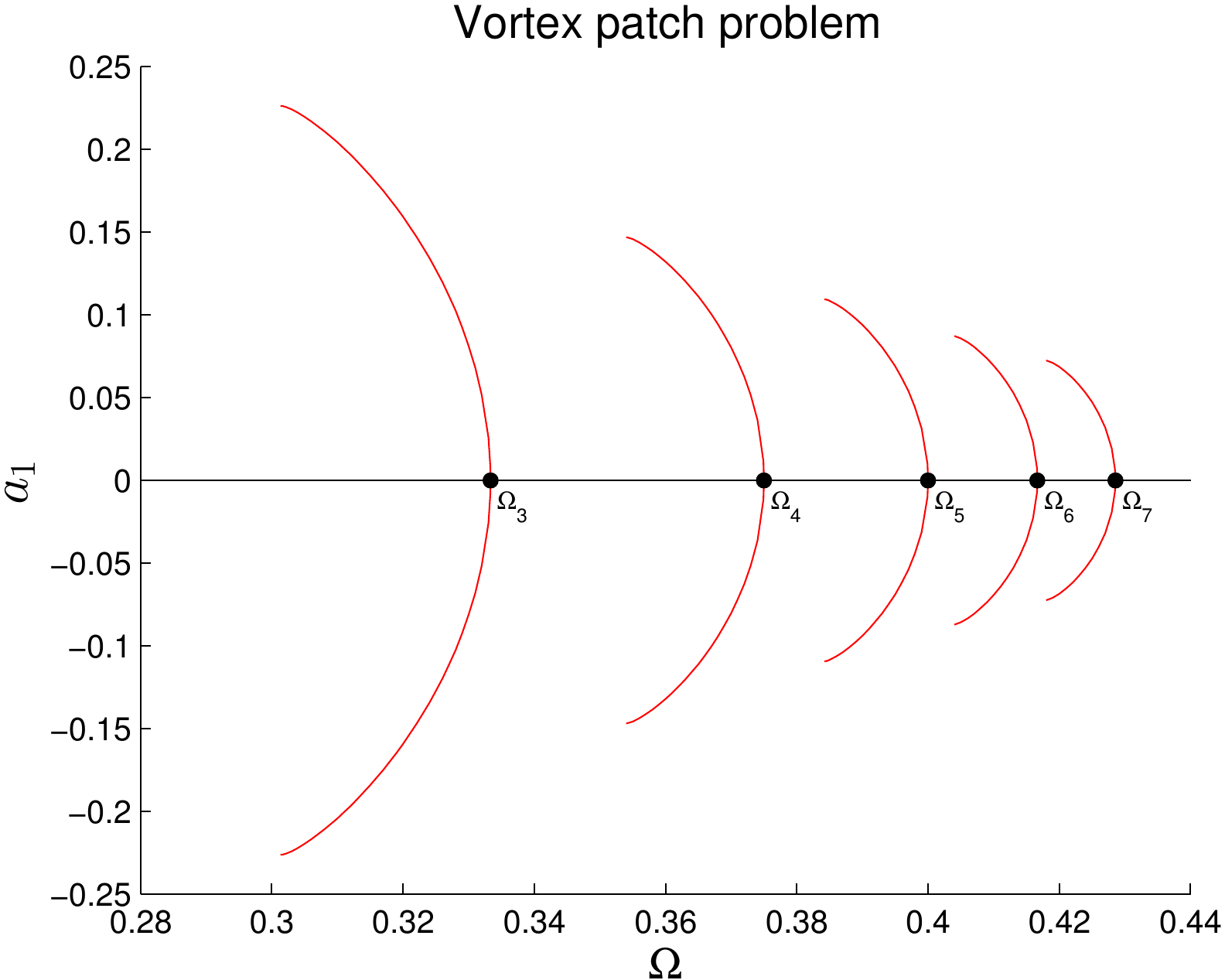}
\caption{Bifurcation curves associated with rotating vortex patches.
The numerical simulations are taken from \cite{DHMV16}.}
\label{fig:bifurcation}
\end{figure}

The remainder of this section is devoted to a detailed proof of
Burbea's theorem. We shall verify
the assumptions of the Crandall--Rabinowitz theorem. This requires the
use of appropriate Banach spaces, the study of the
regularity of the nonlinear functional \(F\), the computation of the
linearized operator around the disk, the characterization of its
kernel and range, and finally the verification of the transversality
condition.

\section{Technical lemmas}

In this section we collect several auxiliary results that will be repeatedly used throughout the analysis. The first lemma provides a simple characterization of functions whose Fourier coefficients are real. This property is closely related to the symmetry of the conformal parametrization and plays an important role in proving that the nonlinear functional $F$ preserves the symmetry of the space $X$. In particular, it allows us to identify the imaginary part of such functions with elements of the target space $Y$.

\begin{mylemma}{}{lem-stabili-realF}
Let
\[
f(w)=\sum_{n\in\mathbb Z} a_n w^n,
\qquad w\in\mathbb T,
\]
be the Fourier expansion of a function
$f\in L^2(\mathbb T)$. Then the following assertions are equivalent:

\begin{enumerate}

\item[(i)] All the Fourier coefficients $a_n$ are real.

\item[(ii)] The function $f$ satisfies
\[
f(\overline w)=\overline{f(w)},
\qquad \text{for a.e. } w\in\mathbb T.
\]
\end{enumerate}
Moreover, if $f\in C^\alpha(\mathbb T)$ satisfies \emph{(ii)}, then
\[
\hbox{Im}(f)\in Y.
\]
\end{mylemma}

\begin{proof}
Assume first that $a_n\in\mathbb R$ for every $n\in\mathbb Z$. Then
\[
f(\overline w)
=
\sum_{n\in\mathbb Z} a_n \overline w^n.
\]
On the other hand,
\[
\overline{f(w)}
=
\sum_{n\in\mathbb Z}\overline{a_n}\,\overline w^n
=
\sum_{n\in\mathbb Z}a_n\overline w^n,
\]
which immediately yields
\[
f(\overline w)=\overline{f(w)}.
\]
Conversely, assume that
\[
f(\overline w)=\overline{f(w)}
\]
for almost every $w\in\mathbb T$. Using the Fourier expansions,
\[
f(\overline w)
=
\sum_{n\in\mathbb Z}a_n\overline w^n,
\qquad
\overline{f(w)}
=
\sum_{n\in\mathbb Z}\overline{a_n}\,\overline w^n,
\]
and invoking the uniqueness of Fourier coefficients, we obtain
\[
a_n=\overline{a_n},
\qquad \forall n\in\mathbb Z.
\]
Hence all Fourier coefficients are real.
\\
Finally, if $f\in C^\alpha(\mathbb T)$ satisfies
\[
f(\overline w)=\overline{f(w)},
\]
then its Fourier coefficients are real. Consequently,
\[
\hbox{Im} f(w)
=
\sum_{n\ge 1} b_n \hbox{Im}(w^n),
\]
for some $ b_n\in\mathbb R,$
which shows that $\hbox{Im}(f)\in Y$.
\end{proof}

The next result is a classical regularity estimate for singular integral operators on the unit circle. It will be repeatedly used to establish the Hölder continuity of various integral expressions arising in the contour dynamics formulation. The assumptions imposed on the kernel are sufficiently weak to cover all operators appearing in the linear and nonlinear analysis.

\begin{mylemma}{}{SI}
Let $K(w,\xi)$ be a measurable function on
\[
\mathbb T\times\mathbb T
\setminus
\{(w,\xi)\in\mathbb T\times\mathbb T:\; w=\xi\},
\]
satisfying, for some constant $C_0>0$,
\[
|K(w,\xi)|\leqslant C_0,
\qquad w\neq\xi,
\]
and assume that, for each fixed $\xi\in\mathbb T$, the map
$w\mapsto K(w,\xi)$ is differentiable away from $\xi$ and satisfies
\[
\left|
{\partial_w}K(w,\xi)
\right|
\leqslant
\frac{C_0}{|w-\xi|}\cdot
\]
Then, for every $\alpha\in(0,1)$, the integral operator
\begin{equation*}
\mathcal L(f)(w)
=
\int_{\mathbb T}K(w,\xi)f(\xi)\,d\xi
\end{equation*}
maps $L^\infty(\mathbb T)$ continuously into $C^\alpha(\mathbb T)$ and satisfies
\[
\|\mathcal L(f)\|_{C^\alpha}
\leqslant
C_\alpha C_0
\|f\|_{L^\infty},
\]
where $C_\alpha$ depends only on $\alpha$.
\end{mylemma}

\begin{proof}
We first establish the boundedness of $\mathcal L(f)$. By the uniform bound on the kernel,
\[
|\mathcal L(f)(w)|
\leqslant
\|f\|_{L^\infty}
\int_{\mathbb T}|K(w,\xi)|\,|d\xi|
\leqslant
2\pi C_0\|f\|_{L^\infty}.
\]
Let now $w_1\neq w_2\in\mathbb T$ and set
\[
r=|w_1-w_2|>0.
\]
We decompose
\[
|\mathcal L(f)(w_1)-\mathcal L(f)(w_2)|
\leqslant
\mathcal I_1+\mathcal I_2+\mathcal I_3,
\]
where
\begin{align*}
\mathcal I_1
&=
\int_{B_{2r}(w_1)}
|f(\xi)|\,|K(w_1,\xi)|\,|d\xi|,
\\
\mathcal I_2
&=
\int_{B_{2r}(w_1)}
|f(\xi)|\,|K(w_2,\xi)|\,|d\xi|,
\\
\mathcal I_3
&=
\int_{B^c_{2r}(w_1)}
|f(\xi)|\,|K(w_1,\xi)-K(w_2,\xi)|\,|d\xi|.
\end{align*}
Using the boundedness of the kernel,
\[
\mathcal I_1+\mathcal I_2
\lesssim
\|f\|_{L^\infty}
|w_1-w_2|.
\]
To estimate $\mathcal I_3$, we invoke the mean value theorem together with the derivative estimate:
\[
|K(w_1,\xi)-K(w_2,\xi)|
\lesssim
\frac{|w_1-w_2|}{|w_1-\xi|},
\qquad
\xi\in B_{2r}^c(w_1).
\]
Therefore,
\[
\mathcal I_3
\lesssim
\|f\|_{L^\infty}
|w_1-w_2|
\int_{B_{2r}^c(w_1)}
\frac{|d\xi|}{|w_1-\xi|}\cdot
\]
Since
\[
\int_{B_{2r}^c(w_1)}
\frac{|d\xi|}{|w_1-\xi|}
\lesssim
1+|\log r|,
\]
we obtain
\[
\mathcal I_3
\lesssim
\|f\|_{L^\infty}
\,r(1+|\log r|).
\]
Finally, for every $\alpha\in(0,1)$,
\[
r(1+|\log r|)
\lesssim r^\alpha,
\]
which yields
\[
|\mathcal L(f)(w_1)-\mathcal L(f)(w_2)|
\lesssim
\|f\|_{L^\infty}
|w_1-w_2|^\alpha.
\]
The proof is complete.
\end{proof}

\section{Regularity of the map $F$}

In this section we investigate the regularity properties of the nonlinear functional $F$ introduced in \eqref{rotsq14}. Our main objective is to establish that $F$ is well defined on a suitable neighborhood of the trivial configuration and possesses enough smoothness to allow the application of bifurcation techniques.
Throughout this section, we fix parameters
\[
r,\alpha\in(0,1),
\]
and consider the open ball
\begin{align}\label{ball-alpha}
B_r^\alpha
=
\Big\{
f\in X:\;
\|f\|_{C^{1+\alpha}}< r
\Big\}.
\end{align}
The smallness assumption on $r$ guarantees that the conformal perturbations under consideration remain close to the identity map. In particular, this prevents the appearance of singularities in the integral kernels defining $F$ and ensures that all geometric quantities involved are well defined.
\\
The main result of this section is the following.

\begin{mytheorem}{}{Theorem-smooth}
Let $\alpha,r\in(0,1)$. The following assertions hold.

\begin{enumerate}

\item  The map
\[
F:\mathbb{R}\times B_r^\alpha\longrightarrow Y
\]
is well defined.

\item The map $F$ is of class $C^1$.

\item The mixed derivative
\[
\partial_\Omega\partial_fF
\]
exists and is continuous.
\end{enumerate}
\end{mytheorem}
\begin{proof}
{\bf{1.}}\quad We decompose $F$ from \eqref{rotsq14} into two parts
$$
F(\Omega,f)=2\Omega \hbox{Im}(F_1(f))+\hbox{Im}(F_2(f)),\quad\hbox{with} \quad F_1(f)(w):=  \overline{\Phi(w}) w\Phi^\prime(w)
$$
and 
\begin{align*}
F_2(f)(w)&=w\,\Phi^\prime(w)\fint_{\mathbb{T}}\tfrac{\overline{\Phi(\xi)}-\overline{\Phi(w)}}{\Phi(\xi)-\Phi(w)}\,\Phi^\prime(\xi)d\xi\\
&:=w\,\Phi^\prime(w)\mathcal{I}(\Phi)(w)
\end{align*}
where
$$
\mathcal{I}(\Phi)(w):=\mathcal{I}_\Phi(\Phi)(w),\quad \mathcal{I}_\Phi(h)(w)=\fint_{\TT}K(w,\xi) h^\prime(\xi)d\xi, \quad K(w,\xi)=\tfrac{\overline{\Phi(\xi)}-\overline{\Phi(w)}}{\Phi(\xi)-\Phi(w)}
$$
As $\Phi\in C^{1+\alpha}(\mathbb T)$, we have 
$\Phi'\in C^\alpha(\mathbb T)$. Hence, by the algebra property 
\eqref{algebra-cal}, we obtain that $F_1\in C^\alpha(\mathbb T)$. 
Moreover, Lemma \ref{lem:lem-stabili-realF} ensures that 
$F_1\in Y$ whenever $f\in X$.
Similarly, in order to prove that $F_2\in Y$, it suffices to establish 
that $\mathcal I(\Phi)\in Y$. We begin by proving the Hölder regularity
\[
\mathcal I(\Phi)\in C^\alpha(\mathbb T).
\]
It is straightforward that
$$
|K(w,\xi)|= 1
$$
Differentiating  the kernel $K$ with respect to \(w\), we get
\[
\partial_w K(w,\xi)
=
\frac{
-\partial_w\overline{\Phi(w)}
\big(\Phi(\xi)-\Phi(w)\big)
+
\big(\overline{\Phi(\xi)}-\overline{\Phi(w)}\big)
\Phi'(w)
}
{\big(\Phi(\xi)-\Phi(w)\big)^2}.
\]
Using the identity
\begin{equation*}\label{diff1}\forall w\in\mathbb{T},\quad{\partial_w \overline{\Phi(w)}} =
-\frac{1}{{w}^2} \overline{\Phi^\prime(w)}.
\end{equation*}
then we get
$$
\forall w\in\mathbb{T},\quad \left|{\partial_w \overline{\Phi(w)}}\right|\leqslant
|\Phi^\prime(w)|.
$$
Therefore,
\begin{align*}
|\partial_w K(w,\xi)|
&\leqslant
\frac{
|\partial_w\overline{\Phi(w)}|
\,|\Phi(\xi)-\Phi(w)|
+
|\overline{\Phi(\xi)}-\overline{\Phi(w)}|
\,|\Phi'(w)|
}
{|\Phi(\xi)-\Phi(w)|^2}\\
&\leqslant \frac{2|\Phi^\prime(w)|}{|\Phi(\xi)-\Phi(w)|}
\end{align*}
Since
\[
\Phi(w)=w+f(w),
\qquad \|f\|_{\mathrm{Lip}}\leqslant r<1,
\]
we have
\[
|\Phi(\xi)-\Phi(w)|
\geqslant |\xi-w|-|f(\xi)-f(w)|
\geqslant (1-r)|\xi-w|.
\]
Similarly,
\[
|\Phi(\xi)-\Phi(w)|
\leqslant (1+r)|\xi-w|,
\]
which implies that 
\[
|\Phi'(w)|\leqslant  1+r.
\]
Hence
\begin{align}\label{est-kern}
\forall w\neq\xi,\quad |\partial_w K(w,\xi)|
\leqslant  2\frac{1+r}{1-r}
\frac{1}{|\xi-w|}\cdot
\end{align}
Now, we can apply Lemma \ref{lem:SI} allowing to get
\begin{align}\label{singular-1}
\|\mathcal{I}_\Phi(h)\|_\alpha\leqslant C\|h^\prime\|_{L^\infty}
\end{align}
which shows in particular  that $\mathcal{I}(\Phi)\in C^\alpha(\TT)$. Next, we shall prove that the Fourier coefficients of $\mathcal{I}(\Phi)$ are real. For this aim, we shall use Lemma \ref{lem:lem-stabili-realF} and show that
\begin{align}\label{Sym-id}
\overline{\mathcal{I}(\Phi)(w)}=\mathcal{I}(\Phi)(\overline{w})
\end{align}
Since $\Phi$ has real Fourier coefficients, we have
\[
\Phi(\overline w)=\overline{\Phi(w)},
\qquad
\Phi'(\overline w)=\overline{\Phi'(w)}.
\]
Therefore
\[
K(\overline w,\overline \xi)
=
\frac{\overline{\Phi(\overline \xi)}-\overline{\Phi(\overline w)}}
{\Phi(\overline \xi)-\Phi(\overline w)}
=
\frac{\Phi(\xi)-\Phi(w)}
{\overline{\Phi(\xi)}-\overline{\Phi(w)}}
=
\overline{K(w,\xi)}.
\]
We make the change of variable
$
\eta=\overline{\xi},\,\xi\in\mathbb T
,$ which reverses the orientation of the unit circle, we get
\begin{align*}
\mathcal I(\Phi)(\overline w)
&=
\fint_{\mathbb T}K(\overline w,\eta)\Phi'(\eta)\,d\eta\\
&=-\fint_{\mathbb T}K(\overline w,\overline{\xi})\Phi'(\overline{\xi})\,d\overline{\xi}\\
&=\overline{
\fint_{\mathbb T}K(w,\xi)\Phi'(\xi)\,d\xi
}
\end{align*}
Thus
\[
\mathcal I(\Phi)(\overline w)
=
\overline{
\mathcal I(\Phi)(w)
}.
\]
This ends the proof of \eqref{Sym-id}.
\\
We now compute the differential of $F$ with respect to $f$, while keeping
$\Omega$ fixed.We shall first compute Gateaux derivatives and one can easily show that it will coincide with  Fréchet derivative. Let $h\in X$
and define
\[
\Phi_t(w)=w+f(w)+t h(w)=\Phi(w)+t h(w).
\]
Then
\[
\partial_fF(\Omega,f)[h]
=
\left.\frac{d}{dt}F(\Omega,f+th)\right|_{t=0}.
\]
Since
\[
F(\Omega,f)(w)
=
\operatorname{Im}
\Bigg[
\Big(
2\Omega\overline{\Phi(w)}
+
\mathcal I(\Phi)(w)
\Big)
w\Phi'(w)
\Bigg],
\]
we get
\begin{align}\label{Linear-op}
\nonumber \partial_fF(\Omega,f)[h](w)
=&
\operatorname{Im}
\Big[
\Big(
2\Omega\overline{h(w)}
+
\partial_f\mathcal I(\Phi)[h](w)
\Big)w\Phi'(w)
\Big]\\
&
+\operatorname{Im}\Big[\Big(
2\Omega\overline{\Phi(w)}+
\mathcal I(\Phi)(w)
\Big)w h'(w)
\Big].
\end{align}
From direct computations, we obtain
\[
\partial_f\mathcal I(\Phi)[h](w)
=
\fint_{\mathbb T}
\partial_fK(\Phi)[h](w,\xi)\Phi'(\xi)\,d\xi
+
\mathcal I_\Phi(h)(w)
\]
where
\[
K_1(w,\xi):=\partial_fK(\Phi)[h](w,\xi)
=
\frac{
\big(\overline{h(\xi)}-\overline{h(w)}\big)
\big(\Phi(\xi)-\Phi(w)\big)
-
\big(\overline{\Phi(\xi)}-\overline{\Phi(w)}\big)
\big(h(\xi)-h(w)\big)
}
{\big(\Phi(\xi)-\Phi(w)\big)^2}.
\]
It follows that 
\[
\begin{aligned}\label{eq-lineariz}
\nonumber \partial_fF(\Omega,f)[h](w)
=
\operatorname{Im}
\Bigg[
&
\Bigg(
2\Omega\overline{h(w)}
+
\fint_{\mathbb T}
K_1(w,\xi)\Phi'(\xi)\,d\xi
+
\mathcal{I}_\Phi(h)(w)
\Bigg)
w\Phi'(w)
\\
&
+
\Bigg(
2\Omega\overline{\Phi(w)}
+\mathcal{I}(\Phi)(w)
\Bigg)
w h'(w)
\Bigg].
\end{aligned}
\]
Applying \eqref{singular-1}, we get  for $f\in B_r^\alpha$, using Sobolev embeddings
\begin{align}\label{singular-2}
\|\mathcal{I}(\Phi)\|_\alpha\leqslant C,\quad\hbox{and}\quad\|\mathcal{I}_\Phi(h)\|_\alpha\leqslant C\|h\|_{1+\alpha}
\end{align}
On the other hand, and similarly to \eqref{est-kern} we get 
\begin{align*}
\forall \xi\neq w,\quad \Big|K_1(w,\xi)\Big|&\leqslant 2
\frac{
\big|{h(\xi)}-{h(w)}\big|
}
{\big|\Phi(\xi)-\Phi(w)\big|}\\
&\leqslant \frac{2}{1-r}\|h\|_{\textnormal{Lip}}\\
&\leqslant \frac{2}{1-r}\|h\|_{1+\alpha}.
\end{align*}
and
\begin{align*}
\forall \xi\neq w,\quad \Big|\partial_wK_1(w,\xi)\Big|&\leqslant C\|h\|_{\textnormal{Lip}}\frac{1}{|w-\xi|}\\
&\leqslant C\|h\|_{1+\alpha}\frac{1}{|w-\xi|}.
\end{align*}
Applying, once again Lemma \ref{lem:SI} we get
\begin{align}\label{singular-3}
\nonumber\left\|\fint_{\mathbb T}
K_1(\cdot,\xi)\Phi'(\xi)\,d\xi\right\|_\alpha\leqslant C\|\Phi^\prime\|_{L^\infty}\|h\|_{\textnormal{Lip}}\\
\leqslant C\|h\|_{1+\alpha}
\end{align}
Putting together \eqref{singular-2} and \eqref{singular-3} and using the fact that $X$ is an algebra and $f\in B_r^\alpha$, we find
\begin{align*}
\nonumber\|\partial_fF(\Omega,f)[h]\|_\alpha
\leqslant  C\|h\|_{1+\alpha}
\end{align*}
which shows that the linear operator $\partial_fF(\Omega,f)$ is a continuous bounded operator from $X$ \mbox{to $Y$, that is, $\partial_fF(\Omega,f)\in \mathcal{L}(X,Y)$ .} The continuity of the functional 
$$f\in B_r^\alpha\to \partial_fF(\Omega,f)\in  \mathcal{L}(X,Y)
$$ can be done in a similar way. For more details we refer to the papers \cite{HH15,HMV13}. This allows to show that the functional $f\in B_r^\alpha\mapsto F(\Omega,f)\in Y$ is of class $C^1.$ The last point is to differentiate in $\Omega$ the operator \eqref{Linear-op}, leading to 
\begin{align}\label{Linear-op}
\nonumber \partial_\Omega\partial_fF(\Omega,f)[h](w)
=&
\operatorname{Im}
\Big[
2\overline{h(w)}w\Phi'(w)+2\overline{\Phi(w)}
w h'(w)
\Big].
\end{align}
We can easily see  that it is continuous on $B_r^\alpha.$ This ends the proof of the desired results.
\end{proof}
\section{Spectral study of the linearized operator}

The purpose of this section is to analyze the linearized operator
associated with the nonlinear functional \(F\) around the trivial
solution corresponding to the Rankine vortex.  This spectral analysis
constitutes the core of Burbea's bifurcation argument and reveals the
critical angular velocities at which nontrivial rotating vortex
patches emerge.
\\
Recall that the trivial branch of solutions is given by
\[
(\Omega,0),
\qquad
\Omega\in\mathbb R,
\]
and satisfies
\[
F(\Omega,0)=0.
\]
We define the linearized operator
\[
L_\Omega:=\partial_fF(\Omega,0).
\]
As we shall see below, a remarkable feature of the problem is that the Fourier basis
diagonalizes the linearized operator. Consequently, each Fourier mode
evolves independently and the spectral study reduces to an explicit
computation of the corresponding eigenvalues.

\subsection{Linearization at the disk}

We now turn to the linearization of the rotating patch equation around the
unit disk, corresponding to the trivial configuration $f=0$. This computation
is the key spectral step in the bifurcation analysis. A remarkable feature
of the disk is that the linearized operator is diagonal in the Fourier basis,
so that its spectrum can be computed explicitly. In particular, the loss of
invertibility occurs at a  discrete sequence of angular velocities, which will
provide precisely the bifurcation values used later in the
Crandall--Rabinowitz theorem.
The main result of this section is the following.
\begin{myproposition}{}{Spectral-prop}
Let
\[
h(w)=\sum_{n\geqslant0}a_n\overline w^{\,n}\in X.
\]
Then
\[
L_\Omega h(w)
=
2\sum_{n\geqslant1}
n\left(\Omega-\Omega_n\right)
a_{n-1}\operatorname{Im}(w^n),
\]
where
\[
\Omega_n=\tfrac{n-1}{2n},
\qquad n\geqslant1.
\]
\end{myproposition}
\begin{proof}
We start from the general expression for the differential of \(F\) with
respect to \(f\in B_r^\alpha\). Let
\[
\Phi(w)=w+f(w)
\]
and let \(h\in X\). It follows from \eqref{Linear-op} that
\begin{align}\label{structur1}
\partial_fF(\Omega,f)h(w)
=
\operatorname{Im}
\Big[
\big(
2\Omega\overline{h(w)}
+
\partial_f\mathcal I(\Phi)h(w)
\big)w\Phi'(w)
+
\big(
2\Omega\overline{\Phi(w)}
+
\mathcal I(\Phi)(w)
\big)w h'(w)
\Big],
\end{align}
where
\[
\mathcal I(\Phi)(w)
=
\fint_{\mathbb T}
K(w,\xi)\Phi'(\xi)\,d\xi,
\qquad
K(w,\xi)
=
\frac{\overline{\Phi(\xi)}-\overline{\Phi(w)}}
{\Phi(\xi)-\Phi(w)}.
\]
Moreover,
\[
\partial_f\mathcal I(\Phi)h(w)
=
\fint_{\mathbb T}
\partial_fK(\Phi)h(w,\xi)\Phi'(\xi)\,d\xi
+
\fint_{\mathbb T}
K(w,\xi)h'(\xi)\,d\xi,
\]
with
\[
\partial_fK(\Phi)h(w,\xi)
=
\frac{
\big(\overline{h(\xi)}-\overline{h(w)}\big)
\big(\Phi(\xi)-\Phi(w)\big)
-
\big(\overline{\Phi(\xi)}-\overline{\Phi(w)}\big)
\big(h(\xi)-h(w)\big)
}
{\big(\Phi(\xi)-\Phi(w)\big)^2}.
\]
We now specialize this formula to the trivial solution \(f=0\). In
that case,
\[
\Phi(w)=w,\qquad \Phi'(w)=1,
\]
and
\[
\forall w,\xi\in\mathbb T,\quad K(w,\xi)
=
\frac{\overline{\xi}-\overline w}{\xi-w}=-\overline{w}\,\overline{\xi}\,.
\]
Consequently,
\begin{align}\label{Id-1}
\mathcal I(\hbox{Id})(w)
=
\fint_{\mathbb T}
K(w,\xi)\,d\xi
=
-\frac1w
\fint_{\mathbb T}
\frac{d\xi}{\xi}
=
-\frac1w.
\end{align}
Similarly,
\begin{align*}
\partial_fK(\hbox{Id})h(w,\xi)
&=
\frac{
\big(\overline{h(\xi)}-\overline{h(w)}\big)(\xi-w)
-
(\overline{\xi}-\overline w)(h(\xi)-h(w))
}
{(\xi-w)^2}\\
&=\frac{\overline{h(\xi)}-\overline{h(w)}}{\xi-w}
+
\frac{h(\xi)-h(w)}{w\xi(\xi-w)}.
\end{align*}
Therefore,
\[
\partial_f\mathcal I(\hbox{Id}))[h](w)
=
\fint_{\mathbb T}
\left[
\frac{\overline{h(\xi)}-\overline{h(w)}}{\xi-w}
+
\frac{h(\xi)-h(w)}{w\xi(\xi-w)}
\right]d\xi
-\overline{w}
\fint_{\mathbb T}
\overline{\xi}\,h'(\xi)\,d\xi.
\]
As $h\in X$, then the map 
\[\xi\in\mathbb{D}\mapsto 
\frac{\overline{h(\xi)}-\overline{h(w)}}{\xi-w}\]
is holomorphic and has continuous extension up to the  boundary, and we get
$$
\fint_{\mathbb T}
\frac{\overline{h(\xi)}-\overline{h(w)}}{\xi-w}
d\xi
=0.
$$
On the other hand, the map
\[
\xi\in\mathbb{C}\backslash\overline{\mathbb{D}}\mapsto \frac{h(\xi)-h(w)}{w\xi(\xi-w)}
\]
is holomorphic and decays  like $\xi^{-2}$ at $\infty$, then by residue theorem, we infer
\[
\fint_{\mathbb T}
\frac{h(\xi)-h(w)}{w\xi(\xi-w)}
d\xi=0.
\]
As before, using residue theorem yields
\[
\fint_{\mathbb T}
\overline{\xi}{h'(\xi)}\,d\xi=\fint_{\mathbb T}
\frac{h'(\xi)}{\xi}\,d\xi=0.
\]
Putting together the preceding identities give
$$
\partial_f\mathcal I(\hbox{Id})[h]=0.
$$
Plugging this identity together with \eqref{Id-1}  into \eqref{structur1} allows to get
\begin{align}\label{structur2}
L_\Omega h(w)
=
\operatorname{Im}
\Big[
2\Omega\, w\,\overline{h(w)}
+
\big(
2\Omega-1\big) h'(w)
\Big],
\end{align}
For
\[
h(w)=\sum_{n\geqslant0}h_n\overline{w}^{n}\in X,
\]
we find
\[
\begin{aligned}
L_\Omega h(w)
&=
\operatorname{Im}
\Big\{
(2\Omega-1)h'(w)+2\Omega w\overline{h(w)}
\Big\}
\\
&=
\sum_{n\geqslant1}
n\left(2\Omega-\frac{n-1}{n}\right)
h_{n-1}\operatorname{Im}(w^n).
\end{aligned}
\]
Thus, the linearized operator at the disk is diagonal in the Fourier basis,
with the corresponding Fourier multipliers given explicitly above.
This completes the proof of the proposition.
\end{proof}

\subsection{Spectral structure}
We now exploit the Fourier representation of the linearized operator
obtained in Proposition \ref{prop:Spectral-prop} to identify its kernel
and range at the critical values $\Omega_m$.
The following proposition summarizes the spectral properties required for the bifurcation analysis.
\begin{myproposition}{}{Spectral structure-proppp}
For each $m\geqslant1,$
\[
N(L_{\Omega_m})
=
\operatorname{span}
\{\overline w^{\,m-1}\},
\]
and
\[
R(L_{\Omega_m})
=
\Bigl\{
h\in Y:
h_m=0
\Bigr\},
\]
In particular,
\[
\dim N(L_{\Omega_m})=1,
\]
and
\[
\dim\bigl(Y/R(L_{\Omega_m})\bigr)=1.
\]
\end{myproposition}
\begin{proof}
According to Proposition \ref{prop:Spectral-prop}, we get for any 
\[
h(w)=\sum_{n\geqslant0}h_n\overline{w}^{n}\in X,
\]
we find
\[
\begin{aligned}
L_\Omega h(w)
&=
\operatorname{Im}
\Big\{
(2\Omega-1)h'(w)+2\Omega w\overline{h(w)}
\Big\}
\\
&=
\sum_{n\geqslant1}
n\left(2\Omega-\tfrac{n-1}{n}\right)
h_{n-1}\operatorname{Im}(w^n).
\end{aligned}
\]
This shows that the linearized operator at the equilibrium is a Fourier multiplier and it is easy to see that its kernel is nontrivial if and only if 
$$
\Omega=\Omega_m
$$
for some $m\geqslant1.$ Moreover, it is plain that the sequence $n\geqslant1\mapsto \Omega_n$ is strictly increasing, implying that the kernel is one dimensional when $\Omega=\Omega_m$ and
the kernel of the linearized operator is
\[
N(\partial_fF(\Omega_m,0))
=
\operatorname{span}
\{\overline w^{\,m-1}\}.
\]
 Now, we shall fix  $\Omega=\Omega_m$ for some $m\geqslant1$ and explore the range of $L_{\Omega_m}$. One can first see the embedding
$$
R(L_{\Omega_m})
\subset
\Bigl\{
h\in Y:
h_m=0
\Bigr\}:=\mathcal{Y}.
$$
It remains to show the converse. Take an element $g\in \mathcal{Y}$ and let's solve the equation
$$
L_{\Omega_m}h=g.
$$
Remind from \eqref{space-Y} that
$$
w\in\T\mapsto g(w)
=
\sum_{n\geqslant 1}g_n\,\operatorname{Im}(w^n)\in C^\alpha(\mathbb T),
\quad g_n\in\mathbb R\,.
$$
Using Fourier expansion, we consider $h\in X$, defined in \eqref{space-X}, we get
$$
2n\left(\Omega_m-\Omega_n\right)
h_{n-1}=g_n, \forall n\geqslant1, n\neq m
$$
that is
$$
\tfrac{m-n}{m}
h_{n-1}=g_n, \forall n\geqslant1, n\neq m
$$
which implies that
\begin{eqnarray*}
 \nonumber h_{n-1} & =& \,\tfrac{m}{m-n}\,g_{n}\\
&=&-\tfrac{m}{n}\,g_{n}+\tfrac{m^2}{(m-n)n}\,g_{n}.
\end{eqnarray*}
Consequently
$$
h(w)=-m\sum_{n\geqslant1\atop n\neq m}\tfrac{g_{n}}{n}\overline{w}^{n-1}+m^2\sum_{n\geqslant1\atop n\neq m}\tfrac{g_{n}}{n(m-n)}\overline{w}^{n-1}.
$$
Taking the complex conjugate
\begin{align*}
H(w):=w\overline{h(w)}&=-m\sum_{n\geqslant1\atop n\neq m}\tfrac{g_{n}}{n}{w}^{n}+m^2\sum_{n\geqslant1\atop n\neq m}\tfrac{g_{n}}{n(m-n)}{w}^{n}\\
&:=-m G_1(w)+m^2 G_2(w).
\end{align*}
The next goal is to check that $h\in C^{1+\alpha}(\mathbb{T}).$ This is equivalent to show that $H \in C^{1+\alpha}(\mathbb{T})$ as $h(w)=w\overline{H}$ and one can use the law products in $ C^{1+\alpha}(\mathbb{T})$ which is an algebra. We first show that $G_1 \in C^{1+\alpha}(\mathbb{T})$. Notice that $G_1\in L^\infty(\T)$ since by Cauchy-Schwarz inequality
\begin{align*}
|G_1(w)|&\leqslant \sum_{n\geqslant1\atop n\neq m}\tfrac{|g_{n}|}{n}\\
 &\lesssim \Big(\sum_{n\geqslant1\atop n\neq m}{|g_{n}|^2}\Big)^{\frac{1}{2}}\\
 &\lesssim \|g\|_{L^2(\T)}.
\end{align*} 
Therefore, 
\begin{align*}
\|G_1\|_{L^\infty(\T)}&\lesssim  \|g\|_{L^\infty(\T)}\\
&\lesssim \|g\|_{C^\alpha}.
\end{align*}
 Moreover, by differentiation we get
$$
G_1^\prime(w)=\overline{w}\sum_{n\geqslant1\atop n\neq m}{g_{n}}{w}^{n}.
$$
Recall that the Cauchy projection (or Szeg\H{o} projection)
\begin{equation*}\label{Cauchyproj}
\Pi_+:\sum_{n\in\mathbb Z}c_n w^n
\longmapsto
\sum_{n\geqslant 0}c_n w^n
\end{equation*}
is bounded on $C^\alpha(\mathbb T)$, whereas it is not bounded on
$L^\infty(\mathbb T)$. Therefore, we may write
$$
G_1^\prime(w)=\overline{w}\Pi_+(g)(w)
$$
allowing to deduce that $G_1^\prime\in C^{\alpha}(\T)$. It follows that $G_1\in C^{1+\alpha}(\T)$.   \\
Let's now show that $G_2\in C^{1+\alpha}(\T)$. As before, we may easily show that $G_2\in L^\infty(\T).$ By differentiation
\begin{align*}
wG_2^\prime(w)&=\sum_{n\geqslant1\atop n\neq m}\tfrac{g_{n}}{m-n}{w}^{n}\\
&=(K*\Pi_+(g))(w),\quad w\in\T
\end{align*}
where
$$K(w)= \sum_{ n=1\atop  n\neq m}^\infty \frac{{w}^n}{n-m}, \quad w\in\T.$$
Thus  $K \in L^2(\T) \subset L^1(\T)$ and by the convolution laws, as $\Pi_+(g) \in
C^{\alpha}(\T),$ we deduce that 
$$
wG_2^\prime\in C^{\alpha}(\T).
$$
Hence $G_2^\prime\in C^{\alpha}(\T).$ This achieves that $H\in C^{1+\alpha}(\T).$ We finally get
$$
R(L_{\Omega_m})
=
\Bigl\{
h\in Y:
h_m=0
\Bigr\}.
$$
This achieves the proof of the desired results.

\end{proof}
\subsection{Transversality condition}\label{Section-Transv}

We have seen in Proposition \ref{prop:Spectral structure-proppp} that for $\Omega=\Omega_m$ the kernel of $L_\Omega$ is one dimensional and generated by  the vector
\[
v_m(w)=\overline w^{\,m-1}.
\]
Differentiating the expression of the linearized operator with respect
to \(\Omega\), described by Proposition \ref{prop:Spectral-prop},  we obtain for every 
$
h(w)=\sum_{n\ge0}h_n\overline{w}^{n}\in X,
$,
\[
\partial_\Omega\partial_fF(\Omega,0)h
=
\sum_{n\ge1}
2n
h_{n-1}\operatorname{Im}(w^n).
\]
Evaluating this identity at
$
\Omega=\Omega_m$ 
and applying it to \(v_m\), we get
\[
\partial_\Omega\partial_fF(\Omega_m,0)
(v_m)
=
2m\,\operatorname{Im}(w^m).
\]
On the other hand, by the characterization of the range from Proposition \ref{prop:Spectral structure-proppp} we easily get that 
\[
\operatorname{Im}(w^n)\notin
R(\partial_fF(\Omega_m,0)).
\]
Therefore
\[
\partial_\Omega\partial_fF(\Omega_m,0)
(v_m)
\notin
R(\partial_fF(\Omega_m,0)).
\]
This is precisely the transversality condition required by the
Crandall--Rabinowitz theorem.
%

\section{Proof of Theorem \ref{thm:Burbea}}\label{sec:Burea-q}

We are now in a position to prove Burbea's result, in the formulation
revisited in \cite{HMV13} and stated in Theorem \ref{thm:Burbea}.
To this end, it remains to verify the hypotheses of the
Crandall--Rabinowitz theorem. Recall that
\[
F:\mathbb R\times B_r^\alpha\longrightarrow Y
\]
is the nonlinear functional defined by \eqref{rotsq14} and a rotating patch  solution to Euler equations can be viewed as a zero to this functional $F$, that is
$$
F(\Omega,f)=0
$$
From Theorem \ref{thm:Theorem-smooth}, the map
is of class \(C^1\) and the partial derivative $\partial_\Omega\partial_fF
$
exists and is continuous. Therefore assumption (b) of Theorem \ref{thm:Crandall-Rab}  is satisfied.
In addition, the Rankine vortex generates a trivial branch of solutions which is equivalent to 
\[
F(\Omega,0)=0,
\qquad
\forall\,\Omega\in\mathbb R.
\]
This verifies assumption (a).
Moreover, 
fix an integer \(m\geqslant2\) and set
\[
\Omega=\Omega_m.
\]
Assumption $(c)$ of the Crandall--Rabinowitz theorem follows directly
from Proposition~\ref{prop:Spectral structure-proppp}, while the
transversality condition $(d)$ was established in
Section~\ref{Section-Transv}. Hence, all the hypotheses of the
Crandall--Rabinowitz theorem are satisfied. Consequently, for each
$m\geqslant 2$, there exists a local curve of nontrivial solutions
bifurcating from the trivial branch at $\Omega=\Omega_m,$ that is
\[
(\Omega_s,f_s), \forall s\in(-\delta,\delta)
\]
with $\delta>0$
and
\[
\Omega_0=\Omega_m,\quad f_0=0.
\]
The formal mode \(m=1\) corresponds to translations of the disk and is
excluded by our normalization, while the \(m=2\) branch recovers the
classical Kirchhoff ellipses.
\\Finally 
\[\Phi_s(z)=z+f_s(z), \; |z|\geqslant 1,
\] is a
conformal mapping of $\C_\infty \setminus {\overline{\mathbb{D}}}$
into some domain $U_s$ and $D_s = \C_\infty
\setminus{\overline{U_s}}$ is a simply connected vortex patch
which rotates with angular velocity $\Omega_s.$ We know that $D_s$ is a domain with boundary
of class $C^{1+\alpha}$, but nothing else can be said about its
symmetry properties without further arguments. 
The $m$-fold symmetry of the bifurcating solutions is obtained by restricting the functional setting to suitable invariant subspaces. Let $m\geqslant2$ be fixed and define $X_m$ as the closed subspace of $X$ consisting of functions whose Fourier expansion contains only the frequencies compatible with an $m$-fold symmetry, namely
\begin{equation}\label{expansionm}
f(w)=\sum_{n=1}^{\infty}a_{nm-1}\,\overline{w}^{\,nm-1},
\qquad w\in\mathbb T.
\end{equation}
If $f$ belongs to the open unit ball of $X_m$, then the associated conformal mapping
\[
\Phi(z)=z+f(z),\qquad |z|\geqslant1,
\]
admits the expansion
\begin{equation}\label{expansionmfi}
\Phi(z)
=
z\left(1+\sum_{n=1}^{\infty}\frac{a_{nm-1}}{z^{nm}}\right),
\qquad |z|\geqslant1.
\end{equation}
From this representation, one immediately obtains
\[
\Phi\big(e^{2\pi i/m}z\big)
=
e^{2\pi i/m}\Phi(z),
\qquad |z|\geqslant1,
\]
which is precisely the $m$-fold rotational symmetry of the corresponding vortex patch.
Next, let $Y_m$ be the closed subspace of $Y$ consisting of functions whose Fourier coefficients vanish at all frequencies that are not nonzero multiples of $m$.
A direct computation shows that the nonlinear functional $F$, defined by \eqref{rotsq14}, is equivariant under the action of the cyclic group generated by the rotation
\[
z\longmapsto e^{2\pi i/m}z.
\]
Consequently,
\[
F:\mathbb{R}\times B_{r,m}^{\alpha}\longrightarrow Y_m,
\]
where
\[
B_{r,m}^{\alpha}=B_r^{\alpha}\cap X_m
\]
denotes the restriction of the ball defined in \eqref{ball-alpha} to the subspace \(X_m\).
Therefore, the Crandall--Rabinowitz theorem can be applied directly to the restricted mapping.
This concludes the proof of the Burbea's result.\\
The preceding analysis relies crucially on the explicit spectral structure of the linearized Euler operator. In particular, the bifurcation values can be computed explicitly, and their simplicity and nondegeneracy follow directly from the formula
\(
\Omega_n=\frac{n-1}{2n}.
\)
This explicit computation is a special feature of the logarithmic Euler kernel and is no longer available for general active scalar equations. The purpose of Chapter~7 is to replace this explicit Euler computation by an abstract mechanism deriving the required spectral properties directly from the interaction kernel.

\chapter{A unified theory of $V$-States for active scalar equations}
{\it 
This chapter develops a unified bifurcation theory for rotating vortex
patches in a broad class of active scalar equations. The approach
identifies structural assumptions on the interaction kernel that ensure
the existence of nontrivial $m$-fold symmetric $V$-states bifurcating
from the disk. The main ingredient is a general spectral analysis of the
linearized operator, based on complete monotonicity properties of the
kernel, which yields the monotonicity and nondegeneracy of the relevant
eigenvalues. The framework encompasses the Euler, generalized SQG and
quasi-geostrophic shallow-water equations, among other active scalar
models. The presentation is based on $\cite{HXX26},$ but is developed here
within the  conformal parametrization framework, rather than the polar
coordinate formulation used there.
}
\section{General motivation}

The preceding chapters culminated in Burbea's bifurcation construction
for the two-dimensional Euler equations. We now examine how this
mechanism extends to other active scalar models.\\
A first major extension is provided by the generalized surface
quasi-geostrophic (gSQG) equations, which interpolate between the Euler
and surface quasi-geostrophic equations through a fractional power of
the Laplacian. As the singularity of the interaction kernel increases,
the contour dynamics equation becomes more singular, and the classical
Euler estimates are no longer sufficient. Hmidi and Hassainia
\cite{HH15} extended the bifurcation approach to the gSQG family by
developing functional settings adapted to the singularity of the kernel.
The corresponding problem in the unit disk was subsequently investigated
in \cite{HXX23}. In the more singular, supercritical regime, the
bifurcation analysis was carried out by Castro, C\'ordoba and
G\'omez-Serrano \cite{CCG16}.

A different difficulty arises for the quasi-geostrophic shallow-water
(QGSW) equation, whose interaction kernel is expressed in terms of
modified Bessel functions. Unlike the Euler and gSQG kernels, it is not
homogeneous, due to the presence of a finite Rossby deformation length,
and arguments based on scaling are therefore no longer available.
Dritschel, Hmidi and Renault \cite{DHR19} established
the existence of rotating patches in this setting through a contour
dynamics formulation and estimates specifically adapted to the Bessel
kernel.

These developments naturally raise the question of whether the
bifurcation of $V$-states is tied to the particular form of each kernel,
or is instead governed by more fundamental structural properties. The
latter perspective was developed by Hmidi, L. Xue and Z. Xue
\cite{HXX26}, who identified general conditions on the interaction
kernel under which the bifurcation theory can be carried out
simultaneously for a large class of active scalar equations.

The common structure is particularly transparent at the level of the
linearization around the disk. After reformulating the rotating patch
problem as a nonlinear equation on the unit circle, the disk gives rise
to a trivial branch of solutions, and the bifurcation analysis is
reduced to understanding the spectrum of the corresponding linearized
operator. The crucial ingredient is a common representation of its
eigenvalues, from which the monotonicity and nondegeneracy properties
required by the Crandall--Rabinowitz theorem can be derived.

A central role is played by the complete monotonicity of a suitable
radial profile associated with the interaction kernel. Combined with
appropriate assumptions controlling its behavior near the origin, this
property provides a robust spectral mechanism encompassing the Euler,
gSQG and QGSW equations, as well as a much broader class of active scalar
models. 

\section{A general class of active scalar equations}
We consider the class of active scalar equations
\begin{equation}\label{AS}
\left\{
\begin{aligned}
&\partial_t\omega+v\cdot\nabla\omega=0,\\
v&=\nabla^\perp\psi,\\
\psi(x)&=\int_{\mathbb R^2}K(x-y)\omega(y)\,dy,
\end{aligned}
\right.
\end{equation}
where
\(
\nabla^\perp=(\partial_2,-\partial_1).
\) and 
 the interaction kernel 
\[
K:\mathbb R^2\setminus\{0\}\longrightarrow\mathbb R
\]
is assumed to be radial,
\begin{equation}\label{kernel-general}
K(x)=K_0(|x|),
\qquad x\in\mathbb R^2,
\end{equation}
so that \eqref{AS} is invariant under translations and rotations. Since
$v=\nabla^\perp\psi$ is divergence free, \eqref{AS} defines an
incompressible transport dynamics.
We emphasize that our convention for $\nabla^\perp$ is the opposite of
the one adopted in the preceding chapters. This choice is made for
convenience and will simplify the kernel representation used throughout
this chapter.
\\
The scalar $\omega$ may represent different physical quantities
depending on the model under consideration. For instance, it is the
vorticity for the Euler equations, the temperature for the
surface quasi-geostrophic equation, and the potential vorticity for
the shallow-water quasi-geostrophic model.
\\
The terminology \emph{active scalar equation} emphasizes that the
transported scalar determines, through the constitutive law defining
the stream function, the velocity field responsible for its own
transport. Consequently, all the qualitative properties of the
evolution are encoded in the interaction kernel $K$.
Next, we shall give some physical examples.

\medskip

\noindent
\textbf{(i) The two-dimensional Euler equations.}
The two-dimensional incompressible Euler equations, discussed extensively
in the preceding chapters, provide the classical example of an active
scalar equation. In vorticity form, they read

\[
\partial_t\omega+v\cdot\nabla\omega=0,
\qquad
v=\nabla^\perp(-\Delta)^{-1}\omega.
\]
Since the fundamental solution of $\textcolor{red}{-}\Delta$ in $\mathbb{R}^2$ is
\[
K(x)=-\frac{1}{2\pi}\log|x|,
\]
the Euler equation fits the general framework with the logarithmic
interaction kernel.

\medskip

\noindent
\textbf{(ii) The generalized surface quasi-geostrophic equations.}

The generalized surface quasi-geostrophic (gSQG) equations form a
one-parameter family of active scalar equations connecting the
two-dimensional Euler equation to the surface quasi-geostrophic (SQG)
equation. They are given by
\[
\partial_t\omega+v\cdot\nabla\omega=0,
\qquad
v=\nabla^\perp(-\Delta)^{-1+\frac{\alpha}{2}}\omega,
\qquad 0\leqslant\alpha\leqslant1.
\]
For $0<\alpha<1$, the corresponding interaction kernel is
\[
K_\alpha(x)=C_\alpha |x|^{-\alpha},
\]
where
\[
C_\alpha
=
\frac{\Gamma(\alpha/2)}
{2^{2-\alpha}\pi
\Gamma\!\left(\frac{2-\alpha}{2}\right)}.
\]
With this normalization, the operator
$(-\Delta)^{-1+\frac{\alpha}{2}}$ has Fourier symbol
$|\xi|^{-2+\alpha}$.

The parameter $\alpha$ controls the singularity of the interaction.
The endpoint $\alpha=1$ corresponds to the SQG equation, while the
Euler equation is recovered in the limit $\alpha\to0^+$, up to an
irrelevant additive constant in the potential kernel.

\medskip

\noindent
\textbf{(iii) The quasi-geostrophic shallow-water equations.}

The quasi-geostrophic shallow-water (QGSW) equation provides a
nonhomogeneous counterpart of the Euler  model. The stream
function is related to the active scalar through the Helmholtz equation
\[
(-\Delta+\varepsilon^2)\psi=\omega,
\qquad
\varepsilon=\tfrac{1}{L_R},
\]
where $L_R$ denotes the Rossby deformation radius. The corresponding
interaction kernel is
\[
K_\varepsilon(x)
=
\frac{1}{2\pi}\mathbf{K}_0(\varepsilon|x|),
\]
where $\mathbf{K}_0$ is the modified Bessel function of the second kind of order
zero.
The Bessel function $\mathbf{K}_0$ admits the integral representation
\[
\mathbf{K}_0(r)
=
\int_0^\infty e^{-r\cosh t}\,dt,
\qquad r>0,
\]
and has the asymptotic behavior
\[
\mathbf{K}_0(r)
=
-\log r+O(1),
\qquad r\to0^+,
\]
whereas
\[
\mathbf{K}_0(r)
\sim
\sqrt{\frac{\pi}{2r}}\,e^{-r},
\qquad r\to\infty.
\]
Near the origin, the kernel therefore behaves like the Euler kernel,
whereas at large distances it decays exponentially. Physically, this
exponential decay reflects the finite Rossby deformation radius, which
limits the range of the interaction between fluid particles. The QGSW
model is widely used in physical oceanography and atmospheric
dynamics to describe mesoscale coherent structures such as oceanic
eddies and atmospheric vortices.
\\
Unlike the Euler and generalized SQG kernels, the QGSW kernel is no
longer homogeneous. The loss of scaling invariance constitutes one of
the principal analytical difficulties in the study of rotating vortex
patches for this model.\\  For
convenience, the three principal examples discussed above are summarized
in the following table.
\begin{table}[H]
\centering
\begin{tabular}{|c|c|c|}
\hline
Model & Stream function & Kernel \\ \hline
Euler &
$(-\Delta)\psi=\omega$ &
$-\dfrac1{2\pi}\log|x|$ \\[2ex]
\hline
gSQG &
$(-\Delta)^{1-\frac{\alpha}{2}}\psi=\omega$ &
$C_\alpha|x|^{-\alpha}$ \\[2ex]
\hline
QGSW &
$(-\Delta+\varepsilon^2)\psi=\omega$ &
$\dfrac1{2\pi}\mathbf{K}_0(\varepsilon|x|)$ \\[2ex]
\hline
\end{tabular}
\end{table}

\section{Patch solutions and rotating patches}

Among the weak solutions of \eqref{AS}, a distinguished class is formed by
\emph{patch solutions}, for which the active scalar is constant on a
moving domain:
\[
\omega(t,x)=\chi_{D_t}(x).
\]
Here $D_t\subset\mathbb R^2$ is a bounded domain whose boundary is
transported by the velocity field. Formally, the transport structure preserves the patch form, while
incompressibility preserves the area of $D_t$. The essential dynamics
are therefore encoded by the evolution of the interface
$\partial D_t$, reducing the original transport equation to a
nonlocal free-boundary problem.\\
The analytical nature of the contour dynamics problem depends strongly
on the singularity of the interaction kernel. For the Euler equation,
the logarithmic kernel leads to the classical vortex-patch theory, in
which the regularity of smooth boundaries is propagated globally in
time. For the gSQG family, the contour equation becomes progressively
more singular as the parameter increases, and its local well-posedness
requires function spaces adapted to the strength of this singularity;
see, for instance, \cite{Gan08}. We refer the reader to the extensive
literature on contour dynamics and patch regularity for further results
and developments.

Within this broad class of patch solutions, we focus on domains undergoing
a rigid rotation about the origin,
\[
D_t=e^{i\Omega t}D_0,
\]
for some angular velocity $\Omega\in\mathbb R$. These are precisely the
$V$-states considered throughout this monograph. Their interest lies in
the fact that the time-dependent free-boundary dynamics reduce to a
stationary nonlinear problem for the shape of a single domain.
\section{Assumptions on the interaction kernel and examples}
We now specify the assumptions on the radial interaction kernel
\(K_0:(0,\infty)\longrightarrow\mathbb R\) used in \eqref{kernel-general}. We assume that it
is of class $C^\infty$ and satisfies the following two assumptions.

\begin{myassumption}{}{A1}
\begin{enumerate}
    \item Complete monotonicity:
The function \(-K_0'\) is a nontrivial completely monotone function,
that is,
\[
(-1)^n\frac{d^n}{dt^n}\big(-K_0'(t)\big)\geqslant0,
\qquad
\forall t>0,\quad \forall  n\in\mathbb{N}.
\]
Equivalently, by Bernstein's theorem, there exists a non-negative
Radon measure \(\mu\) on \([0,\infty)\) such that
\begin{equation}\label{Bernstein}
-K_0'(t)
=
\int_0^\infty
e^{-tx}\,d\mu(x),
\qquad t>0.
\end{equation}
\item Integrability near the origin: There exists \(\alpha\in(0,1)\) such that
\begin{equation}\label{kernel-integrability}
\int_0^{1}
|K_0(t)|\,
t^{-\,\alpha}
\,dt
<
\infty.
\end{equation}
\end{enumerate}
\end{myassumption}

The complete monotonicity assumption is the cornerstone of the unified
approach developed in this chapter. It allows us to represent the
interaction kernel as a superposition of exponentially decaying kernels,
thereby reducing many analytical questions to elementary estimates for
the exponential kernel. Moreover, it encompasses all the principal
examples discussed above, including the Euler, 
generalized surface quasi-geostrophic equation and the quasi-geostrophic
shallow-water models.

We now verify that the interaction kernels associated with all the models
discussed above satisfy these fundamental assumptions.
\subsubsection{The two-dimensional Euler equation}

For the two-dimensional incompressible Euler equation in the whole plane, we have
\[
K_0(t)=-\frac{1}{2\pi}\log t,
\qquad t>0.
\]
Although \(K_0\) does not have a fixed sign, its derivative satisfies
\[
-K_0'(t)
=
\frac{1}{2\pi}\frac1t.
\]
Using the elementary representation
\[
\frac1t
=
\int_0^\infty e^{-tx}\,dx,
\]
we get
\[
-K_0'(t)
=
\frac{1}{2\pi}
\int_0^\infty e^{-tx}\,dx
=
\int_0^\infty e^{-tx}\,d\mu(x),
\]
with
\[
d\mu(x)=\frac{1}{2\pi}\,dx.
\]
Therefore \(-K_0'\) is completely monotone. Moreover, \(K_0\) satisfies
the integrability condition \eqref{kernel-integrability} for every
\(
\alpha\in(0,1).
\)

\subsubsection{The generalized SQG equation}

The generalized surface quasi-geostrophic equation in the plane is
given by
\[
\psi=(-\Delta)^{-1+\frac{\beta}{2}}\omega,
\qquad
0<\beta<1.
\]
The associated kernel is
\[
K(\mathbf x)
=
K_0(|\mathbf x|)
=
c_\beta|\mathbf x|^{-\beta},
\]
where
\[
c_\beta
=
\frac{\Gamma\!\left(\frac{\beta}{2}\right)}
{\pi\,2^{2-\beta}\Gamma\!\left(1-\frac{\beta}{2}\right)}.
\]
Thus
\[
K_0(t)=c_\beta t^{-\beta}.
\]
Then
\[
-K_0'(t)
=
\beta c_\beta t^{-\beta-1}.
\]
Using
\[
\int_0^\infty e^{-tx}x^\beta\,dx
=
\Gamma(\beta+1)t^{-\beta-1},
\]
we obtain
\[
-K_0'(t)
=
\frac{\beta c_\beta}{\Gamma(\beta+1)}
\int_0^\infty e^{-tx}x^\beta\,dx.
\]
Since \(\Gamma(\beta+1)=\beta\Gamma(\beta)\), this becomes
\[
-K_0'(t)
=
\frac{c_\beta}{\Gamma(\beta)}
\int_0^\infty e^{-tx}x^\beta\,dx
=
\int_0^\infty e^{-tx}\,d\mu(x),
\]
where
\[
d\mu(x)
=
\frac{c_\beta}{\Gamma(\beta)}x^\beta\,dx.
\]
Hence \(-K_0'\) is completely monotone. Moreover, the integrability
condition \eqref{kernel-integrability} holds for every
$\alpha\in(0,1-\beta)$. Consequently, all the required assumptions are
satisfied.

\subsubsection{The quasi-geostrophic shallow-water equation}

The quasi-geostrophic shallow-water equation in the whole plane is
defined by
\[
\psi=(-\Delta+\varepsilon^2)^{-1}\omega,
\qquad
\varepsilon>0.
\]
Here \(\varepsilon\) is the inverse Rossby deformation radius. The
corresponding Green kernel is
\[
K(\mathbf x)
=
K_0(|\mathbf x|)
=
\frac{1}{2\pi}
\mathbf K_0(\varepsilon|\mathbf x|),
\]
where \(\mathbf K_0\) denotes the modified Bessel function of the
second kind of order zero.
Using the integral representation of \(\mathbf K_0\),
\[
\mathbf K_0(r)
=
\int_1^\infty
\frac{e^{-rx}}{\sqrt{x^2-1}}\,dx,
\]
we obtain
\[
K_0(t)
=
\frac{1}{2\pi}
\mathbf K_0(\varepsilon t)
=
\frac{1}{2\pi}
\int_1^\infty
\frac{e^{-\varepsilon tx}}{\sqrt{x^2-1}}\,dx.
\]
Therefore
\[
-K_0'(t)
=
\frac{1}{2\pi}
\int_1^\infty
\frac{\varepsilon x\,e^{-\varepsilon tx}}
{\sqrt{x^2-1}}\,dx.
\]
Making the change of variables \(s=\varepsilon x\), we find
\[
-K_0'(t)
=
\frac{1}{2\pi}
\int_\varepsilon^\infty
e^{-ts}
\frac{s}{\sqrt{s^2-\varepsilon^2}}\,ds.
\]
Hence
\[
-K_0'(t)
=
\int_0^\infty e^{-tx}\,d\mu(x),
\]
with
\[
d\mu(x)
=
\frac{1}{2\pi}
\frac{x}{\sqrt{x^2-\varepsilon^2}}
\mathbf 1_{\{x>\varepsilon\}}\,dx.
\]
Thus \(-K_0'\) is completely monotone and the integrability condition is satisfied for any $\alpha\in(0,1)$ 

\section{$V$-states equation via the stream function}

We now derive the equation governing rotating patches directly in terms
of the stream function. Although the dynamics may equivalently be
formulated through the velocity field, the stream-function approach is
particularly well suited to the present unified setting. Indeed, rigid
rotation can be characterized by the constancy of an effective stream
function along the boundary, thereby reducing the $V$-state problem to a
scalar free-boundary equation. Moreover, this formulation involves the
interaction kernel itself, rather than its derivative, which is
advantageous when dealing simultaneously with kernels of different
degrees of singularity.
\\
Suppose now that the patch rotates uniformly with angular velocity
$\Omega$ about the origin,
\[
D_t=e^{i\Omega t}D.
\]
In a rotating frame, the effective stream function is
\[
\Psi_D(z)+\tfrac{\Omega}{2}|z|^2,\quad \Psi_D(z)=\int_{D}K_0(|z-y|)\,dA(y).
\]
The boundary of a rotating patch is therefore characterized by
\begin{align}\label{stream-general}
\Psi_D(z)+\tfrac{\Omega}{2}|z|^2=\mu,
\qquad \forall z\in \partial D,
\end{align}
for some constant \(\mu\in\mathbb R\). Notice that the sign of $\Omega$ is opposite to that in
\eqref{rvortex3}. This difference stems from the convention adopted here
for $\nabla^\perp$, which results in a stream function with the opposite
sign to that used in the preceding chapters. The main result of this chapter is the following.
\begin{mytheorem}{}{Hmidi-Xue}
Let $\alpha\in(0,1)$ and let $m\geqslant 2$ be an integer. Assume that the kernel  $K_0$ satisfies the assumptions \eqref{Bernstein} and \eqref{kernel-integrability}. Then there exists
a family of $m$-fold symmetric rotating vortex patches
$
(V_m)_{m\ge2}
$
for the equation \eqref{stream-general}. Moreover, for each $m$ the branch
\(V_m\) bifurcates from the trivial Rankine vortex
at some  angular velocity
\(
\Omega_m.
\)
In addition, the boundary of the corresponding vortex patch belongs to
the H\"older class
\(
C^{1+\alpha}(\mathbb T).
\)
\end{mytheorem}
Let
\[
\Phi:\mathbb{D}\longrightarrow   D
\]
be the interior conformal mapping normalized by
\[
\Phi(w)=w+f(w),
\qquad |w|<1.
\]
The boundary $\partial D$ is parametrized by
\[
z=\Phi(w),
\qquad w\in\mathbb T.
\]
Thus the boundary equation becomes
\begin{align}\label{eq-ro56}
F_0(\Omega,\Phi)(w):=\Psi_\Phi(w)+\tfrac{\Omega}{2}|\Phi(w)|^2=\mu,
\qquad w\in\mathbb T,
\end{align}
where
\[
\Psi_\Phi(w)
:=
\int_{D_\Phi}
K_0(|\Phi(w)-y|)\,dA(y).
\]
and $D_\Phi:=D=\Phi(\mathbb{D}).$ In addition, the constant $\mu$ is   given by the average
$$
\mu=\frac{1}{2\pi}\int_{\mathbb{T}}F_0(\Omega,\Phi)(e^{i \theta})d\theta:=\langle F_0(\Omega,\Phi)\rangle\,.
$$
Hence the rotating patch problem is therefore reduced to
\begin{align}\label{eq-ro66}
F(\Omega,\Phi)(w)
:=
F_0(\Omega,\Phi)(w)-\langle F_0(\Omega,\Phi)\rangle=0,
\qquad w\in\mathbb T.
\end{align}

\subsection*{The trivial branch.}

The unit disk corresponds to the conformal mapping
\[
\Phi_0(w)=w.
\]
In this case,
\[
\Psi_{\Phi_0}(w)
=
\int_{\mathbb D}
K_0(|w-y|)\,dA(y),
\qquad w\in\mathbb T.
\]
By radial symmetry, \(\Psi_{\Phi_0}(w)\) is constant on
\(\mathbb T\). Moreover,
\[
|\Phi_0(w)|^2=|w|^2=1.
\]
Therefore, we infer from \eqref{eq-ro66}

\[
F(\Omega,\Phi_0)=0,
\qquad \forall \Omega\in\mathbb R.
\]
Thus the unit disk generates the trivial branch of rotating solutions.

\section{Boundary formula for the stream function}
The $V$-state equation obtained above involves the stream function
$\Psi_\Phi$ evaluated on the boundary of the patch. In its original
form, however, $\Psi_\Phi$ is given by an integral over the moving domain
$D_\Phi$, which is not well suited to the functional-analytic framework
required for bifurcation. It is therefore convenient to rewrite this
quantity entirely in terms of the boundary parametrization.

The radial structure of the interaction kernel allows us to convert the
area integral into a boundary integral by means of the divergence
theorem. The resulting formula has two important advantages: the
dependence on the domain is encoded solely through the conformal map
$\Phi$, and all integrations are performed over the fixed reference
circle $\mathbb T$. This representation will be the starting point for
the regularity and linearization analysis of the $V$-state equation.\\
To state the formula, we introduce a radial primitive of the interaction
kernel,
\[
G_0(r):=\int_0^1 sK_0(rs)\,ds,
\qquad r>0.
\]
\begin{myproposition}{}{prop-Bound-int}
Let $\Psi_\Phi$ be defined as in \eqref{stream-func1}. Then
\[
\Psi_\Phi(w)
=
\int_{\mathbb T}
G_0\bigl(|\Phi(w)-\Phi(\xi)|\bigr)
\textnormal{Re}\left(
\bigl(\Phi(\xi)-\Phi(w)\bigr)
\overline{\xi\Phi'(\xi)}
\right)
\,|d\xi|.
\]
\end{myproposition}

\begin{proof}
Fix \(w\in\mathbb T\) and set
\[
x:=\Phi(w).
\]
For \(y\neq x\), define the vector field
\[
F_x(y):=G_0(|y-x|)(y-x).
\]
Writing \(r=|y-x|\), we compute
\[
\operatorname{div}_y F_x(y)
=
2G_0(r)+rG_0'(r).
\]
Since, by change of variables, 
\[
r^2G_0(r)=\int_0^r sK_0(s)ds=: H_0(r),
\]
and differentiation leads to
\[
2rG_0(r)+r^2G_0'(r)=H_0'(r)=rK_0(r).
\]
Consequently,
\[
2G_0(r)+rG_0'(r)=K_0(r).
\]
Thus,
\[
\operatorname{div}_yF_x(y)=K_0(|y-x|).
\]
Because \(x\in\partial D_\Phi\), we first remove a small neighborhood of
the singular point. For \(\varepsilon>0\), set
\[
D_{\Phi,\varepsilon}
:=
D_\Phi\setminus \overline{B(x,\varepsilon)}.
\]
Applying the divergence theorem on \(D_{\Phi,\varepsilon}\), we obtain
\[
\int_{D_{\Phi,\varepsilon}}
K_0(|y-x|)\,dA(y)
=
\int_{\partial D_\Phi\setminus B(x,\varepsilon)}
F_x(y)\cdot n(y)\,d\sigma(y)
+
\int_{\partial B(x,\varepsilon)\cap D_\Phi}
F_x(y)\cdot n_\varepsilon(y)\,d\sigma(y),
\]
where \(n\) denotes the outward unit normal to \(D_\Phi\), while
\(n_\varepsilon\) denotes the outward unit normal to
\(D_{\Phi,\varepsilon}\) along the circular part of the boundary.
On \(\partial B(x,\varepsilon)\), we have
\[
|F_x(y)|
=
\frac{|H_0(\varepsilon)|}{\varepsilon},
\]
and hence
\[
\left|
\int_{\partial B(x,\varepsilon)\cap D_\Phi}
F_x(y)\cdot n_\varepsilon(y)\,d\sigma(y)
\right|
\leqslant
C|H_0(\varepsilon)|.
\]
Since
\[
|H_0(\varepsilon)|
\leqslant
\int_0^\varepsilon s|K_0(s)|\,ds
\longrightarrow0
\qquad\text{as }\varepsilon\to0,
\]
the circular boundary contribution vanishes. Letting
\(\varepsilon\to0\), we obtain
\[
\Psi_\Phi(w)
=
\int_{\partial D_\Phi}
G_0(|y-x|)(y-x)\cdot n(y)\,d\sigma(y).
\]
We now parametrize the boundary by
\[
y=\Phi(\xi),
\qquad
\xi=e^{i\theta}\in\mathbb T.
\]
Since the parametrization is counterclockwise,
\[
\frac{d}{d\theta}\Phi(e^{i\theta})
=
i\xi\Phi'(\xi),
\]
and the outward normal measure satisfies
\[
n(\Phi(\xi))\,d\sigma
=
-i\frac{d}{d\theta}\Phi(e^{i\theta})\,d\theta
=
\xi\Phi'(\xi)\,d\theta.
\]
Identifying the Euclidean scalar product in \(\mathbb C\) with
\[
a\cdot b=\textnormal{Re}(a\overline b),
\]
we deduce
\[
(y-x)\cdot n(y)\,d\sigma(y)
=
\textnormal{Re}\left(
\bigl(\Phi(\xi)-\Phi(w)\bigr)
\overline{\xi\Phi'(\xi)}
\right)\,|d\xi|.
\]
Substituting this identity into the boundary integral gives
\[
\Psi_\Phi(w)
=
\int_{\mathbb T}
G_0\bigl(|\Phi(w)-\Phi(\xi)|\bigr)
\textnormal{Re}\left(
\bigl(\Phi(\xi)-\Phi(w)\bigr)
\overline{\xi\Phi'(\xi)}
\right)
\,|d\xi|.
\]

\end{proof}
\section{Linearization of the nonlinear functional}

In order to investigate the local bifurcation of rotating patches from the
trivial branch, we study the linearization of the nonlinear functional
defined in \eqref{eq-ro66}. As in the classical Crandall--Rabinowitz
framework, the linearized operator will provide the spectral information
required to identify the bifurcation values of the angular velocity.
\\
Recall that the nonlinear functional is defined by
\[
F(\Omega,\Phi)
=
F_0(\Omega,\Phi)
-
\langle F_0(\Omega,\Phi)\rangle,
\]
where
\[
F_0(\Omega,\Phi)(w)
=
\Psi_\Phi(w)
+
\frac{\Omega}{2}|\Phi(w)|^2,
\qquad
w\in\mathbb T,
\]
and
\begin{align}\label{stream-func1}
\Psi_\Phi(w)
=
\int_{D_\Phi}
K_0(|\Phi(w)-y|)\,dA(y).
\end{align}
The subtraction of the average simply removes the arbitrary additive
constant appearing in the integrated formulation of the boundary
equation. Since the averaging operator is linear and continuous, the
Fréchet derivative of \(F\) is readily expressed in terms of that of
\(F_0\) as
\[
\partial_\Phi F(\Omega,\Phi)h
=
\partial_\Phi F_0(\Omega,\Phi)h
-
\Big\langle
\partial_\Phi F_0(\Omega,\Phi)h
\Big\rangle.
\]
Consequently, it suffices to compute the derivative of \(F_0\).

The functional \(F_0\) consists of two contributions. The first one,
\(\Psi_\Phi\), represents the stream function evaluated on the boundary
of the patch, while the second one corresponds to the centrifugal
potential in the rotating frame. Their linearizations are of a different
nature.

The quadratic term admits the straightforward expansion
\[
\partial_\varepsilon
\left(
\tfrac{\Omega}{2}
|\Phi+\varepsilon h|^2
\right)_{\varepsilon=0}
=
\Omega\,
\operatorname{Re}
\big(
h\,\overline{\Phi}
\big).
\]
The linearization of the stream function is considerably more delicate,
since the dependence on the conformal mapping is twofold. On the one
hand, the observation point \(\Phi(w)\) moves under the perturbation,
leading to the variation of the kernel with respect to its first
argument. On the other hand, the integration domain \(D_\Phi\) itself is
deformed, giving rise to a genuine shape derivative of the volume
integral.

To derive a concise expression for the Fréchet derivative, we shall
combine the interior conformal parametrization of the patch with
Hadamard's variational formula. This approach leads to a remarkably
compact representation of the linearized operator, separating
geometrically the contribution of the moving observation point from that
of the domain deformation.


\subsection{Linearization at a general  state}
We first recall the version of Hadamard's formula that will be used
below.
\begin{mytheorem}{Hadamard's Variational Formula}{Hadamard}
Let \(D\subset\mathbb R^2\) be a bounded domain with \(C^1\) boundary and
let \(v:\mathbb R^2\to\mathbb R^2\) be a smooth vector field.
For sufficiently small \(\varepsilon\), define the perturbed domain

\[
D_\varepsilon
=
\{x+\varepsilon v(x):x\in D\}.
\]
Assume that
\[
J(\varepsilon)
=
\int_{D_\varepsilon}
f(\varepsilon,x)\,dx,
\]
where \(f\) is continuously differentiable with respect to
\(\varepsilon\).
Then, we have 
\[
{
J^\prime(0)
=
\int_D
\partial_\varepsilon f(0,x)\,dx
+
\int_{\partial D}
f(0,x)
v_n(x)
\,d\sigma(x),
}
\]
where
\(
v_n=v\cdot n
\)
is the normal component of the deformation field.
\end{mytheorem}
\begin{proof}
Using the change of variables
\[
x=T_\varepsilon(y)
=
y+\varepsilon v(y),
\]
we obtain
\[
J(\varepsilon)
=
\int_D
f(\varepsilon,T_\varepsilon(y))
\det DT_\varepsilon(y)
\,dy.
\]
Since
\[
DT_\varepsilon
=
\textnormal{Id}+\varepsilon Dv,
\]
we have
\[
\det DT_\varepsilon
=
1+\varepsilon\operatorname{div}v
+
o(\varepsilon).
\]
Therefore
\[
\begin{aligned}
J'(0)
=
\int_D
\partial_\varepsilon f(0,y)
\,dy
+
\int_D
\nabla f(0,y)\cdot v(y)
\,dy
+
\int_D
f(0,y)\operatorname{div}v(y)
\,dy.
\end{aligned}
\]
The last two terms combine into
\[
\operatorname{div}(fv)
=
\nabla f\cdot v
+
f\operatorname{div}v.
\]
Hence
\[
J'(0)
=
\int_D
\partial_\varepsilon f dy
+
\int_D
\operatorname{div}(fv) dy.
\]
Finally, the divergence theorem yields
\[
\int_D
\operatorname{div}(fv)
=
\int_{\partial D}
fv_n\,d\sigma.
\]
This proves the result.
\end{proof}
The following result provides an explicit expression for the Fréchet
derivative of the nonlinear functional $F_0$ defined by \eqref{eq-ro56}. More precisely, we have
the following.
\begin{myproposition}{}{Lin-Gen}
Let \(\Phi\) be a sufficiently smooth conformal mapping and let
\(h\) be an admissible perturbation. Then the Gateaux derivative of
\(F_0\) at \(\Phi\) is given by
\[
\begin{aligned}
\partial_\Phi F_0(\Omega,\Phi)[h](w)
={}&
\int_{\mathbb T}
K_0\!\left(|\Phi(w)-\Phi(\xi)|\right)
\operatorname{Re}
\left(
\big[h(\xi)-h(w)\big]\overline{\xi\Phi'(\xi)}
\right)
\, |d\xi|\\
&+\Omega\operatorname{Re}
\Big(
\overline{\Phi(w)}
\,h(w)
\Big),
\end{aligned}
\]
and
\[
\partial_\Phi F(\Omega,\Phi)[h]
=
\partial_\Phi F_0(\Omega,\Phi)[h]
-
\Big\langle
\partial_\Phi F_0(\Omega,\Phi)[h]
\Big\rangle.
\]
\end{myproposition}

\begin{proof}
Recall that 
\[
F_0(\Omega,\Phi)(w)
=
\Psi_\Phi(w)+\tfrac{\Omega}{2}|\Phi(w)|^2,
\qquad w\in\mathbb{T},
\]
where
\[
\Psi_\Phi(w)
=
\int_{D_\Phi}
K_0(|\Phi(w)-y|)\,dA(y).
\]
Here \(D_\Phi=\Phi(\mathbb D)\), and \(\Phi:\mathbb D\to D_\Phi\) is the
interior conformal mapping.
Let
\[
\Phi_\varepsilon=\Phi+\varepsilon h.
\]
Then
\[
\partial_\Phi F_0(\Omega,\Phi)[h]=\partial_\varepsilon F_0(\Omega,\Phi+\varepsilon h)_{|\varepsilon=0}
=
\partial_\varepsilon{\Psi_{\Phi+\varepsilon h}}_{|\varepsilon=0}
+
\Omega\operatorname{Re}\big(h\overline{\Phi}\big)
.
\]
Thus it remains to compute \(\partial_\Phi\Psi_\Phi h\).
Fix \(w\in\mathbb T\) and set
\[
x_\varepsilon:=\Phi_\varepsilon(w),
\qquad
x:=\Phi(w).
\]
The deformation of the domain \(D_\Phi\) induced by
\(\Phi_\varepsilon=\Phi+\varepsilon h\) has boundary velocity
\[
V_h(\Phi(\xi))=h(\xi).
\]
By the Hadamard transport formula,
\[
\begin{aligned}
\partial_\Phi\Psi_\Phi[h](w)
={}&
h(w)\cdot
\int_{D_\Phi}
\nabla_xK_0(|x-y|)\,dA(y)
\\
&+
\int_{\partial D_\Phi}
K_0(|x-y|)
V_h(y)\cdot n_\Phi(y)\,d\sigma(y).
\end{aligned}
\]
Since the kernel depends only on \(x-y\),
\[
\nabla_xK_0(|x-y|)
=
-\nabla_yK_0(|x-y|).
\]
Consequently, applying the divergence theorem in the \(y\)-variable,
\[
\int_{D_\Phi}
\nabla_xK_0(|x-y|)\,dA(y)
=
-\int_{\partial D_\Phi}
K_0(|x-y|)n_\Phi(y)\,d\sigma(y).
\]
It follows that
\[
\partial_\Phi\Psi_\Phi[h](w)
=
\int_{\partial D_\Phi}
K_0(|x-y|)
\bigl(V_h(y)-h(w)\bigr)\cdot n_\Phi(y)\,d\sigma(y).
\]
We now parametrize the boundary counterclockwise by
\[
y=\Phi(\xi),
\qquad
\xi=e^{i\theta}\in\mathbb T.
\]
The outward normal measure is
\[
n_\Phi(\Phi(\xi))\,d\sigma
=
\xi\Phi'(\xi)\,|d\xi|.
\]
Therefore, 
we obtain
\[
\begin{aligned}
\bigl(V_h(\Phi(\xi))-h(w)\bigr)
\cdot n_\Phi(\Phi(\xi))\,d\sigma
&=
\textnormal{Re}\left(
\bigl(h(\xi)-h(w)\bigr)
\overline{\xi\Phi'(\xi)}
\right)|d\xi|.
\end{aligned}
\]
Substitution into the preceding boundary formula yields
\begin{align}\label{psi-Id}
\partial_\Phi\Psi_\Phi[h](w)
=
\int_{\mathbb T}
K_0\bigl(|\Phi(w)-\Phi(\xi)|\bigr)
\textnormal{Re}\left(
\bigl(h(\xi)-h(w)\bigr)
\overline{\xi\Phi'(\xi)}
\right)
\,|d\xi|,
\end{align}
which completes the proof.
\end{proof}

\subsection{Linearization at the trivial equilibrium}

We now specialize the general linearization formula to the trivial
equilibrium corresponding to the unit disk. At this state, the rotational
invariance of the kernel implies that the linearized operator reduces to
a convolution operator on the unit circle. Our result reads as follows.

\begin{mycorollary}{}{cor-linearization-id}
Define 
\[ 
\lambda_1
=
\int_{0}^{2\pi}
K_0\!\left(
2\left|\sin\tfrac{\theta}{2}\right|
\right)
\cos\theta\,d\theta.
\]
Then,
\[
\begin{aligned}
\partial_\Phi F_0(\Omega,\mathrm{Id})[h](w)
=&
(\Omega-\lambda_1)\operatorname{Re}\!\big(\overline w h(w)\big)
+
\int_{\mathbb T}
K_0(|w-\xi|)
\operatorname{Re}\!\big(\overline\xi\,h(\xi)\big)
\, |d\xi|\,,
\end{aligned}
\]
and
\begin{align*}
\partial_\Phi F(\Omega,\mathrm{Id})[h]
&=
\partial_\Phi F_0(\Omega,\mathrm{Id})[h]
-
\Big\langle \partial_\Phi F_0(\Omega,\mathrm{Id})[h]\Big\rangle.
\end{align*}
\end{mycorollary}

\begin{proof}
For the identity map \(\Phi(w)=w\), we have in view of \eqref{psi-Id}
\begin{align*}
\partial_\Phi\Psi_{\textnormal{Id}}[h](w)
&=
\int_{\mathbb T}
K_0\bigl(|w-\xi|\bigr)
\textnormal{Re}\left(
\bigl(h(\xi)-h(w)\bigr)
\overline{\xi}
\right)
\,|d\xi|\\
&=-\textnormal{Re} \left( h(w)\int_{\mathbb T}
K_0\bigl(|w-\xi|\bigr)
\overline{\xi}
\,|d\xi|\right)+\int_{\mathbb T}
K_0\bigl(|w-\xi|\bigr)
\textnormal{Re}\left(
h(\xi)
\overline{\xi}
\right)
\,|d\xi|,
\end{align*}
As \(w\in\mathbb T\), then making a change of variables gives
$$\int_{\mathbb T}
K_0\bigl(|w-\xi|\bigr)
\overline{\xi}
\,|d\xi|=\overline{w} \int_{\mathbb T}
K_0\bigl(|1-\zeta|\bigr)
\overline{\zeta}
\,|d\zeta|$$
Writing \(\zeta=e^{i\theta}\), \(0\le\theta\le2\pi\), gives
\[
d\zeta
=
ie^{i\theta}\,d\theta,
\]
so that
\[
\lambda_1=\int_{\mathbb T}
K_0\bigl(|1-\zeta|\bigr)
\overline{\zeta}
\,|d\zeta|
=
\int_{0}^{2\pi}
K_0\!\left(
2\left|\sin\tfrac{\theta}{2}\right|
\right)
e^{-i\theta}\,d\theta.
\]
Since
\(\theta\mapsto 
K_0\!\left(
2\left|\sin\frac{\theta}{2}\right|
\right)
\)
is an even function, the sine contribution vanishes, and we conclude that
\[
\lambda_1=
\int_{0}^{2\pi}
K_0\!\left(
2\left|\sin\tfrac{\theta}{2}\right|
\right)
\cos\theta\,d\theta.
\]
Substituting into the general formula of
Proposition~\ref{prop:Lin-Gen}, we get
\[
\begin{aligned}
\partial_\Phi F_0(\Omega,\mathrm{Id})[h](w)
={}&
\operatorname{Re}
\Big(
(\Omega-\lambda_1)\overline w\, h(w)
\Big)
+
\int_{\mathbb T}
K_0(|w-\xi|)
\operatorname{Re}
\left(
h(\xi)\overline{\xi}
\right)
\,|d\xi|.
\end{aligned}
\]
Hence
\[
\begin{aligned}
\partial_\Phi F_0(\Omega,\mathrm{Id})[h](w)
={}&
(\Omega\textcolor{red}{-\lambda_1}) 
\operatorname{Re}\!\big(\overline w h(w)\big)
+
\int_{\mathbb T}
K_0(|w-\xi|)
\operatorname{Re}\!\big(h(\xi)\overline\xi\big)
\,|d\xi|.
\end{aligned}
\]
Finally, since
\[
F=F_0-\langle F_0\rangle,
\]
and the averaging operator is linear, we obtain
\[
\partial_\Phi F(\Omega,\mathrm{Id})[h]
=
\partial_\Phi F_0(\Omega,\mathrm{Id})[h]
-
\Big\langle \partial_\Phi F_0(\Omega,\mathrm{Id})[h]\Big\rangle.
\]
This completes the proof.
\end{proof}

\section{Functional tools and regularity properties}

The bifurcation analysis requires a functional framework that is compatible with the conformal parametrization of the boundary and provides sufficient regularity to control the nonlinear integral operators. Since the admissible perturbations arise from conformal deformations of the unit disk, they are naturally described as traces on $\mathbb T$ of functions holomorphic in the unit disk. We therefore work in Hölder spaces endowed with the appropriate Fourier structure.
Let $0<\alpha<1$. We define the spaces
\begin{align}\label{X-alpha}
X_\alpha
:=
\left\{
h\in C^{1+\alpha}(\mathbb T)
:
h(w)=\sum_{n\geqslant 2}h_n w^n,\quad h_n\in\mathbb R
\right\}
\end{align}
and
\begin{align}\label{Y-alpha}
\nonumber Y_{\alpha,0}
&:=
\Big\{
g\in C^{1+\alpha}(\mathbb T)
:
g(w)=\sum_{n\geqslant 0}g_n\,\operatorname{Re}(w^n),
\, g_n\in\mathbb R
\Big\},\\
\, Y_{\alpha}&=\big\{h\in Y_{\alpha,0},\, g_0=0\big\}.
\end{align}
For $\varepsilon>0$, and $f\in X_\alpha$ we define the open ball
$$
B_\varepsilon^\alpha(f):=\{h\in X_\alpha,\, \|h-f\|_{C^{1+\alpha}(\mathbb{T})} {<} \varepsilon\}, 
$$
and 
$$
B_\varepsilon^\alpha(\mathrm{Id}):=\mathrm{Id}+B_\varepsilon^\alpha(0), 
$$
while $\mathrm{Id}\not\in X_{\alpha}$. 
We identify each $f\in X_\alpha$ with its holomorphic extension
to $\mathbb D$. Choose $\varepsilon>0$ so small that
\[
 q_\varepsilon
 :=
 \sup_{f\in B_\varepsilon^\alpha(0)}
 \|f'\|_{L^\infty(\overline{\mathbb D})}
 <1.
\]
With the standard $C^{1+\alpha}$ norm, one has
$q_\varepsilon\leqslant\varepsilon$, so that it is enough to assume
$0<\varepsilon<1$.
\\
Let $\Phi=\mathrm{Id}+f\in
B_\varepsilon^\alpha(\mathrm{Id})$. Then, for every
$z,\zeta\in\overline{\mathbb D}$,
\[
 (1-q_\varepsilon)|z-\zeta|
 \leqslant |\Phi(z)-\Phi(\zeta)|
 \leqslant (1+q_\varepsilon)|z-\zeta|.
\]
Indeed,
\[
 f(z)-f(\zeta)
 =
 (z-\zeta)\int_0^1
 f'\bigl(\zeta+t(z-\zeta)\bigr)\,dt,
\]
and therefore
\[
 |f(z)-f(\zeta)|
 \leqslant q_\varepsilon |z-\zeta|.
\]
Consequently, $\Phi$ is injective on
$\overline{\mathbb D}$, hence univalent in $\mathbb D$,
and $\Phi(\mathbb T)$ is a $C^{1+\alpha}$ Jordan curve.
In particular, $\Phi$ is uniformly bi-Lipschitz on
$\overline{\mathbb D}$.
\\
Let \(\gamma\in[0,1)\), and let \(K_0:(0,\infty)\to\mathbb R\) be locally integrable and satisfy
\[
\int_0^1 |K_0(r)|\,r^{-\gamma}\,dr<\infty.
\]
Define the radial averaging of $K_0$
\[
G_0(r)
:=
\int_0^1 sK_0(rs)\,ds
=
\frac{1}{r^2}\int_0^r tK_0(t)\,dt,
\qquad 0<r\leqslant 1.
\]
Then \(G_0\) is locally absolutely continuous on \((0,1]\), and
\[
2G_0(r)+rG_0'(r)=K_0(r)
\]
for almost every \(r\in(0,1)\).
\\
The regularity analysis of the nonlinear functional requires quantitative control of the derivatives of the interaction kernel near the origin. Under the complete monotonicity assumption on $-K_0'$, the weighted integrability of $K_0$ automatically propagates to all its higher derivatives. Moreover, this structure is preserved by the radial averaging defining $G_0$. The following lemma collects these properties and provides the estimates that will be repeatedly used in the sequel.
\begin{mylemma}{}{k_0to G_0}
Let {$0\leqslant \gamma<1$},  $L>0,$ and  $K_0\in C^\infty((0,\infty))$ such that $-K_0'$ is completely monotone. 
Then, the following holds true.
\begin{enumerate}
    \item For every integer $k\geqslant1$,

\[
\int_0^L
|K_0^{(k)}(r)|\,r^{k-\gamma}\,dr
\leqslant 
\frac{\Gamma(k+1-\gamma)}
{\Gamma(1-\gamma)}
\int_0^L
\big(K_0(r)-K_0(L)\big)\,r^{-\gamma}\,dr.
\]
    \item The function $-G_0^\prime$ is also completely monotone and for every integer $k\geqslant 1$,
\[
\int_0^L
|G_0^{(k)}(r)|\,r^{k-\gamma}\,dr
\leqslant
\frac{\Gamma(k+1-\gamma)}
{(1+\gamma)\Gamma(1-\gamma)}
\int_0^L
\big(K_0(r)-K_0(L)\big)\,r^{-\gamma}\,dr,
\]
and
$$
\int_0^L
|G_0(r)|\,r^{-\gamma}\,dr\leqslant 
\frac{1}{1+\gamma}
\int_0^L|K_0(t)|t^{-\gamma}\,dt.$$
\item Let  $C\geqslant c>0$ and define for every integer $k\geqslant0$, 
\[
\mathcal M_k(r)
:=
\sup_{\rho\in[cr,Cr]}
|G_0^{(k)}(\rho)|.
\]
Then
\[
\mathcal M_0(r)
\le
2|G_0(cr)|+|G_0(Cr)|,
\]
and, for every integer $k\geqslant 1$,
\[
\mathcal M_k(r)
=
|G_0^{(k)}(cr)|.
\]
\end{enumerate}
\end{mylemma}

\begin{proof}
{\bf 1.} 
Let  $k\geqslant 1$. By replacing \(K_0\) with \(K_0-K_0(L)\), we may assume, without loss of generality, that \(-K_0'\) is completely monotone and that
\[
K_0(L)=0.
\]
Since $-K_0'$ is completely monotone,
\[
(-1)^kK_0^{(k)}(r)\geqslant 0,
\qquad r>0,
\]
and therefore
\[
|K_0^{(k)}(r)|
=
(-1)^kK_0^{(k)}(r).
\]
Set
\[
I_k
:=
\int_0^L
(-1)^kK_0^{(k)}(r)\,
r^{k-\gamma}\,dr.
\]
Then
\[
I_k
=
-\int_0^L
(-1)^{k-1}K_0^{(k)}(r)\,
r^{k-\gamma}\,dr.
\]
Integrating by parts,
\[
\begin{aligned}
I_k
&=
-\Bigl[
(-1)^{k-1}K_0^{(k-1)}(r)
r^{k-\gamma}
\Bigr]_0^L\\
&\qquad+
(k-\gamma)
\int_0^L
(-1)^{k-1}
K_0^{(k-1)}(r)
r^{k-1-\gamma}\,dr.
\end{aligned}
\]
Since
\[
\forall k\geqslant 1,\quad (-1)^{k-1}K_0^{(k-1)}(L)\geq0,
\]
the boundary contribution at $L$ is nonpositive. Moreover, we claim that 
\[
\lim_{r\to0}
r^{k-\gamma}
K_0^{(k-1)}(r)
=0\,.
\]
Indeed, for $k=1,$ as $K_0$ is decreasing and positive on $[0,L]$ we get
\begin{align*}
    \int_{\frac{r}{2}}^r s^{-\gamma} K_0(s) ds
    &\geqslant K_0(r)\int_{\frac{r}{2}}^r s^{-\gamma}  ds\\
    &\geqslant C(\gamma)  r^{1-\gamma}K_0(r)\geqslant 0.
\end{align*}
As 
$$\lim_{r\to 0}  \int_{\frac{r}{2}}^r s^{-\gamma} K_0(s) ds=0,
$$ then we get
$$
\lim_{r\to0} r^{1-\gamma}K_0(r)=0.
$$
The result for $k\geqslant2$ can be done by induction. 
Indeed, repeated integration by parts imply that
\[
\lim_{r\to0}r^{j-\gamma}
K_0^{(j-1)}(r)=0,
\qquad
1\leqslant j\leqslant k.
\]
Hence
\[
I_k
\leqslant
(k-\gamma)
I_{k-1}.
\]
Iterating,
\[
I_k
\leqslant
(k-\gamma)(k-1-\gamma)\cdots(1-\gamma)I_0.
\]
Since
\[
I_0
=
\int_0^L
K_0(r)r^{-\gamma}\,dr,
\]
and
\[
\prod_{j=1}^k(j-\gamma)
=
\frac{\Gamma(k+1-\gamma)}
{\Gamma(1-\gamma)},
\]
we conclude that
\[
\int_0^L
|K_0^{(k)}(r)|\,r^{k-\gamma}\,dr
=
I_k
\leqslant
\frac{\Gamma(k+1-\gamma)}
{\Gamma(1-\gamma)}
\int_0^L
K_0(r)\,r^{-\gamma}\,dr.
\]
{\bf{2.}}
{We first prove that $-G_0'$ is completely monotone.
By the Bernstein representation \eqref{Bernstein} and Tonelli's theorem,
\[
 \begin{aligned}
 -G_0'(r)
 &=
 \int_0^1s^2\bigl(-K_0'(rs)\bigr)\,ds  \\
 &=
 \int_{[0,\infty)}
 \int_0^1s^2e^{-rsx}\,ds\,d\mu(x).
 \end{aligned}
\]
For each fixed $s\in(0,1]$ and $x\geq0$, the function
$r\mapsto e^{-rsx}$ is completely monotone. Hence
$-G_0'$ is a positive superposition of completely monotone
functions and is therefore completely monotone.}

We reduce the estimate to the case $L=1$. Indeed, for the general case $L>0$ we make a dilation argument by setting
\[
\widetilde K_0(x):=K_0(Lx),
\qquad
\widetilde G_0(x):=G_0(Lx),
\qquad 0<x\leq1.
\]
We leave the details for the readers. 
{ 
For $k=0$, since $K_0$ is not assumed to be nonnegative, we
first use the triangle inequality in the definition of $G_0$:
\[
 |G_0(r)|
 =
 \left|\int_0^1 sK_0(rs)\,ds\right|
 \leqslant
 \int_0^1s|K_0(rs)|\,ds.
\]
Since the integrand below is nonnegative, Tonelli's theorem and
the change of variables $t=rs$ yield
\[
\begin{aligned}
 \int_0^1|G_0(r)|r^{-\gamma}\,dr
 &\leqslant
 \int_0^1\int_0^1
 s|K_0(rs)|r^{-\gamma}\,ds\,dr \\
 &=
 \int_0^1s^\gamma
 \int_0^s|K_0(t)|t^{-\gamma}\,dt\,ds \\
 &\leqslant
 \frac1{1+\gamma}
 \int_0^1|K_0(t)|t^{-\gamma}\,dt.
\end{aligned}
\]
}
Assume now that $k\geqslant1$. Differentiating under the integral sign yields
\[
G_0^{(k)}(r)
=
\int_0^1
s^{k+1}
K_0^{(k)}(rs)\,ds.
\]
Since $-K_0'$ is completely monotone,
\[
(-1)^kK_0^{(k)}\geqslant0,
\]
and hence
\[
|G_0^{(k)}(r)|
=
\int_0^1
s^{k+1}
|K_0^{(k)}(rs)|\,ds.
\]
Applying Fubini's theorem and changing variables $t=rs$, we obtain
\[
\begin{aligned}
\int_0^1
|G_0^{(k)}(r)|
r^{k-\gamma}\,dr
&=
\int_0^1
\int_0^1
s^{k+1}
|K_0^{(k)}(rs)|
r^{k-\gamma}\,ds\,dr\\
&=
\int_0^1
s^\gamma
\int_0^s
|K_0^{(k)}(t)|
t^{k-\gamma}\,dt\,ds\\
&=
\frac1{1+\gamma}
\int_0^1
|K_0^{(k)}(t)|
t^{k-\gamma}
(1-t^{1+\gamma})\,dt\\
&\leqslant\frac1{1+\gamma}
\int_0^1
|K_0^{(k)}(t)|
t^{k-\gamma}\,dt
\end{aligned}
\]
By the item \textbf{1}, this completes the proof of the desired result.
\\
{\bf 3.}
Let's now move to the last point.
Since $-G_0'$ is completely monotone, one has
 for every $k\ge1$,
\[
|G_0^{(k)}(\rho)|
=
(-1)^kG_0^{(k)}(\rho).
\]
Moreover,
\[
\frac{d}{d\rho}|G_0^{(k)}(\rho)|
=
(-1)^kG_0^{(k+1)}(\rho)
=
-(-1)^{k+1}G_0^{(k+1)}(\rho)
\leqslant 0.
\]
Hence the function
\[
\rho\longmapsto|G_0^{(k)}(\rho)|
\]
is nonincreasing and it  follows that
\[
\mathcal M_k(r)
=
\sup_{\rho\in[cr,Cr]}
|G_0^{(k)}(\rho)|
=
|G_0^{(k)}(cr)|,
\qquad k\geqslant 1.
\]
For $k=0$, the same argument shows that
\[
G_0'(\rho)\leqslant0,
\]
so that $G_0$ is nonincreasing. Alternatively, without using the sign of $G_0$, for every $\rho\in[cr,Cr]$ we write
\[
G_0(\rho)
=
G_0(cr)+\int_{cr}^{\rho}G_0'(t)\,dt.
\]
Since $G_0'\leqslant0$, we have
\[
\int_{cr}^{Cr}|G_0'(t)|\,dt
=
-\int_{cr}^{Cr}G_0'(t)\,dt
=
G_0(cr)-G_0(Cr).
\]
Therefore,
\[
\begin{aligned}
|G_0(\rho)|
&\leqslant
|G_0(cr)|
+
\int_{cr}^{Cr}|G_0'(t)|\,dt\\
&=
|G_0(cr)|+G_0(cr)-G_0(Cr)\\
&\leqslant
2|G_0(cr)|+|G_0(Cr)|.
\end{aligned}
\]
Taking the supremum over $\rho\in[cr,Cr]$ yields
\[
\mathcal M_0(r)
\leqslant
2|G_0(cr)|+|G_0(Cr)|.
\]
This completes the proof.
\end{proof}

The next step is to establish the regularity properties required to place the $V$-state equation within the classical bifurcation framework. In particular, we need to verify that the nonlinear functional $F_0$ defined in
\eqref{eq-ro56} is well defined on a neighborhood of the unit disk and has the necessary differentiability with respect to both the conformal parametrization and the angular velocity. The estimates developed above lead to the following result.
\begin{myproposition}{}{Theorem-REg-smooth}
Let $\varepsilon\in(0,1)$, and assume that $K_0$ satisfies the assumptions \eqref{Bernstein} and \eqref{kernel-integrability}. Then the functional $F_0$ defined in \eqref{eq-ro56} enjoys the following properties.
\begin{enumerate}

\item The map
\[
F:\mathbb{R}\times B_\varepsilon^\alpha(\textnormal{Id})\longrightarrow Y_{\alpha,0} 
\]
is well defined.

\item The map $F_0:\mathbb{R}\times B_\varepsilon^\alpha(\textnormal{Id})\longrightarrow Y_{\alpha}$ is of class $C^1$.

\item Moreover, the mixed derivative
\(
\partial_\Omega\partial_\Phi F_0
\)
exists and is continuous.
\end{enumerate}
\end{myproposition}

\begin{proof}
{\bf{1.}} For the proof, it suffices to show that, for every $\Phi\in B_\varepsilon^\alpha(\textnormal{Id})$, the corresponding stream function $\Psi_\Phi$ belongs to $Y_{\alpha,0}$. The required symmetry properties can be established exactly as in the case of the Euler kernel; we refer to Theorem \ref{thm:Theorem-smooth}. We shall therefore focus exclusively on the regularity issue. Recall from Proposition \ref{prop:prop-Bound-int} that $\Psi_\Phi$ admits the boundary representation
\[
\Psi_\Phi(w)
=
\int_{\mathbb T}
G_0(|\Phi(w)-\Phi(\xi)|)
\textnormal{Re}\!\left(
(\Phi(\xi)-\Phi(w))
\overline{\xi\Phi'(\xi)}
\right)
\,|d\xi|.
\]
Set
\[
A(w,\xi):=\Phi(\xi)-\Phi(w),
\qquad
r(w,\xi):=|A(w,\xi)|,
\]
and
\[
B(w,\xi)
:=
\textnormal{Re}\left(
A(w,\xi)\overline{\xi\Phi'(\xi)}
\right).
\]
Thus
\[
\Psi_\Phi(w)
=
\int_{\mathbb T}
G_0(r(w,\xi))B(w,\xi)\,|d\xi|.
\]
We first record the basic cancellation satisfied by $B$. Since
$\Phi\in C^{1+\alpha}(\mathbb T)$, Taylor's formula at $\xi$ gives
\[
\Phi(\xi)-\Phi(w)
=
\Phi'(\xi)(\xi-w)+R(w,\xi),
\]
where
\[
|R(w,\xi)|
\leqslant
C\|\Phi\|_{C^{1+\alpha}}
|w-\xi|^{1+\alpha}.
\]
Therefore,
\[
\begin{aligned}
B(w,\xi)
&=
|\Phi'(\xi)|^2
\textnormal{Re}\left((\xi-w)\overline{\xi}\right)
+
\textnormal{Re}\left(
R(w,\xi)\overline{\xi\Phi'(\xi)}
\right).
\end{aligned}
\]
Since $w,\xi\in\mathbb T$,
\[
\textnormal{Re}\left((\xi-w)\overline{\xi}\right)
=
1-\textnormal{Re}(w\overline{\xi})
=
\tfrac12|w-\xi|^2.
\]
Hence, because $0<\alpha<1$, we find
\begin{align}\label{TML1}
|B(w,\xi)|
\leqslant
C|w-\xi|^{1+\alpha}.
\end{align}
Moreover, the chord--arc property gives
\begin{align}\label{Es-hmid1}
c_\Phi|w-\xi|
\leqslant
r(w,\xi)
\leqslant
\|\Phi'\|_{L^\infty}|w-\xi|.
\end{align}
We now differentiate tangentially in $w$, denoted by $\partial_\tau^w$, which agrees for holomorphic functions with
\[
\partial_\tau^w f(w)=iwf'(w).
\]
Since
\[
\partial_\tau^w A(w,\xi)
=
-iw\Phi'(w),
\]
we have
\[
\partial_\tau^w r(w,\xi)
=
-\tfrac{
\textnormal{Re}\left(
iw\Phi'(w)\overline{A(w,\xi)}
\right)}
{r(w,\xi)}.
\]
Also,
\[
\partial_\tau^w B(w,\xi)
=
-\textnormal{Re}\left(
iw\Phi'(w)\overline{\xi\Phi'(\xi)}
\right).
\]
Consequently,
\[
\begin{aligned}
\partial_\tau\Psi_\Phi(w)
={}&
-\int_{\mathbb T}
G_0'(r(w,\xi))
\tfrac{
\textnormal{Re}\left(
iw\Phi'(w)\overline{A(w,\xi)}
\right)}
{r(w,\xi)}
B(w,\xi)
\,|d\xi|
\\
&-
\int_{\mathbb T}
G_0(r(w,\xi))
\textnormal{Re}\left(
iw\Phi'(w)\overline{\xi\Phi'(\xi)}
\right)
\,|d\xi|\\
&:=\int_{\mathbb T}\mathcal I(w,\xi)\,|d\xi|.
\end{aligned}
\]
It follows from \eqref{TML1} and \eqref{Es-hmid1} that
$$
|\partial_\tau\Psi_\Phi(w)|
\lesssim 
\int_{\mathbb T}
|G_0'(r(w,\xi))|
|\xi-w|^{1+\alpha}
\,|d\xi|+
\int_{\mathbb T}
|G_0(r(w,\xi))
|
\,|d\xi|.
$$
Then with the notation and the results  of Lemma \ref{lem:k_0to G_0}, we get 
\begin{align*}
|\partial_\tau\Psi_\Phi(w)|
&\lesssim 
\int_{\mathbb T}
\mathcal{M}_1(|\xi-w|)
|\xi-w|^{1+\alpha}
\,|d\xi|+
\int_{\mathbb T}
\mathcal{M}_0(|\xi-w|)
\,|d\xi|\\
&\lesssim 
\int_0^{4\pi}
|G_0^\prime( cs)| 
s^{1+\alpha}
\,ds+
\int_0^{4\pi}
\big(|G_0(cs)|+|G_0(Cs)|\big)ds<\infty\,.
\end{align*}
Thus
\[
\Psi_\Phi\in C^1(\mathbb T).
\]
It remains to prove that $\partial_\tau\Psi_\Phi$ is $\alpha$-Hölder
continuous. Let $w_1\neq w_2\in\mathbb T$ and set
\[
\delta:=|w_1-w_2|.
\]
We split the circle into
\[
\mathcal N
:=
\left\{
\xi\in\mathbb T:
\min\{|w_1-\xi|,|w_2-\xi|\}\leqslant 2\delta
\right\}
\]
and
\[
\mathcal F:=\mathbb T\setminus\mathcal N.
\]
We write
\begin{align*}
\partial_\tau\Psi_\Phi(w_1)
-
\partial_\tau\Psi_\Phi(w_2)
&=
\int_{\mathcal N}
\bigl(
\mathcal I(w_1,\xi)-\mathcal I(w_2,\xi)
\bigr)
\,|d\xi|+\int_{\mathcal F}
\bigl(
\mathcal I(w_1,\xi)-\mathcal I(w_2,\xi)
\bigr)
\,|d\xi|.
\end{align*}
On the near region $\mathcal N$, we get  in view of the previous estimates together with Lemma \ref{lem:k_0to G_0} applied successively with $\gamma=\alpha$ and $\gamma=0$
\begin{align*}
\int_{\mathcal N} |\mathcal I(w_j,\xi)||d\xi|&\leqslant\int_{\mathcal N}
\left(
|G_0(r(w_j,\xi))|
+
|G_0'(r(w_j,\xi))|
|w_j-\xi|^{1+\alpha}
\right)
\,|d\xi|\\
&\leqslant  \int_{0}^{2\delta}
\big(|G_0(c s)|+|G_0(C s)|\big) ds+
\int_0^{2\delta}s^{1+\alpha}|G_0^\prime( c s)| ds \\
&\leqslant  (2\delta)^\alpha\int_{0}^{2\delta}
s^{-\alpha}\big(|G_0(c s)|+|G_0(C s)|\big) ds+  (2\delta)^\alpha
\int_0^1s|G_0^\prime( c s)| ds\\ 
&\leqslant 
C\delta^\alpha.
\end{align*}
We now give the details of the estimate on the far region $\mathcal{F}$. Set
\[
r_j:=|\Phi(w_j)-\Phi(\xi)|,
\]
and define
\[
A_j:=\Phi(\xi)-\Phi(w_j),\qquad
P_j:=iw_j\Phi'(w_j),\qquad
Q:=\xi\Phi'(\xi),
\]
so that
\[
\mathcal I(w_j,\xi)
=
-G_0'(r_j)
\tfrac{\textnormal{Re}(P_j\overline{A_j})}{r_j}
B_j
-
G_0(r_j)\textnormal{Re}(P_j\overline Q).
\]
On the far region
\[
\mathcal F
=
\{\xi\in\mathbb{T}:\min(|w_1-\xi|,|w_2-\xi|)>2\delta\},
\]
the bi-Lipschitz property of \(\Phi\) implies
\[
r_1\asymp r_2\asymp |w_1-\xi|\asymp |w_2-\xi|=:r,
\]
and
\[
|r_1-r_2|
\leqslant C\delta.
\]
Moreover,
\[
|P_1-P_2|\leqslant C\delta^\alpha,
\qquad
|A_1-A_2|\leqslant C\delta,
\]
and the geometric cancellation in \eqref{TML1} yields
\[
|B_j|
=
\left|
\textnormal{Re}\left(
A_j\overline{\xi\Phi'(\xi)}
\right)
\right|
\leqslant
Cr_j^{1+\alpha}.
\]
Expanding
\[
\mathcal I(w_1,\xi)-\mathcal I(w_2,\xi),
\]
all the terms are estimated by combining the preceding bounds with the mean
value theorem. The only nontrivial contribution is the one involving
\(G_0''\), namely
\[
\bigl(G_0'(r_1)-G_0'(r_2)\bigr)
\tfrac{\textnormal{Re}(P_1\overline{A_1})}{r_1}
B_1.
\]
Indeed, by the mean value theorem, and the notation of Lemma \ref{lem:k_0to G_0}
\[
|G_0'(r_1)-G_0'(r_2)|
\leqslant
|r_1-r_2|
\sup_{\rho\in[cr,Cr]}|G_0''(\rho)|
\leqslant
C\delta
\sup_{\rho\in[cr,Cr]}|G_0''(\rho)|.
\]
Since
\[
\left|
\tfrac{\textnormal{Re}(P_1\overline{A_1})}{r_1}
\right|
\leqslant C,
\]
and
\[
|B_1|
\leqslant
Cr^{1+\alpha},
\]
we obtain
\[
\begin{aligned}
&
\left|
\bigl(G_0'(r_1)-G_0'(r_2)\bigr)
\tfrac{\textnormal{Re}(P_1\overline{A_1})}{r_1}
B_1
\right|
\\
&\qquad\le
C\delta r^{1+\alpha}
\sup_{\rho\in[cr,Cr]}|G_0''(\rho)|.
\end{aligned}
\]
Since \(\delta\leqslant Cr\) on \(\mathcal F\), then
\[
\delta r^{1+\alpha}
=
\delta^\alpha\delta^{1-\alpha}r^{1+\alpha}
\le
C\delta^\alpha r^2,
\]
and therefore
\begin{align*}
\left|
\bigl(G_0'(r_1)-G_0'(r_2)\bigr)
\tfrac{\textnormal{Re}(P_1\overline{A_1})}{r_1}
\textnormal{Re}(A_1\overline Q)
\right|
&\leqslant 
C\delta^\alpha
r^2
\sup_{\rho\in[cr,Cr]}|G_0''(\rho)|\\
 &\leqslant 
C\delta^\alpha
r^2\mathcal{M}_2(r).
\end{align*}
The remaining terms are estimated similarly and are in fact simpler, since
they involve only \(G_0\) and \(G_0'\). One gets
\[
|\mathcal I(w_1,\xi)-\mathcal I(w_2,\xi)|
\le
C\delta^\alpha
\left(
\mathcal M_0(r)
+
{(r+r^{1-\alpha})}\mathcal M_1(r)
+
r^2\mathcal M_2(r)
\right)\,.
\]
Applying once again Lemma \ref{lem:k_0to G_0}, allows to get
\begin{align*}
\int_{\mathcal{F}}
\bigl|
\mathcal I(w_1,\xi)-\mathcal I(w_2,\xi)
\bigr|
\,|d\xi|&\leqslant 
C\delta^\alpha\int_0^C
\Bigl(
|G_0(r)|
+
{(r+ r^{1-\alpha})}|G_0'(r)|
+
r^2|G_0''(r)|
\Bigr)
dr\\
&\leqslant C\delta^\alpha.
\end{align*}
Therefore,
\[
|\partial_\tau\Psi_\Phi(w_1)
-
\partial_\tau\Psi_\Phi(w_2)|
\leqslant
C|w_1-w_2|^\alpha.
\]
Thus
\[
\partial_\tau\Psi_\Phi\in C^\alpha(\mathbb T),
\]
and consequently
\[
\Psi_\Phi\in C^{1+\alpha}(\mathbb T).
\]
{\bf{2.}} From Proposition \ref{prop:Lin-Gen} we have
\[
\begin{aligned}
\partial_\Phi F_0(\Omega,\Phi)h(w)
={}&
\Omega\operatorname{Re}
\Big(
\overline{\Phi(w)}
\,h(w)
\Big)
\\
&+
\int_{\mathbb T}
K_0\!\left(|\Phi(w)-\Phi(\xi)|\right)
\operatorname{Re}
\left(
\big[h(\xi)-h(w)\big]\overline{\xi\Phi'(\xi)}
\right)
\, |d\xi|.
\end{aligned}
\]
We claim that for $\Phi\in B_{\varepsilon}^\alpha(\textnormal{Id})$  and $h\in X_\alpha$
\[
\|\partial_\Phi F_0(\Omega,\Phi)h\|_{C^{1+\alpha}(\mathbb{T})}
\lesssim C(\|\Phi\|_{C^{1+\alpha}(\mathbb{T})})\|h\|_{C^{1+\alpha}(\mathbb{T})}.
\]
The first term, involving $\Omega$, is handled directly using the fact that $C^{1+\alpha}(\mathbb T)$ is an algebra.
\\
It remains to control the integral term
\[
\mathcal T_\Phi[h](w)
:=
\int_{\T}
K_0\bigl(|\Phi(w)-\Phi(\xi)|\bigr)
\operatorname{Re}
\left(
[h(\xi)-h(w)]\overline{\xi\Phi'(\xi)}
\right)
\,|d\xi|.
\]
Contrary to the previous estimate for $\Psi_\Phi$, we do not have here
the stronger geometric cancellation of order $1+\alpha$. However, the
difference $h(\xi)-h(w)$ provides the first-order cancellation
\[
|h(\xi)-h(w)|
\leqslant
C\|h\|_{C^{1+\alpha}}|\xi-w|.
\]
Since $\Phi$ is bi-Lipschitz on $\T$, one also has
\[
|\Phi(w)-\Phi(\xi)|\asymp |w-\xi|.
\]
Therefore
\[
|\mathcal T_\Phi[h](w)|
\lesssim
\|h\|_{C^{1+\alpha}}
\int_{\T}
|K_0(c|w-\xi|)|\,|w-\xi|\,|d\xi|,
\]
which is finite by the assumptions on $K_0$.

We next differentiate tangentially with respect to $w$. Setting
\[
r=r(w,\xi):=|\Phi(w)-\Phi(\xi)|,
\]
we obtain terms of the form
\[
K_0'(r)
\frac{
\operatorname{Re}\bigl(
iw\Phi'(w)\overline{\Phi(\xi)-\Phi(w)}
\bigr)
}{r}
\operatorname{Re}
\left(
[h(\xi)-h(w)]\overline{\xi\Phi'(\xi)}
\right),
\]
together with terms involving
\[
K_0(r)
\operatorname{Re}
\left(
iw h'(w)\overline{\xi\Phi'(\xi)}
\right).
\]
Denoting the resulting integrand by $J_h(w,\xi)$, we obtain
\[
|J_h(w,\xi)|
\lesssim
\|h\|_{C^{1+\alpha}}
\left(
|K_0'(r)|\,|w-\xi|
+
|K_0(r)|
\right).
\]
Consequently,
\[
\|\partial_\tau\mathcal T_\Phi[h]\|_{L^\infty}
\lesssim
\|h\|_{C^{1+\alpha}}
\int_0^C
\left(
r|K_0'(r)|+|K_0(r)|
\right)\,dr
\lesssim
\|h\|_{C^{1+\alpha}},
\]
where we have used Lemma~\ref{lem:k_0to G_0}.
\\
It remains to prove the $\alpha$-H\"older continuity of
$\partial_\tau\mathcal T_\Phi[h]$. Let $w_1,w_2\in\T$, set
\[
\delta:=|w_1-w_2|,
\]
and split
\[
\T=\mathcal N\cup\mathcal F,
\]
where
\[
\mathcal N
:=
\left\{
\xi\in\T:
\min\{|w_1-\xi|,|w_2-\xi|\}\leqslant 2\delta
\right\},
\qquad
\mathcal F:=\T\setminus\mathcal N.
\]
We first consider the near region. For $j=1,2$ and $\xi\in\mathcal N$,
the triangle inequality gives
\[
|w_j-\xi|\leqslant 3\delta.
\]
Thus, by the bi-Lipschitz property of $\Phi$,
\[
r_j:=|\Phi(w_j)-\Phi(\xi)|
\lesssim\delta.
\]
Starting from
\[
|J_h(w_j,\xi)|
\lesssim
\|h\|_{C^{1+\alpha}}
\left(
|K_0'(r_j)|\,|w_j-\xi|
+
|K_0(r_j)|
\right),
\]
and using the chord--arc estimate
\[
|w_j-\xi|\lesssim r_j,
\]
we obtain
\[
\begin{aligned}
\int_{\mathcal N}
|J_h(w_j,\xi)|
\,|d\xi|
&\lesssim
\|h\|_{C^{1+\alpha}}
\int_{\mathcal N}
\left(
r_j|K_0'(r_j)|
+
|K_0(r_j)|
\right)
\,|d\xi|
\\
&\lesssim
\delta^\alpha
\|h\|_{C^{1+\alpha}}
\int_{\mathcal N}
\left(
r_j^{1-\alpha}|K_0'(r_j)|
+
r_j^{-\alpha}|K_0(r_j)|
\right)
\,|d\xi|.
\end{aligned}
\]
Here we have used that $r_j\lesssim\delta$ on $\mathcal N$, and hence
\[
r_j
=
r_j^\alpha r_j^{1-\alpha}
\lesssim
\delta^\alpha r_j^{1-\alpha},
\qquad
1
=
r_j^\alpha r_j^{-\alpha}
\lesssim
\delta^\alpha r_j^{-\alpha}.
\]
Since
\[
r_j\asymp |w_j-\xi|,
\]
the last integral can be estimated, after parametrizing the circle
locally by $s=|w_j-\xi|$, as
\[
\begin{aligned}
\int_{\mathcal N}
|J_h(w_j,\xi)|
\,|d\xi|
&\lesssim
\delta^\alpha
\|h\|_{C^{1+\alpha}}
\left[
\int_0^{C\delta}
s^{1-\alpha}|K_0'(cs)|\,ds
+
\int_0^{C\delta}
s^{-\alpha}|K_0(cs)|\,ds
\right]
\\
&\lesssim
\delta^\alpha
\|h\|_{C^{1+\alpha}}
\left[
\int_0^C
s^{1-\alpha}|K_0'(s)|\,ds
+
\int_0^C
s^{-\alpha}|K_0(s)|\,ds
\right].
\end{aligned}
\]
The last two integrals are finite: the second one follows from the
weighted integrability assumption on $K_0$, while the first follows
from Lemma~\ref{lem:k_0to G_0} with $k=1$ and
$\gamma=\alpha$. Consequently,
\[
\int_{\mathcal N}
|J_h(w_j,\xi)|
\,|d\xi|
\lesssim
\delta^\alpha
\|h\|_{C^{1+\alpha}}.
\]
We now consider the far region $\mathcal F$. There one has
\[
|w_j-\xi|\gtrsim\delta,
\qquad
r_j\asymp |w_j-\xi|=:r,
\qquad
|r_1-r_2|\lesssim\delta.
\]
Expanding
\[
J_h(w_1,\xi)-J_h(w_2,\xi),
\]
the most singular contribution is the one involving the difference
$K_0'(r_1)-K_0'(r_2)$. By the mean value theorem,
\[
|K_0'(r_1)-K_0'(r_2)|
\lesssim
\delta M_2(r),
\qquad
M_2(r):=
\sup_{\rho\in[cr,Cr]}|K_0''(\rho)|.
\]
Since the factor $h(\xi)-h(w_j)$ contributes one power of $r$, this
term is bounded by
\[
C\|h\|_{C^{1+\alpha}}\delta rM_2(r).
\]
Moreover, $\delta\lesssim r$ on $\mathcal F$, and therefore
\[
\delta r
=
\delta^\alpha\delta^{1-\alpha}r
\lesssim
\delta^\alpha r^{2-\alpha}.
\]
It follows that
\[
\delta rM_2(r)
\lesssim
\delta^\alpha r^{2-\alpha}M_2(r).
\]
The remaining terms are estimated similarly and are simpler, since
they involve only $K_0$ and $K_0'$. Altogether, one obtains
\[
|J_h(w_1,\xi)-J_h(w_2,\xi)|
\lesssim
\delta^\alpha\|h\|_{C^{1+\alpha}}
\left(
M_0(r)
+
r^{1-\alpha}M_1(r)
+
r^{2-\alpha}M_2(r)
\right),
\]
where
\[
M_k(r):=
\sup_{\rho\in[cr,Cr]}|K_0^{(k)}(\rho)|.
\]
Applying Lemma~\ref{lem:k_0to G_0}, we conclude that
\[
\int_{\mathcal F}
|J_h(w_1,\xi)-J_h(w_2,\xi)|
\,|d\xi|
\lesssim
\delta^\alpha\|h\|_{C^{1+\alpha}}.
\]
Combining the estimates on $\mathcal N$ and $\mathcal F$ yields
\[
|\partial_\tau\mathcal T_\Phi[h](w_1)
-
\partial_\tau\mathcal T_\Phi[h](w_2)|
\lesssim
|w_1-w_2|^\alpha
\|h\|_{C^{1+\alpha}}.
\]
Therefore,
\[
\|\mathcal T_\Phi[h]\|_{C^{1+\alpha}}
\leqslant
C\bigl(\|\Phi\|_{C^{1+\alpha}}\bigr)
\|h\|_{C^{1+\alpha}}.
\]

On the far region, one has $\delta\lesssim r$. Expanding the difference
between the integrands at $w_1$ and $w_2$, the most singular term is the
one involving the difference of $K_0'$. By the mean value theorem,
\[
|K_0'(r_1)-K_0'(r_2)|
\lesssim
\delta\,M_2(r),
\]
where
\[
M_2(r)
:=
\sup_{\rho\in[cr,Cr]}
|K_0''(\rho)|.
\]
Since the factor $h(\xi)-h(w_j)$ contributes one power of $r$, we obtain
\[
\delta\,r\,M_2(r)
\lesssim
\delta^\alpha r^{2-\alpha}M_2(r),
\]
using $\delta\lesssim r$ on $F$. The remaining terms are estimated
similarly and are simpler. Consequently,
\[
|\partial_\tau\mathcal T_\Phi[h](w_1)
-
\partial_\tau\mathcal T_\Phi[h](w_2)|
\lesssim
\delta^\alpha
\|h\|_{C^{1+\alpha}},
\]
by Lemma~\ref{lem:k_0to G_0}. Therefore
\[
\|\mathcal T_\Phi[h]\|_{C^{1+\alpha}}
\leqslant
C\bigl(\|\Phi\|_{C^{1+\alpha}}\bigr)
\|h\|_{C^{1+\alpha}}.
\]
The continuity of the mapping $\Phi\in B_\varepsilon^\alpha(\textnormal{Id})\to \partial_\Phi F_0(\Omega,\Phi)$ is straightforward and  can be done in a similar way.\\
{\bf3.}  Concerning the regularity of the mixed derivative $\partial_\Omega\partial_\Phi F_0$, it is  identically the same as for Euler equations.
\end{proof}

\section{Spectral study at the equilibrium}

We now turn to the spectral analysis of the linearized operator at the trivial equilibrium. This is the key step in identifying the angular velocities at which nontrivial branches of rotating patches may bifurcate from the disk. Our goal is to determine the Fredholm structure of the linearized operator, diagonalize it in Fourier variables, and characterize the values of $\Omega$ for which its kernel is nontrivial.
A central difficulty is that the resulting eigenvalues depend on the interaction kernel $K_0$ and are not, in general, explicitly computable. The complete monotonicity assumption on $-K_0'$ will allow us to overcome this difficulty by deriving a universal factorization of the spectrum. This representation separates the dependence on the kernel from the Fourier modes and ultimately yields the monotonicity and simplicity properties required for the bifurcation argument.
\subsection{Fredholm structure}
Before entering into the detailed spectral analysis, we first establish the Fredholm character of the linearized operator, given according to Corollary \ref{cor:cor-linearization-id} by
\begin{equation}\label{Tomega}
\mathit L_\Omega h(w)
:=D_\Phi F_0(\Omega,\mathrm{Id})h(w)=
(\Omega-\lambda_1)
\operatorname{Re }\,\!\bigl(\overline wh(w)\bigr)
+
\int_{\mathbb T}
K_0(|w-\xi|)
\operatorname{Re }\,\!\bigl(\overline\xi h(\xi)\bigr)
\,|d\xi|.
\end{equation}
The key observation is that $\mathit L_\Omega$ can be decomposed into a multiple of a bounded isomorphism and a compact convolution operator. Away from the distinguished value $\Omega=\lambda_1$, this immediately leads to the following Fredholm property.
\begin{myproposition}{}{pro-Fredh}
Let \(0<\alpha<1\), and recall the spaces in \eqref{X-alpha} and \eqref{Y-alpha}. Assume that
\(
\theta\longmapsto
K_0\!\left(2\left|\sin\tfrac{\theta}{2}\right|\right)
\)
belongs to \(L^1(\mathbb T)\). 
If
\[
\Omega\neq\lambda_1,
\]
then
\[
\mathit L_\Omega:X_\alpha\longrightarrow Y_\alpha
\]
is Fredholm of index zero.
\end{myproposition}

\begin{proof}
Write
\[
w=e^{i\theta},
\qquad
h(w)=\sum_{n\geqslant2}a_ne^{in\theta}.
\]
Then
\[
\operatorname{Re}\!\bigl(\overline wh(w)\bigr)
=
\sum_{n\geqslant2}a_n\cos((n-1)\theta).
\]
It follows that the map 
\[
\mathcal A:X_\alpha\longrightarrow Y_\alpha,
\qquad
\mathcal Ah(w)
=
\operatorname{Re}\!\bigl(\overline wh(w)\bigr),
\]
is a bounded isomorphism.
Set
\[
\kappa(\theta)
:=
K_0\!\left(2\left|\sin\tfrac{\theta}{2}\right|\right).
\]
Since
\[
|e^{i\theta}-e^{i\eta}|
=
2\left|\sin\tfrac{\theta-\eta}{2}\right|,
\]
we obtain
\[
\int_{\mathbb T}
K_0(|w-\xi|)
\operatorname{Re}\!\bigl(\overline\xi h(\xi)\bigr)
\,|d\xi|
=
\int_0^{2\pi}
\kappa(\theta-\eta)
(\mathcal Ah)(e^{i\eta})\,d\eta.
\]
Hence, we get the decomposition 
\[
\mathit L_\Omega
=
\Big((\Omega-\lambda_1)\textnormal{Id}+\mathcal K\Big)\mathcal A,
\]
where $\mathcal K$ is the convolution-type operator defined by
\[
\mathcal Kq(\theta)
:=
\int_0^{2\pi}\kappa(\theta-\eta)q(\eta)\,d\eta.
\]
We will prove the following classical result which claims that 
\[
\mathcal K:Y_\alpha\longrightarrow Y_\alpha
\]
is compact. Since \(\kappa\in L^1(\mathbb T)\) is  real and even, we may choose a sequence of even real trigonometric
polynomials \(\kappa_N\) such that
\[
\|\kappa-\kappa_N\|_{L^1(\mathbb T)}
\longrightarrow0.
\]
Define
\[
\mathcal K_Nq
:=
\kappa_N*q.
\]
Each \(\mathcal K_N\) has finite-dimensional range and $\mathcal K_N:Y_\alpha\to Y_\alpha$ is well-defined. Moreover,
convolution commutes with differentiation, and therefore
\[
\|(\mathcal K-\mathcal K_N)q\|_{C^{1+\alpha}}
\leqslant
\|\kappa-\kappa_N\|_{L^1}
\|q\|_{C^{1+\alpha}}.
\]
Consequently,
\[
\|\mathcal K-\mathcal K_N\|_{
\mathcal L(Y_\alpha,Y_\alpha)}
\longrightarrow0,
\]
which proves that \(\mathcal K\) is compact.
Since \(\Omega-\lambda_1\neq0\), the operator
\[
(\Omega-\lambda_1)\textnormal{Id}:Y_\alpha\longrightarrow Y_\alpha
\]
is an isomorphism. Therefore,
\[
(\Omega-\lambda_1)\operatorname{Id}+\mathcal K: Y_\alpha\longrightarrow Y_\alpha
\]
is a compact perturbation of an isomorphism and is consequently Fredholm of index zero. Since $\mathcal A$ is an isomorphism, we conclude that
\[
 L_\Omega:X_\alpha\longrightarrow Y_\alpha
\]
is Fredholm of index zero. This completes the proof.
\end{proof}
\subsection{Fourier decomposition}
The Fredholm property obtained above reduces the loss of invertibility of $ L_\Omega$ defined in \eqref{Tomega} to a finite-dimensional phenomenon. To identify precisely when this occurs, we now exploit the rotational invariance of the kernel and decompose the linearized operator into Fourier modes. Since the kernel depends only on the angular difference, each Fourier mode is invariant under $ L_\Omega$, and the operator becomes diagonal in the trigonometric basis. This leads naturally to a sequence of spectral coefficients $\lambda_n$ and to the associated dispersion relation, whose zeros determine the possible bifurcation values of the angular velocity. The following proposition makes this decomposition explicit and provides a first description of the kernel of $ L_\Omega$.
\begin{myproposition}{}{prop-Fourier}
For \(n\geqslant1\), define
\[
\lambda_n
:=
\int_0^{2\pi}
K_0\!\left(
2\left|\sin\tfrac{\eta}{2}\right|
\right)
\cos(n\eta)\,d\eta\quad\hbox{and}\quad \mu_n(\Omega):=
\Omega-\lambda_1+\lambda_n\,,
\]
and consider the dispersion set
$$\mathcal{S}:=\big\{\Omega\in\mathbb{R}: \exists n\geqslant 1, \mu_n(\Omega)=0 \big\}=\big\{\Omega=\lambda_1-\lambda_n, n\geqslant 1 \big\}\,.
$$
Then, for any function
\[
w\in\mathbb{T}\mapsto h(w)=\sum_{n\geqslant2}a_nw^n\in X_\alpha,
\]
we get
\[
 L_\Omega h(e^{i\theta})
=
\sum_{n\geqslant1}
\mu_{n}(\Omega)
a_{n+1}
\cos(n\theta),
\]
In addition, for $\Omega\in\mathcal{S},$ we have
\[
\ker L_\Omega
=
\operatorname{span}
\left\{
w^{n+1}:
\mu_n(\Omega)=0
\right\}.
\]
\end{myproposition}

\begin{proof}
Let
\[
w=e^{i\theta},
\qquad
\xi=e^{i\eta},\quad
h(w)=\sum_{n\geqslant2}a_nw^n.
\]
Then, we have
\[
\operatorname{Re}\,\!\bigl(\overline wh(w)\bigr)
=
\sum_{n\geqslant2}
a_n\cos((n-1)\theta),
\]
and similarly,
\[
\operatorname{Re}\,\!\bigl(\overline\xi h(\xi)\bigr)
=
\sum_{n\geqslant2}
a_n\cos((n-1)\eta).
\]
Moreover,
\[
K_0(|w-\xi|)
=
\kappa(\theta-\eta),\quad\hbox{with}\quad 
\kappa(s)
:=
K_0\!\left(
2\left|\sin\tfrac{s}{2}\right|
\right).
\]
Hence, we obtain from  \eqref{Tomega}
\[
\begin{aligned}
 L_\Omega h(e^{i\theta})
={}&
(\Omega-\lambda_1)
\sum_{n\geqslant2}
a_n
\cos((n-1)\theta)
+\sum_{n\geqslant 2}
a_n
\int_0^{2\pi}
\kappa(\theta-\eta)
\cos((n-1)\eta)\,d\eta.
\end{aligned}
\]
Fix \(m\geqslant1\). By the change of variables
\(
s=\theta-\eta
\),
\[
\begin{aligned}
\int_0^{2\pi}
\kappa(\theta-\eta)
\cos(m\eta)\,d\eta
&=
\int_0^{2\pi}
\kappa(s)
\cos(m(\theta-s))\,ds
\\
&=
\cos(m\theta)
\int_0^{2\pi}
\kappa(s)\cos(ms)\,ds+
\sin(m\theta)
\int_0^{2\pi}
\kappa(s)\sin(ms)\,ds.
\end{aligned}
\]
Since \(\kappa\) is even,
\[
\int_0^{2\pi}
\kappa(s)\sin(ms)\,ds
=0.
\]
Therefore,
\[
\int_0^{2\pi}
\kappa(\theta-\eta)
\cos(m\eta)\,d\eta
=
\lambda_m\cos(m\theta),
\]
where
\[
\lambda_m
=
\int_0^{2\pi}
K_0\!\left(
2\left|\sin\tfrac{\eta}{2}\right|
\right)
\cos(m\eta)\,d\eta.
\]
Substituting this identity with \(m=n-1\) yields
\[
 L_\Omega h(e^{i\theta})
=
\sum_{n\geqslant 1}
\bigl(
\Omega-\lambda_1+\lambda_{n}
\bigr)
a_{n+1}
\cos(n\theta),
\]
which proves the diagonalization. The result on kernel structure is easy to prove. This achieves the proof.
\end{proof}
\subsection{Spectrum factorization}
The spectral analysis admits a remarkable factorization which is one of the central features of the present approach. Although the coefficients $\lambda_n$ depend on the particular interaction kernel $K_0$, this dependence can be completely separated from the oscillatory mechanism responsible for the spectral structure. More precisely, the assumptions on $K_0$, through Bernstein's representation, allow us to encode the kernel solely by the positive measure $\mu$, while all the dependence on the Fourier mode $n$ is carried by a universal family of functions $\phi_n$. This separation is particularly important: it reveals that the essential spectral properties are not tied to the specific form of the kernel, but are governed by a common underlying mechanism shared by the whole class under consideration.

The functions $\phi_n$ are therefore universal objects in the analysis. Once their positivity, monotonicity, and asymptotic behavior are understood, the corresponding properties of the spectral coefficients $\lambda_n$ can be transferred to every admissible kernel through integration against the positive measure $d\mu(x)/x$. In this sense, the factorization below reduces a kernel-dependent spectral problem to the study of a single universal family.

More precisely, we have the following key result.
\begin{mylemma}{}{lem:lamb-n}
Assume that $K_0$ satisfies \eqref{Bernstein}  and \eqref{kernel-integrability}.
Then for every $n\geqslant 1$, the coefficient $\lambda_n$ given in Proposition  \ref{prop:prop-Fourier} admits the representation
\begin{eqnarray*}
  \lambda_{n}=2\int_0^{\infty}\tfrac{\phi_{n}(x)}{x}{d {\mu}(x)},
\end{eqnarray*}
with
\begin{eqnarray*}
  \phi_n(x) := \int_0^\pi e^{-2x \sin(\eta)}e^{i {2 n\eta}} d\eta 
 {\,= \int_0^\pi e^{-2x \sin(\eta)} \cos(2n\eta) d \eta}.
\end{eqnarray*}
Notice that $\phi_n\in C^\infty(\mathbb R)$ and $\phi_n(0)=0$. Consequently, the quotient $\phi_n(x)/x$ has a removable singularity at $x=0$, and thus the integral is well defined near the origin.
\end{mylemma}

\begin{proof}
Under the assumption \eqref{Bernstein},
we infer the existence of a Borel measure ${\mu}$ on $[0, \infty)$ such that
\begin{eqnarray}\label{eq:K0prim2}
  -K'_0(t)=\int_0^{\infty}e^{-tx }d{\mu}(x),\quad \forall t >0.
\end{eqnarray}
Integrating \eqref{eq:K0prim2} with respect to $t$-variable, and using Fubini's theorem we obtain
\begin{align}\label{Expression-K}
  \nonumber K_0(t) & = K_0(2) - \int_{2}^t \int_0^{\infty}e^{-\tau x} \,d {\mu}(x)d \tau \\
  \nonumber  & = K_0(2)-\int_0^{\infty} \int_{2}^t e^{- \tau x}d \tau d {\mu}(x) \\
  & = K_0(2)+\int_0^{\infty}\tfrac{e^{-t x}-e^{-2x}}{x}d{\mu}(x).
\end{align}
Recall from Proposition \ref{prop:prop-Fourier}, and using change of variables, 
\[
\lambda_n
=
\int_0^{2\pi}
K_0\!\left(
2\left|\sin\tfrac{\eta}{2}\right|
\right)
e^{i n\eta}\,d\eta= 2\int_0^{\pi}
K_0\!\left(
2\sin{\eta}
\right)
e^{i 2 n\eta}\,d\eta.
\]
Applying  Fubini's theorem, we can rewrite 
\begin{align*}
  \lambda_{n} & =  2\int_0^\pi \int_0^\infty
  \Big(\frac{e^{-2\,x \sin \eta }-e^{-2x}}{x} d {\mu}(x)\Big) e^{i2n\eta}d \eta \nonumber \\
  & = 2\int_0^\infty \int_0^\pi \big( e^{-2x \sin \eta }
  - e^{-2x} \big) e^{i2n\eta} \,d \eta\, \tfrac{d{\mu}(x)}{x} \nonumber \\
  & = 2\int_0^{\infty}\phi_{n}(x)\tfrac{d {\mu}(x)}{x}.
\end{align*}
This achieves the proof of the desired result.
\end{proof}
\subsection{Monotonicity of the spectrum and kernel structure}
By Lemma~ \ref{lem:lem:lamb-n}, the spectral coefficients are determined by the
universal functions \(\phi_n\). It therefore remains to establish their
positivity and monotonicity, which yield the ordering of the eigenvalues
and the simplicity of the corresponding kernels.
Rather than estimating the oscillatory integral defining \(\phi_n\)
directly, we derive a second-order differential equation satisfied by
\(\phi_n\) and apply a comparison principle. We begin with the elementary
ODE lemma needed for this purpose.

\begin{mylemma}{}{lem-comparison}
{Let $a:(0,\infty)\to \mathbb{R}$ and $b:(0,\infty)\to (0,\infty)$ be two given continuous functions and define the second order operator
$$
\mathbb{L} f(x)=f^{\prime\prime}(x) + a(x)f^\prime(x) - b(x)f(x), x>0\,.$$ 
The following statements hold true.
\begin{enumerate}[label=(\arabic*)]
\item 
Suppose that $f\in C^2((0,\infty))$ is a function satisfying the conditions
\begin{eqnarray*}
  \left\{\begin{array}{ll}
  \mathbb{L} f(x)\leqslant 0, \quad\forall  x>0,\\
  \displaystyle{\lim_{x\to0^+}}f(x)\geqslant 0,\quad \lim\limits_{x\to \infty}f(x)\geqslant 0.\end{array}\right.
\end{eqnarray*}
Then $f$ is non-negative on $(0,\infty)$.
Moreover, if $f$ additionally satisfies
\begin{align*}
  \mathbb{L} f(x) < 0,\quad \forall  x>0,
\end{align*}
then $f$ is strictly positive on $(0,\infty)$.
\item 
Suppose that $f\in C^2((0,\infty))$ is a function satisfying the conditions
\begin{eqnarray*}
  \left\{\begin{array}{ll}
  \mathbb{L} f(x)\geqslant 0, \quad\forall  x>0,\\
  \displaystyle{\lim_{x\to0^+}}f(x)\leqslant 0,\quad \lim\limits_{x\to \infty}f(x)\leqslant 0.\end{array}\right.
\end{eqnarray*}
Then $f$ is non-positive on $(0,\infty)$.
Moreover, if $f$ additionally satisfies
\begin{align*}
  \mathbb{L} f(x) > 0,\quad \forall  x>0,
\end{align*}
then $f$ is strictly negative on $(0,\infty)$.
\end{enumerate}
}

\end{mylemma}

\begin{proof}
\textbf{(1)} We will {start by} proving the first assertion. 
For this aim, we shall argue by contradiction. Assume that $f$ takes strictly negative values at some points of $(0,\infty)$.
Then in light of the assumptions $\displaystyle{\lim_{x\to0^+}}f(x)\geqslant 0$ and $\lim\limits_{x\to \infty}f(x)\geqslant 0,$ one can find some $x_0>0$ such that
\begin{align*}
  \inf_{x>0}f(x)=f(x_0) < 0.
\end{align*}	
Hence,
\begin{align}\label{prop:f-inf}
  f^\prime(x_0)=0, \quad f^{\prime\prime}(x_0)\geqslant 0.
\end{align}
Coming back to the differential inequality we find
\begin{align*}
  f''(x_0)\leqslant b(x_0)f(x_0) < 0,
\end{align*}
which is a contradiction.\\
For the second assertion, we assume that $f$ takes non-positive values at some points of $(0,\infty)$,
then there exists some $x_0>0$ so that $\inf\limits_{x>0} f(x) = f(x_0) \leqslant 0$ satisfies \eqref{prop:f-inf},
but using the strict differential inequality gives $f''(x_0) < b(x_0) f(x_0) \leqslant 0$, and it yields a contradiction. 
This concludes the proof of the desired result.\\
\textbf{(2)} It follows from applying statement (1) to the function $-f$. This ends the proof.
\end{proof}

The next goal is to exploit the comparison theorem to establish some qualitative properties of the universal functions $\phi_n$ introduced in Lemma \ref{lem:lem:lamb-n}. These properties will play a crucial role in the analysis of the spectral structure of the linearized operator, and in particular in the characterization of its kernel. More precisely, we shall prove the following result.
\begin{myproposition}{}{thm:mono}
The following results hold true.
\begin{enumerate}
    \item 
  For every $n\geqslant 1$ and $x>0$, $\phi_n(x)>0$ and the map $n\mapsto \phi_n(x)$ is strictly decreasing.
  \item For each $n\geqslant 1,$ we have $ \lambda_n>0$ and the sequence $(\lambda_n)_{n\geqslant1}$ is strictly decreasing.
  \item Consider the dispersion set $\mathcal{S}$ introduced in Proposition \ref{prop:prop-Fourier}. Then for $\Omega\in\mathcal{S},$ the kernel $\ker L_\Omega$ is one-dimensional  with 
\[
\ker L_\Omega
=
\operatorname{span}
\left\{
w^{n+1}
\right\},\quad \Omega=\lambda_1-\lambda_n, n\geqslant1.
\]
  \end{enumerate}
\end{myproposition}

\begin{proof}
{\bf{1.}} For $\mathbf{z} \in \mathbb{C}$, define
\begin{align*}
  \Phi_n(\mathbf{z}) \triangleq \frac{1}{\pi} \int_0^\pi e^{i(- \mathbf{z}\sin \eta + 2n\eta) } d \eta.
\end{align*}
Recall the Anger and Weber functions defined successively by, see 8.580 in \cite{GR15},
\begin{align*}
  \mathbf{J}_\nu(\mathbf{z}) = \frac{1}{\pi} \int_0^\pi \cos(\nu \eta - \mathbf{z} \sin \eta) d\eta\quad \hbox{and}\quad 
  \mathbf{E}_\nu(\mathbf{z}) = \frac{1}{\pi} \int_0^\pi \sin(\nu\eta - \mathbf{z} \sin \eta) d \eta.
\end{align*}
Then we find
\begin{align*}
  \Phi_n(\mathbf{z})=\mathbf{J}_{2n}(\mathbf{z})+i \mathbf{E}_{2n}(\mathbf{z})
\end{align*}
and
\begin{align}\label{Rela-1}
  \phi_n(x)={\pi\Phi_n(-2 i x)}.
\end{align}
Now, it is a classical fact that  the functions $\mathbf{J}_{2n}(\mathbf{z})$ and $\mathbf{E}_{2n}(\mathbf{z})$  satisfy the following ODEs, for instance see 8.584 in \cite{GR15},
\begin{align*}
  \mathbf{J}_{2n}^{\prime\prime}(\mathbf{z}) + \mathbf{z}^{-1} \mathbf{J}_{2n}^\prime(\mathbf{z}) + \left(1-\tfrac{4n^2}{\mathbf{z}^2}\right)
  \mathbf{J}_{2n}(\mathbf{z}) = 0
\end{align*}
and
\begin{align*}
  \mathbf{E}_{2n}^{\prime\prime}(\mathbf{z}) + \mathbf{z}^{-1}\mathbf{E}_{2n}^\prime(\mathbf{z}) + \left(1-\tfrac{4n^2}{\mathbf{z}^2}\right) \mathbf{E}_{2n}(\mathbf{z}) = -\frac{2}{\pi \mathbf{z}}\cdot
\end{align*}
It follows that
\begin{align*}
  \Phi_n^{\prime\prime}(\mathbf{z}) + \mathbf{z}^{-1}\Phi_n^\prime(\mathbf{z}) +\left(1-\tfrac{4n^2}{\mathbf{z}^2}\right)\Phi_n(\mathbf{z})
  = - \frac{2 i}{\pi \mathbf{z}}\cdot
\end{align*}
This implies by virtue of \eqref{Rela-1} that
\begin{align}\label{eq:phi-ODE}
  \phi_n^{\prime\prime}(x) + x^{-1}\phi_n^\prime(x) - 4\left(1+\tfrac{n^2}{x^2}\right)\phi_n(x) = -\frac{4}{ x}, \quad  x>0.
\end{align}
On the other hand, one may get  from Riemann-Lebesgue's lemma applied with the integral formula of $\phi_n$ that
\begin{align}\label{eq:phi-n0}
  \forall n\geqslant 1,\quad \phi_n(0)=0,\quad \textrm{and}\quad  
  \lim_{x\to\infty}\phi_n(x)=0.
\end{align}
Hence, Lemma \ref{lem:lem-comparison} guarantees that
\begin{align*}
  \forall x>0, \quad \phi_n(x)>0.
\end{align*}
Now we show that for any $x>0$ the sequence $n\mapsto \phi_n(x)$ is strictly decreasing.
For this aim, we define
\begin{align*}
  \chi_n(x) := \phi_n(x) - \phi_{n+1}(x).
\end{align*}
Then using  the equation \eqref{eq:phi-ODE} we find
\begin{align}\label{d-ff-2}
  \chi_n^{\prime\prime}(x) + x^{-1}\chi_n^\prime(x) - 4\left(1+\tfrac{n^2}{x^2}\right)\chi_n(x)
  = - 4 \tfrac{2n+1}{x^2}(x) \phi_{n+1}(x),
  \quad x>0,
\end{align}
with
\begin{align*}
  \chi_n(0)=0, \quad  
  \lim_{x\to\infty}\chi_n(x) = 0.
\end{align*}
Thus, \eqref{d-ff-2} and Lemma \ref{lem:lem-comparison} ensure that
\begin{align*}
  \forall x>0, \quad \chi_n(x) > 0,
\end{align*}
which implies the strict monotonicity of $\phi_n$.
This concludes the proof of the first result.\\
{\bf{2.}} From the formula
\begin{eqnarray*}
  \lambda_{n}=2\int_0^{\infty}\phi_{n}(x)\tfrac{d {\mu}(x)}{x},
\end{eqnarray*}
together with the first point we get the desired result.\\
{\bf{3.}} The result follows from the second point together with Proposition \ref{prop:prop-Fourier}.

\end{proof}
\section{Proof of Theorem \ref{thm:Hmidi-Xue}}

We are now in a position to complete the proof of Theorem \ref{thm:Hmidi-Xue}. The argument follows the same general bifurcation strategy developed for the Euler equation in Section \ref{sec:Burea-q}, based on the Crandall--Rabinowitz theorem. The required regularity properties have already been established in Proposition \ref{prop:Theorem-REg-smooth}. The main additional difficulty in the present setting stems from the general structure of the interaction kernel, for which the spectral properties of the linearized operator are no longer available in an explicit form.
This difficulty has been overcome in the preceding sections through the factorization of the spectral coefficients in terms of the universal functions $\phi_n$. The qualitative properties established for this family provide the necessary information on the eigenvalues and allow us to characterize the kernel and the range of the linearized operator at each bifurcation value. In particular, they ensure the simplicity of the relevant eigenvalue and the Fredholm property with index zero; see Propositions \ref{prop:pro-Fredh} and \ref{prop:thm:mono}. Finally, the transversality condition can be verified exactly as in Section \ref{Section-Transv}. All the hypotheses of the Crandall--Rabinowitz theorem are therefore satisfied, and the desired result follows.

\chapter{Concluding Perspectives}

The theory developed in this monograph illustrates several complementary
mechanisms governing the existence and geometry of rotating vortex
patches. For the two-dimensional Euler equations, the Cauchy transform
and the exterior conformal mapping provide two natural formulations of
the free-boundary problem, revealing respectively its potential-theoretic
and complex-analytic structures. These tools lead to rigidity results
and, through the conformal parametrization, to the bifurcation theory
initiated by Burbea.

A central conclusion of the preceding chapters is that the local
bifurcation mechanism is considerably more universal than the explicit
Euler computations might suggest. For a broad class of active scalar
equations, complete monotonicity of the radial interaction kernel
provides a representation of the spectral coefficients that separates
the Fourier structure from the particular model. This makes it possible
to recover the simplicity, monotonicity, and nondegeneracy properties
required by the Crandall--Rabinowitz theorem without relying on explicit
special-function computations. In this sense, the Euler, gSQG, and QGSW
bifurcation theories can be viewed as manifestations of a common
spectral mechanism.

The results presented here also leave several natural questions open.
We briefly mention some directions that appear particularly relevant.

\section{Beyond local bifurcation}

The bifurcation results of Chapters~6 and~7 are local: they construct
nontrivial branches of rotating patches in a neighborhood of the disk.
A substantially more difficult problem is to understand the global
continuation of these branches. Numerical and analytical studies
indicate a rich global structure, including turning points and limiting
configurations with corners or self-contact. A general description of
the global geometry of these branches, and in particular of their
possible singular limits, remains largely open.

It would be especially interesting to determine to what extent the
structural assumptions on the interaction kernel used in Chapter~7 can
also support a global bifurcation theory. The complete monotonicity
framework provides robust control of the spectrum near the disk, but
its relevance far from the trivial branch is much less understood.

\section{Rigidity and classification}

The rigidity theory developed in Chapter~5 shows that, for the Euler
equations, the angular velocity strongly constrains the geometry of a
rotating patch. Outside the admissible interval of angular velocities,
the disk is the only simply connected rotating configuration. Inside
the bifurcation regime, however, the classification problem is much
more subtle, as the existence of the Kirchhoff ellipses and the
higher-fold Burbea branches demonstrates.

A natural problem is therefore to understand which additional geometric
or analytical assumptions lead to classification results within the
nontrivial regime. More generally, one may ask whether the analytic and
variational rigidity mechanisms discussed in Chapter~5 admit analogues
for the broader class of active scalar equations considered in
Chapter~7.

\section{Topology and singular geometry}

The theory of rotating vortex patches extends well beyond the simply connected setting considered in most of this monograph. In particular, a substantial theory has been developed for doubly connected $V$-states, where the patch is bounded by two interacting interfaces. Local bifurcation from annular configurations, symmetry properties, and global branches have been investigated for the Euler equation and for several related active scalar models. More generally, patches with several interfaces exhibit a richer spectral structure, since the different boundary components interact through the contour dynamics.

Many questions nevertheless remain concerning the global organization of these branches, higher-connectivity configurations, and the possible degenerations of the interfaces. In particular, smooth families of $V$-states may approach limiting configurations involving corners, self-contact, or collision between distinct boundary components. Understanding these singular limits and the mechanisms leading to their formation remains an important problem in the global theory.

\section{From relative equilibria to recurrent dynamics}

Rotating patches are relative equilibria: after passing to a uniformly
rotating frame, the free boundary becomes stationary. Recent
developments show that vortex dynamics also admit genuinely
time-dependent coherent structures, including periodic and
quasi-periodic vortex patches. These solutions require techniques that
go substantially beyond the local bifurcation theory considered here,
involving the analysis of infinite-dimensional Hamiltonian systems and
small-divisor phenomena.
It is natural to ask how the geometric and spectral structures arising
in the study of $V$-states interact with these more general recurrent
motions. In particular, rotating patches may serve as reference states
around which periodic or quasi-periodic dynamics can bifurcate, thereby
connecting the theory of relative equilibria with the broader problem
of recurrent vortex motion.

\section{Beyond complete monotonicity}
Complete monotonicity provides the structural principle behind the
unified bifurcation theory of Chapter~7. Through Bernstein's
representation, it yields a factorization of the spectral coefficients
and, consequently, the monotonicity and simplicity needed for local
bifurcation.
It is unlikely, however, that complete monotonicity is the optimal
condition. The essential issue is to identify the weakest assumptions
on the interaction kernel that preserve the relevant spectral
properties. Such a characterization would clarify the precise scope of
the universal bifurcation mechanism and could extend the theory to
active scalar models lying beyond the class considered here.


\begin{thebibliography}{99}	

\bibitem{AbS64}
M. Albramowitz and I. A. Stegun,
\textit{Handbook of Mathematical Functions with Formulas, Graphs, and Mathematical Tables}.
Vol. 55 of National Bureau of Standards Applied Mathematics Series, (1964).




\bibitem{BHM23}
M. Berti, Z. Hassainia and N. Masmoudi,
Time quasi-periodic vortex patches of Euler equation in the plane.
Invent. Math., \textbf{233} (2023), no. 3, 1279--1391.


\bibitem{BertozziConstantin1993}
A.~L. Bertozzi and P.~Constantin, Global regularity for vortex patches,
Comm. Math. Phys., \textbf{152} (1993), 19--28.
	
\bibitem{Burbea82}
J. Burbea, Motions of vortex patches.
Letters in Mathematical Physics, \textbf{6} (1982), no. 1, 1--16.	
\bibitem{Burbea81} J. Burbea, Vortex motions and conformal mappings, pp. 276–298 in Nonlinear evolution equations and
dynamical systems (Lecce, Italy, 1979), edited by M. Boiti et al., Lecture Notes in Phys. 120, Springer, Berlin, 1980.




\bibitem{CQZZ} D. Cao, G. Qin, W. Zhan and C. Zou,
Existence and regularity of co-rotating and traveling-wave vortex solutions for the generalized SQG equation.
J. Differ. Equations, \textbf{299} (2021), 429--462.

	
\bibitem{CCG16}
A. Castro, D. C\'ordoba and J. G\'omez-Serrano,
Existence and regularity of rotating global solutions for the generalized surface Quasi-Geostrophic equations.
Duke Math. J., \textbf{165} (2016), no. 5, 935--984.
	
\bibitem{CCG16b}
A. Castro, D. C\'ordoba and J. G\'omez-Serrano,
Uniformly rotating analytic global patch solutions for active scalars.
Annals of PDE, \textbf{2} (2016), no. 1, Art. 1.





%
\bibitem{Chemin1998} J.-Y. Chemin, Perfect incompressible Fluids. Oxford University Press 1998.
\bibitem{Chemin1993} J.-Y. Chemin, Sur le mouvement des particules d’un fluide parfait incompressible bidimensionnel. Invent. Math.
103, no. 3, 599–629 (1991).
\bibitem{Chemin19930}
J.-Y. Chemin, Persistance de structures g\'eom\'etriques dans les fluides incompressibles bidimensionnels.
Annales Scientifiques de l'\'Ecole Normale Sup\'erieure, \textbf{26} (1993), no. 4, 517--542.











\bibitem{C-R71} M. G. Crandall and P. H. Rabinowitz,
Bifurcation from simple eigenvalues.
J. Funct. Anal., {\bf 8} (1971), 321--340.
\bibitem{Curtis} J. H. Curtiss, {Faber polynomials and Faber series}, The American Mathematical Monthly, Vol. 78, No. 6 (Jun. - Jul., 1971), pp. 577-596.

\bibitem{DZ78} G. S. Deem and N. J. Zabusky,
Vortex waves: stationary ``V-states", interactions, recurrence, and breaking.
Phys. Rev. Lett., \textbf{40} (1978), no. 13, 859--862.
\bibitem{Delort1991}
J.-M.~Delort,
\newblock Existence de nappes de tourbillon en dimension deux,
\newblock \emph{J. Amer. Math. Soc.} \textbf{4} (1991), 553--586.
\bibitem{DHH16}
F. De la Hoz, Z. Hassainia and T. Hmidi,
Doubly connected V-states for the generalized surface Quasi-Geostrophic equations.
Arch. Ration. Mech. Anal., \textbf{220} (2016), no. 3, 1209--1281.
	
\bibitem{DHHM} F. De la Hoz, Z. Hassainia, T. Hmidi and J. Mateu,
An analytcal and numerical study of steady paches in the disk.
Analysis \& PDE, 9 (2016), no. 7, 1609--1670.



\bibitem{DHMV16}
F. De la Hoz, T. Hmidi, J. Mateu and J. Verdera,
Doubly connected $V$-states for the planar Euler equations.
SIAM J. Math. Anal., \textbf{48} (2016), no. 3, 1892--1928.



	
\bibitem{DHR19}
D. G. Dritschel, T. Hmidi and C. Renault, Imperfect bifurcation for the quasi-geostrophic shallow-water equations.
Arch. Ration. Mech. Anal., \textbf{231} (2019), no. 3, 1853--1915.


	

\bibitem{Gan08}
F. Gancedo, Existence for the $\alpha$-patch model and the QG sharp front in Sobolev spaces.
Adv. Math., \textbf{217} (2008), no. 6, 2569--2598.


\bibitem{Garcia21} C. Garc\'ia, Vortex patches choreography for active scalar equations.
J. Nonlinear Sci., \textbf{31} (2021), no. 5, Paper No. 75.

\bibitem{Garcia20}
C. Garc\'ia, K\'arm\'an vortex street in incompressible fluid models.
Nonlinearity, \textbf{33} (2020), no. 4, 1625--1676.


\bibitem{GS23}
C. Garc\'ia and S. V. Haziot,
Global bifurcation for corotating and counter-rotating vortex pairs.
Comm. Math. Phys., \textbf{402} (2023), 1167--1204.


\bibitem{GHJ20}
C. Garc\'ia, T. Hmidi and J. Soler,
Non uniform rotating vortices and periodic orbits for the two-dimensional Euler equations.
Arch. Ration. Mech. Anal., \textbf{238} (2020), no. 2, 929--1085.




\bibitem{HMH23} 
C. Garc\'ia, T. Hmidi and J. Mateu, 
Time periodic solutions close to localized radial monotone profiles for the 2D Euler equations. 
Annals of PDE, \textbf{10}(2024), Art. No. 1.




\bibitem{Gom19} J. G\'omez-Serrano, On the existence of stationary patches.
Adv. Math., \textbf{343} (2019), 110--140.


\bibitem{GPSY} J. G\'omez-Serrano, J. Park, J. Shi and Y. Yao,
Symmetry in stationary and uniformly rotating solutions of active scalar equations.
Duke Math. J., \textbf{170} (2021), no. 13, 2957--3038.


\bibitem{GR15}
I. S. Gradshteyn, and I. M. Ryzhik, 
\textit{Table of Integrals, Series, and Products}. Translated from the Russian.
Translation edited and with a preface by Daniel Zwillinger and Victor Moll. Eighth edition.
Elsevier/Academic Press, Amsterdam, (2015).


		
\bibitem{HH15}
Z. Hassainia and T. Hmidi, 
On the V-states for the generalized quasi-geostrophic equations.
Comm. Math. Phys., \textbf{337} (2015), no. 1, 321--377.

\bibitem{HHM23} Z. Hassainia, T. Hmidi and N. Masmoudi,
KAM theory for active scalar equations. 
Memoirs Amer. Math. Soc., \textbf{314} (2025), no. 1596, 279 pp.
		
\bibitem{HMW20} Z. Hassainia, N. Masmoudi and M. H. Wheeler,
Global bifurcation of rotating vortex patches.
Comm. Pure Appl. Math., \textbf{73} (2020), no. 9, 1933--1980.


\bibitem{HR22}
Z. Hassainia and E. Roulley,
Boundary effects on the emergence of quasi-periodic solutions for Euler equations. Nonlinearity, \textbf{38}(2025), no. 1, Art. no 015016.

\bibitem{HHR23}
Z. Hassainia, T. Hmidi and E. Roulley,
Invariant KAM tori around annular vortex patches for 2D Euler equations. Commun. Math. Phys., \textbf{405}, no. 11, Paper No. 270.

\bibitem{HW22} Z. Hassainia and M. H. Wheeler,
Multipole vortex patch equilibria for active scalar equations.
SIAM J. Math. Anal., \textbf{54} (2022), no. 6, 6054--6095.
	

	\bibitem{Hmidi15} T. Hmidi, On the trivial solutions for the rotating patch model,
J. Evol. Equ. 15 (2015), no. 4, 801–816.
	
\bibitem{HM16b} T. Hmidi and J. Mateu,
Degenerate bifurcation of the rotating patches.
Adv. Math., \textbf{302} (2016), 799--850.

\bibitem{HM17} T. Hmidi and J. Mateu,
Existence of corotating and counter-rotating vortex pairs for active scalar equations.
Comm. Math. Phys., \textbf{350} (2017), no. 2, 699--747.

\bibitem{HMV13} T. Hmidi, J. Mateu and J. Verdera,
Boundary regularity of rotating vortex patches.
Arch. Ration. Mech. Anal., \textbf{209} (2013), no. 1, 171--208.

\bibitem{HMV15} T. Hmidi, J. Mateu and J. Verdera, On rotating doubly connected vortices, J. Differential Equations 258 (2015), no. 4, 1395–1429.
\bibitem{HR21}
T. Hmidi and E. Roulley,
Time quasi-periodic vortex patches for quasi-geostrophic shallow-water equations.
ArXiv:2110.13751 [math.AP]. To appear in M\'emoires de la SMF.



\bibitem{HXX23} T. Hmidi, L. Xue and Z. Xue,
Emergence of time periodic solutions for the generalized surface quasi-geostrophic equation in the disk.
J. Funct. Anal., \textbf{285} (2023), no. 10, Paper No. 110142.

\bibitem{HXX26} T. Hmidi, L. Xue and Z. Xue, Unified theory on V-states structures for active scalar equations, Adv. Math. 486 (2026), Paper No. 110750, 77 pp.





		

\bibitem{Kato1967}
T.~Kato,
\newblock On classical solutions of the two-dimensional non-stationary
Euler equation,
\newblock \emph{Arch. Rational Mech. Anal.} \textbf{25} (1967), 188--200.


\bibitem{Keller} H. B. Keller, W. F.  Langford, , Iterations, perturbations, and multiplicities for nonlinear
bifurcation problems, Arch. Rational Mech. Anal. 48, 82-108 (1972).
\bibitem{Kirch}
G. Kirchhoff, \textit{Vorlesungen uber Mathematische Physik}, Leipzig (1874).













\bibitem{Lesley} Lesley, Frank D.; Vinge, Vernor S.; Warschawski, Stefan E. Approximation by Faber polynomials for a class of Jordan domains. Math. Z. 138 (1974), 225-237.




\bibitem{MB02} A. J. Majda and A. L. Bertozzi, 
\textit{Vorticity and Incompressible Flow}.
Cambridge Texts in Applied Mathematics, 27. Cambridge University Press, Cambridge, (2002).













	
				





\bibitem{Rou23a} E. Roulley,
Vortex rigid motion in quasi-geostrophic shallow-water equations.
Asymptot. Anal., \textbf{133} (2023), no. 3, 397--446.

\bibitem{Rou23b} E. Roulley, Periodic and quasi-periodic Euler-$\alpha$ flows close to Rankine vortices.
Dyn. Partial Differ. Equ., \textbf{20} (2023), no. 3, 311--366.




\bibitem{Talenti} G. Talenti, Elliptic equations and rearrangements, Ann. Sc. Norm. Super. Pisa Cl.
Sci. (4) 3 (1976), no. 4, 697–718. 

\bibitem{V} J.\ Verdera, {\it $L^2$ boundedness of the Cauchy Integral and Menger
curvature,} Contemporary Mathematics {\bf 277} (2001), 139--158.


	\bibitem{Wolibner1933}

W.~Wolibner,
\newblock Un th\'eor\`eme sur l'existence du mouvement plan d'un fluide
parfait, homog\`ene, incompressible, pendant un temps infiniment long,
\newblock \emph{Math. Z.} \textbf{37} (1933), 698--726.
\bibitem{Yudovich1963}
V. I. Yudovich, Non-stationary flow of an ideal incompressible liquid.
USSR Computational Mathematics and Mathematical Physics, \textbf{3} (1963), no. 6, 1407--1456.
	

\end{thebibliography}
\end{document}